\documentclass[12pt]{article}
\usepackage[latin9]{inputenc}
\usepackage{color}
\usepackage{mathrsfs}
\usepackage{amsthm}
\usepackage{amsmath}
\usepackage{amssymb}
\usepackage{graphicx}
\usepackage{setspace}
\usepackage{esint}
\usepackage[authoryear]{natbib}
\usepackage[unicode=true,pdfusetitle,
 bookmarks=true,bookmarksnumbered=false,bookmarksopen=false,
 breaklinks=false,pdfborder={0 0 1},backref=section,colorlinks=true]
 {hyperref}
\hypersetup{
 linkcolor=cyan,citecolor=black}
\usepackage{breakurl}

\makeatletter

  \theoremstyle{definition}
  \newtheorem*{example*}{\protect\examplename}
  \theoremstyle{remark}
  \newtheorem*{rem*}{\protect\remarkname}

\usepackage{mathrsfs}

\usepackage{breakurl}

\usepackage{mathrsfs}\usepackage{amsthm}

  \theoremstyle{definition}
  \newtheorem*{problem*}{\protect\problemname}
  \theoremstyle{remark}
  \newtheorem{rem}{\protect\remarkname}
\theoremstyle{plain}
\newtheorem{thm}{\protect\theoremname}
 \theoremstyle{definition}
 \newtheorem*{defn*}{\protect\definitionname}
  \theoremstyle{definition}
  
  \theoremstyle{plain}
  \newtheorem{cor}{\protect\corollaryname}
 \theoremstyle{definition}
  
  \theoremstyle{plain}
  \newtheorem{lem}{\protect\lemmaname}
  \theoremstyle{remark}
  \newtheorem*{summary*}{\protect\summaryname}

\usepackage{fullpage}\usepackage{titlesec}\titleformat{\section}{\large\sc\center}{\thesection}{1em}{}
\titleformat{\subsection}{\sc\center}{\thesubsection}{1em}{}
\titleformat{\subsubsection}{\sc}{\thesubsubsection}{1em}{}
\usepackage{mathrsfs}%\mathsrc
\usepackage[nolists]{endfloat}%pic at the end

\usepackage{graphics}\usepackage{amsfonts}\usepackage{amstext}\usepackage{tabularx}\usepackage{mathrsfs}\usepackage{subfigure}\usepackage{lscape}\usepackage{bm}

\allowdisplaybreaks[4]

  \providecommand{\conditionname}{Condition}
  \providecommand{\definitionname}{Definition}
  \providecommand{\examplename}{Example}
  \providecommand{\lemmaname}{Lemma}
  \providecommand{\problemname}{Problem}
  \providecommand{\remarkname}{Remark}
  \providecommand{\summaryname}{Summary}
\providecommand{\corollaryname}{Corollary}
\providecommand{\theoremname}{Theorem}

\makeatother

  \providecommand{\examplename}{Example}
  \providecommand{\remarkname}{Remark}

\begin{document}

\title{Minimum Integrated Distance Estimation in Simultaneous Equation Models%
\thanks{The authors would like to express their appreciation to Ivan Fernandez-Val, Sergio Firpo,
Keisuke Hirano, Joel Horowitz, Simon Lee, George Neumann, Hyungsik
Roger Moon, Aureo de Paula, Alexandre Poirier, Andres Santos, Alex Torgovitsky and participants in the seminars at CORE, ECARES, and 2014 LAMES for helpful comments and discussions. Part of this research was undertaken while Zhengyuan Gao was affiliated with University of Iowa and Southwestern University of Finance and Economics whose Financial supports are gratefully acknowledged. All the remaining
errors are ours.%
}}

\author{Zhengyuan Gao%
\thanks{Corresponding author, CORE, Universiti\'{e} catholique de Louvain. E-mail: \texttt{zhengyuan.gao@uclouvain.be} %
} %
\\
\and Antonio F. Galvao%
\thanks{Department of Economics, Tippie College of Business, University of
Iowa. E-mail: \texttt{antonio-galvao@uiowa.edu}%
} }

%\date
\maketitle
\begin{abstract}
\begin{singlespace}
This paper considers estimation and inference in semiparametric econometric
models. Standard procedures estimate the model based on
an independence restriction that induces a minimum distance between
a joint cumulative distribution function and the product of the marginal
cumulative distribution functions. This paper develops a new estimator
which generalizes estimation by allowing endogeneity of the weighting measure
and estimating the optimal measure nonparametrically. The optimality corresponds to the minimum of the integrated distance. 
To accomplish this aim we use Kantorovich's
formulation of the optimal transportation problem. 
The minimizing distance is equivalent to the total variation distance and thus characterizes
finer topological structures of the distributions. The estimation
also provides greater generality by dealing with probability measures
on compact metric spaces without assuming existence of densities.
Asymptotic statistics of the empirical estimates have standard convergent
results and are available for different statistical analyses. In addition,
we provide a tractable implementation for computing the estimator
in practice.\end{singlespace}

\end{abstract}
\thispagestyle{empty}

\begin{singlespace}

Key Words: Minimum distance, Kantorovich's duality, Kernel representation,
Hilbert space\\

JEL Classification: C12, C13, C14
\end{singlespace}

\newpage

\baselineskip19.15pt \pagenumbering{arabic}

\section{Introduction\label{sec:Introduction}}

Economists often use system of simultaneous equations to describe
the relationship among economic variables. In particular, nonlinear simultaneous equation models have provided a valuable method of statistical analysis of policy variables on economic effects. This is especially true for studies where these methods help to analyze the affects of the outcome distributions of interest. Empirically, given data
on the dependent and independent variables in the system, one is usually
interested in estimating functions, distributions, and primitives
describing the system. 
Identification and estimation of nonlinear structural models is often achieved by assuming that
the model's latent variables are independent of the exogenous variables.
Examples of such arguments include, among others, \citet{Manski1983}, \citet{Brown1983},
\citet{Roehrig1988}, \citet{BrownMatzkin1998}, \citet{Matzkin2003},
\citet{BrownWegkamp2002}, \citet{BenkardBerry2006},
\citet{BrownDebWegkamp2008}, and \citet{LintonSperlichvanKeilegom2008}. 

This paper considers estimation and inference in semiparametric econometric
models. We develop a new minimum distance estimator for separable models based on minimizing the distance from the independence condition, where the weighting measure is allowed to be endogenous and is estimated nonparametrically.
This is an important innovation for several reasons. First, the new estimator allows general estimation without external restrictions on the weighting measure and provides an automatic optimal measure selection. Second, the estimator is more efficient than the others available in the literature. Finally, in this general formulation, the proposed method is beneficial to practitioners since the range of models for which the methods are applicable is very broad, for instance, the framework includes an important class of nonlinear simultaneous equation models without requiring knowledge of the reduced form of the model.

We focus on nonlinear simultaneous equations models
with an exogenous observed random vector $X$, a dependent observed
random vector $Y$, and the model errors which are endogenous and
related to the parameterized model via $\varepsilon=\rho(X,Y,\theta_{0})$.
The underlying $\varepsilon$ is assumed to be drawn from a fixed
but unknown distribution and to be stochastically independent of $X$.
However, given the estimated parameters $\theta^{*}\neq\theta_{0}$,
the model may induce $\varepsilon=\rho(X,Y,\theta^{*})$ that are
not necessary independent of $X$. 
The intuition on the importance and development of the new approach is as follows. Under the independence condition, the criterion function is an integral of a
distance function $\mathbf{d}(\cdot,\cdot)$ given by
\begin{equation}
\int\mathbf{d}(H_{\theta}(x,\varepsilon),P_{\theta}(x,\varepsilon))d\mu(x,\varepsilon),\label{eq:empiricalMMD-1}
\end{equation}
where $\theta$ is the parameter of interest, $(H_{\theta}, P_{\theta})$ are distribution functions
of $(X,\varepsilon)$, $\mu$ is a certain probability measure. %If $\mu$ is well defined, i.e. $d\mu(x,\varepsilon)=w_{\mu}dxd\varepsilon$, then $w_{\mu}$ is a weight of $(x,\varepsilon)$. However, the choice of $w_{\mu}$ is ambiguous. For simplicity, $\mu$ is often replaced by an empirical measure so that $w_{\mu}$ is nothing but a Dirac delta function. 
Let $\varepsilon = \rho (X, Y, \theta)$ and $(X,Y)$ be pairs of observations. Thus, the measure $\mu$ in \eqref{eq:empiricalMMD-1} can be expressed as $\mu(X,\rho(X,Y,\theta))$ which, in general, depends on a function of $\theta$. 
We argue it is important to endogenize the measure $\mu$ because changes in $\theta$ induce changes in the other inputs.
For instance, a change in $\theta$ directly induces a change in $\mu(X,\rho(X,Y,\theta))$ as well as in $P_{\theta}(\varepsilon)$ and $H_{\theta}(\varepsilon)$. A second reason to endogenize $\mu$ is practical. The weighting measure $\mu$ in the criterion function characterizes the universal differences between the distributions $H_{\theta}$ and $P_{\theta}$ via an integral. In principle, this measure could be independent of the choice
of $\theta$. However, in practice these weighted measures are unobservable
as the underlying innovation $\varepsilon=\rho(X,Y,\theta_{0})$ is not available. Moreover,
the empirical measure on the product space of $\varepsilon$ and $x$
is sensitive to the choice of $\theta$ in the estimation procedure, hence, in applications, it
is difficult to specify \textit{a priori} measure to deliver the desired
integrated values. 
Therefore, we view the endogeneity of $\mu$ thorough the interaction amongst the weighting measures, $H_{\theta}(\varepsilon)$ and  $P_{\theta}(\varepsilon)$. 

%A new inference approach is proposed to estimate \eqref{eq:empiricalMMD-1} by allowing the endogeneity of the weighting measure $\mu$.  The endogenous measure, in contrast to the standard probability measure, belongs to a generalized measure space that admits all possible outcomes from $H_{\theta}(x,\varepsilon)$ (resp. $P_{\theta}(x,\varepsilon)$) conditioning on any $P_{\theta}(x,\varepsilon)$ (resp. $H_{\theta}(x,\varepsilon)$). Intuitively speaking, the weighting measures associate not only with $(X,\varepsilon)$ from one probability measure but with $(X,\varepsilon)$ from two interacted probability measures. 

To achieve the goal of endogenizing the weighting measure and developing the estimator, we proceed in four steps. First, we are required to select the optimal measure. Hence, we represent $\mu(\cdot)$
as a transport measure, $\gamma(\cdot)$, that pushes $H_{\theta}$ towards $P_{\theta}$. Consider a set of joint measures $\Gamma(H_{\theta},P_{\theta})$. Let $\gamma(x,\varepsilon;x',\varepsilon')$ belong to this set and satisfy
\begin{equation*}
P_{\theta}(x,\varepsilon) = \int d\gamma (x,\varepsilon;x',\varepsilon') dH_{\theta}(x',\varepsilon'), \ \text{and,} \
H_{\theta}(x',\varepsilon') = \int d\gamma (x,\varepsilon;x',\varepsilon') dP_{\theta}(x,\varepsilon).
\end{equation*}
The measure  $\gamma(x,\varepsilon;x',\varepsilon')$ replaces the role of $\mu$ in the problem \eqref{eq:empiricalMMD-1} and accounts for interaction with $P_{\theta}(\varepsilon)$ and  $H_{\theta}(\varepsilon)$.
For two pairs of samples $(X,Y)$, one is drawn from $P_{\theta}(x,\varepsilon)$ and the other is drawn from $H_{\theta}(x,\varepsilon)$,  joint measure $\gamma(x,\varepsilon;x',\varepsilon')$  is the measure of a transport map connecting these two samples.  
Second, given the transport measure, we reformulate the main problem using a generalized estimation criterion function, $W(\theta)$, based on an integrated distance w.r.t. the joint probability measure $\gamma$ as
\begin{equation*}
W(\theta):=\inf_{\gamma\in\Gamma(H_{\theta},P_{\theta})}\left\{ \int\mathbf{d}(q,q')d\gamma(q,q'):\: q\sim H_{\theta}(x,\varepsilon)\mbox{ and }q'\sim P_{\theta}(x,\varepsilon)\right\}. \label{eq:Kantonovich}
\end{equation*}
Third, to solve this problem practically, the above weighting measure is dually represented by
a transition kernel that pushes one measure to another under the new
metric. Fourth, from the dual representation, we establish a kernel representation theorem, such that we are able to embed the infinite dimensional criterion function into a tractable space. In doing this, the criterion function is represented in terms of a kernel-based distance.
When the transition kernel induces the zero criterion under
the new metric, the parameters are identified at $\theta_{0}$, and a practical estimator for the parameters of interest is developed.

The contributions of this paper to the literature are as
following. First, we propose a novel weighted minimum distance estimator
from independence condition which employs a nonparametric estimation
of the distance function, and provides an automatic optimal measure selection.
Second, we establish identification of the
model in the dual representation. The dual distance characterizes finer topologies of probability measures.
Third, based on the new representation, we provide
both asymptotically biased (in the sense of root-$n$ inconsistent) and unbiased
statistics for estimation. In addition, we establish consistency and
derive their limiting distributions. Fourth, we develop practical
inference procedures. Fifth, we show that the proposed method is more efficient relative to the existing estimators.
Finally, we provide tractable implementation
for computing the estimator in practice.

From a technical point of view, the construction of the estimation procedure, and derivations of the statistical
and limiting properties of the proposed estimator are of independent
interest. 
The central mathematical tool used to accomplish these
is the use of the optimal transportation theory and its corresponding
dual problem. In particular, we make use of the Kantorovich's formulation
of the optimal transportation problem. The novel integrated distance
function can be interpreted as the optimal cost of transferring one
mass distributed according to joint cumulative distribution function (c.d.f.) of $\varepsilon$ and $x$
to another mass distributed according to the product of $\varepsilon$
and $x$'s marginal c.d.f.s. The use of optimal weight is essential
to the general formulation of integral minimum distance problems.
Therefore, the main technical contribution of this paper is to introduce a new
set of theorems for establishing the asymptotic results (consistency
and weak convergence) for the weighted minimum distance from
independence estimators for general distance functions in both primal
and dual problems. The problem of approximating measures with respect to transportation
distances has connections with the fields of probability theory \citep{Talagrand1991},
information theory \citep{GrafLuschgy2000}, and optimal transport \citep{CedricVillani2009}.
The transportation distance also appears in economics and econometrics. Current applications include rearrangements, identification, matching and quantile regression
see, e.g. \citet{GalichonFernandezValChernozhukov2010}, \citet{EkelandGalichonHenry},
\citet{ChiapporiMcCannNesheim} and \citet{GalichonSalanie2010}.
Techniques from the transportation distance are often used for finding
more tractable dual problems.

Estimation of econometric
models based on the independence between the exogenous variables and the unobserved
disturbance term has been explored in the literature. 
In a seminal work,  \cite{Manski1983} developed an approach of minimizing the distance from the independence condition for estimation.
The procedure to estimate the parameters $\theta_{0}$ uses a criterion function that
compares the mean-square distance between a joint c.d.f. and the product of its marginal c.d.f.s. 
In order to derive the asymptotic properties of Manski's estimatior, \citet{BrownWegkamp2002}
extend the criterion function to an integrated mean-square
distance criterion function. Importantly, the integral is taken with
respect to a weighted measure on the product space of $\varepsilon$
and $x$. Their approach provides a foundation for estimating $\theta_{0}$
from general nonlinear simultaneous equation models.\footnote{As discussed in \cite{BenkardBerry2006}, the identification condition that  $\varepsilon$ is independent of $X$ generally does not hold for $\theta\neq \theta_{0}$. For an integrated criterion, such a dependence between $X$ and $\varepsilon$ under different $\theta$ reflects thorough the weighted measure. } However, it has been left
unspecified the essential argument on the practical weighting measure
of the product space of $\varepsilon$ and $x$.   
The estimator studied in this paper is also related to other alternative methods. For instance, \cite{KomunjerSantos2010} develop a semiparametric estimator for invertible nonseparable models with scalar latent variables and an infinite dimensional component. \cite{Santos2011} proposes an M-estimator under the assumption that the model is strictly monotonic in the scalar error term and derives its corresponding asymptotic properties.

The remaining of the paper is organized as follows. In Section \ref{sec:Model-Framework}
we present the basic formulation of the simultaneous equation models.
Section \ref{sub:Estimator} describes the criterion function including its primal
and dual. In addition, it presents a method for representing the infinite dimensional criterion
function. Section \ref{sec:Inference} discusses identification, and provides both asymptotically biased and unbiased
estimators. Large sample statistical theorems are also
included. Practical computation is presented in Section \ref{sec:Computation}.
Section \ref{sec:Related-Inferential-Methods} discusses
several related estimates of the optimizing distance or divergence
of probability measures. Conclusions appear in Section \ref{sec:Conclusion}.

\textbf{Notations}: Throughout the paper, we use capital alphabet
to denote random variable, i.e. $X$; use bold letter alphabet to
denote a bundle of realizations, i.e. $\mathbf{x}$; and use the alphabet
to denote the deterministic value, i.e. $x$; $i$th-realization is
denoted with a subscript i.e. $x_{i}$.

In this paper, $\mathcal{H}$ always denotes a Hilbert space, i.e.
a complete, norm vector space endowed with an inner product $\left\langle \cdot,\cdot\right\rangle $
giving rise to its norm via $\|x\|=\sqrt{\left\langle x,x\right\rangle }$.
Let $B_{\mathcal{H}}:=\{x\in\mathcal{H}:\|x\|\leq1\}$ be the closed
unit ball of $\mathcal{H}$. 
%and $S_{\mathcal{H}}:=\{x\in\mathcal{H}:\|x\|=1\}$ be its sphere. 
Every separable Hilbert space is isometrically isomorphic
to the space of all square-summable sequences ($\ell_{2}$ space).
We denote $\|\cdot\|_{\infty}$ as the supremum norm such that $\|f\|_{\infty}:=\sup_{x\in\mathcal{H}}|f(x)|.$

\section{Model Framework\label{sec:Model-Framework}}

\subsection{Model Setup and Assumptions\label{sub:Model-Setup}}

We begin by describing the model framework, notation, and main assumptions for the subsequent
developments. The setup is similar to \citet[hereafter BW]{BrownWegkamp2002}, but with important differences. We require greater generality,
since we deal with distributions on compact metric spaces, and densities
are not assumed to exist. To account for the general cases, we do not require existence of the reduce form.

Consider the following model 
\begin{equation}
\varepsilon=\rho(X,Y,\theta),\label{eq1}
\end{equation}
where $\rho(\cdot)$ is the function describing the structural model,
$Y$ is the dependent observed random vector, $X$ is an exogenous
observed random vector, $\varepsilon$ is the model error which
is endogenous but is latent exogenous under the truth, and $\theta$ is a vector of unknown parameters.

First we define the distance used in the simultaneous equation model
analyzed in this paper. We define the minimum distance from independence
conditions, where the function $\mathbf{d}(\cdot,\cdot)$ is a metric
on the space of joint cumulative distribution functions (c.d.f.'s)
of $(X,\varepsilon)$, where $\varepsilon$ takes values in $\mathbb{R}^{K}$.
Let $H_{\theta}(x,\varepsilon)$ be the joint c.d.f of $(X,\varepsilon)$,
and $F(x)$ and $G_{\theta}(\varepsilon)$ be the respective associated
marginal c.d.f.'s. %Then, $d(H(x,\varepsilon),F(x)G(\varepsilon)) = 0$ iff $X$ and $\varepsilon$ are stochastically independent. In Manski (1983) $d$ is the mean-square distance, in Brown and Matzkin (1998) $d$ is the metric on the space of c.d.f.'s induced by the Prohorov metric on the space of measures, and in this paper $d$ is the weighted mean-square distance. Our discussion of implicit nonlinear simultaneous equations models follows the expositions of Brown (1983), Roehrig (1988), and Brown-Matzkin (1998). %A structure $S$ is an ordered pair ??????X??Y????????H??x????????.The structural equations are defined as ?? = ????X?? Y ?? ????. Our model consists of all structures S that satisfy the following assumptions:
As a consequence of the identification assumption, which will be discussed
below in detail, for any metric function $\mathbf{d}(\cdot,\cdot)$
on the space of measures, 
\[
\mathbf{d}(H_{\theta}(x,\varepsilon),F(x)G_{\theta}(\varepsilon))=0
\]
if and only if $X$ and $\varepsilon$ are stochastically independent.
Using this condition, we study extremum estimators that minimize
the above distance.

We are interested in estimating the parameters $\theta$ in equation \eqref{eq1}. Consider
the following assumptions. 
\begin{itemize}
\item[C1.] Parameter: The true $\theta_{0}$ belongs to a parameter space $\Theta$. 
\item[C2.] Observations (independent): $\mathbf{x}:=(x_{1},\dots,x_{n})\in\mathbb{R}^{L\times n}$
are $n$ draws of random vector $X$. Any $X\in\mathcal{X}$. Let
$(\Omega,\mathcal{F}_{X},F)$ be a probability space, where $\Omega$
is a sample space, $\mathcal{F}_{X}$ is a $\sigma$-field, and $F$
is a sample probability measure on $(\Omega,\mathcal{F}_{X})$. Then
$\mathbf{x}\in\{X(\omega):\Omega\rightarrow\mathcal{X}\}$. Let $\mathcal{X}$
be a compact metric space. 
\item[C3.] Observations (dependent) or model realizations: $\mathbf{y}:=(y_{1},\dots,y_{n})\in\mathbb{R}^{K\times n}$
are $n$ draws of random vector $Y$. Any $Y\in\mathcal{Y}$. $Y:=\rho^{-1}(X,\varepsilon,\theta_{0}):\Upsilon\rightarrow\mathcal{Y}$
where $\Upsilon:=\mathcal{X}\times\Sigma\times\Theta$. Let $\mathcal{Y}$
be a compact metric space. 
\item[C4.] Observations (both independent and dependent): $\mathbf{z}:=(\mathbf{x},\mathbf{y})\in\mathbb{R}^{(L+K)\times n}$
are $n$ draws of joint observations $Z:=(X,Y)$. Let $\mathcal{Z}:=\mathcal{X}\times\mathcal{Y}$. 
\item[C5.] Unobservable variable: $\varepsilon=\rho(X,Y,\theta)$. For all
$\theta\in\Theta$, $\rho(\cdot,\theta)$ is a mapping from $\mathcal{Z}$
into $\mathbb{R}^{K}$. Let $(\mathcal{Z}\times\Theta,\mathcal{B}(\mathcal{Z}\times\Theta),G_{\theta})$
be a probability space, where $\mathcal{B}(\mathcal{Z}\times\Theta)$
is the Borel $\sigma$-algebra of the Cartesian product $\mathcal{Z}\times\Theta$.
$G_{\theta}(\varepsilon)$ is a probability measure defined on $(\mathcal{Z}\times\Theta,\mathcal{B}(\mathcal{Z}\times\Theta))$. 
\item[C6.] Identification condition: $\varepsilon=\rho(X,Y,\theta)$ is independent
of $X$ if and only if $\theta=\theta_{0}$. Let $H_{\theta}(x,\varepsilon)$
be the joint probability measure of $(X,\varepsilon)$. The independent
assumption between $X$ and $\rho(X,Y,\theta_{0})$ is equivalent
with 
\[
H_{\theta}(x,\varepsilon)=F(x)G_{\theta}(\varepsilon)\quad\mbox{iff }\quad\theta=\theta_{0}
\]
for any $(x,\varepsilon)\in\mathcal{X}\times\Sigma$. 
\item[C7.] Structure: $\mathbb{S}$ is an ordered pair $\left(\rho(X,Y,\theta),H_{\theta_{0}}(x,\varepsilon)\right)$.
The observations $\mathbf{z}=(\mathbf{x},\mathbf{y})$ are generated
by the structure $ $$\mathbb{S}_{0}:=\left(\rho(X,Y,\theta_{0}),H_{\theta_{0}}(x,\varepsilon)\right)$. 
\item[C8.] Empirical probability measures: $F_{n}(x)$, $G_{n\theta}(\varepsilon)$
and $H_{n\theta}(x,\varepsilon)$ are the empirical measures associated
with $F(x)$, $G_{\theta}(\varepsilon)$ and $H_{\theta}(x,\varepsilon)$
respectively based on the observed data $\mathbf{z}$. For example,
$F_{n}(x)$ is a Borel probability measure defined on $\mathcal{X}$
as a discrete measure with a number $n$ of points, such that 
\begin{equation}
G_{n\theta}(\varepsilon):=\frac{1}{n}\sum_{i=1}^{n}\mathbf{1}\{\rho(z_{i},\theta)\leq\varepsilon\},\quad dG_{n\theta}(\varepsilon):=\frac{1}{n}\sum_{i=1}^{n}\delta_{\rho(z_{i},\theta)}(\varepsilon)\label{eq:EmpiricalCDFs}
\end{equation}
where $\mathbf{1}\{\cdot\}$ is the indicator function and $\delta_{\rho(z_{i},\theta)}(\cdot)$
is the Dirac function at position $\rho(z_{i},\theta)$. The notation
$d(\cdot)$ in this paper always denotes the differential notation. 
\item[C9.] No reduced form: $\rho(X,\varepsilon,\theta)$ is non-invertible
for $\theta\neq\theta_{0}$. It means that there is no way of estimating
$\theta$ via $Y=\rho^{-1}(X,\varepsilon,\theta)$ or $\varepsilon=\rho(X,\rho^{-1}(X,\varepsilon,\theta),\theta)$.
But at $\theta_{0}$, $Y=f(X,\varepsilon)$ where $f\in\mathcal{F}$
is an unknown function and $\mathcal{F}$ can be embedded in a Hilbert
space. \end{itemize}
Conditions C1-C4 are standard in the econometric literature. In BW,
compactness of the parameter space $\Theta$ is assumed while we stay
with a general parameter space. But as a compensation, we impose the
compactness for $\mathcal{Y}$, the sample space of dependent observations,
in C3. The motivation of compactness in C2 and C3 is to ensure closeness
and boundeness for both observable variables sample spaces $\mathcal{Z}$
in C4.\footnote{If $\mathcal{Z}$ is non-compact, then a kernel
function with non-compact support may not be able to identify the
true parameter $\theta_{0}$.} C5 concerns the measurability of $\mathcal{Z}$. The measurability of $Z$ induces the measurability of $\varepsilon=\rho(z,\theta)$. Therefore, the probability
measure of $\varepsilon$ always exists. Conditions C6 and C7 impose identification and are the
same as the identification conditions in \citet{Manski1983} and BW.
C6 is the independent condition for the population probability $H_{\theta}(x,\varepsilon)$.
The joint probability $H_{\theta}(x,\varepsilon)$ equals the product
of marginal c.d.f.s $F(x)$ and $G_{\theta}(\varepsilon)$ if and
only if $\theta=\theta_{0}$. This is the key device to uniquely identify
$\theta_{0}$. This paper focuses on estimation. But we highlight that, differently from BW, we establish identification in the dual problem. There is an extensive literature discussing identification in simultaneous equation models as that in our primal problem. 
We refer the reader to \citet{BenkardBerry2006} and the literature therein.
C7 states that the pair observation $(X,Y)$ generated by $\rho(X,Y,\theta_{0})$
will give different simultaneous equation system $\rho(X,Y,\theta)$
under different value of $\theta$. C8 defines the empirical c.d.f.
and considers the weighted Dirac measure as the Randon-Nikodym derivative
of the empirical c.d.f.. Similar forms of C7 and C8 can be found in Section
2 and Section 3 respectively in BW.
Assumption C9 is novel and 
states that the simultaneous structure
is not invertible. This is an important assumption allowing flexibility and generality of the proposed methods because
the estimator does not require knowledge on the reduced form.
Thus, to attain an estimator of $\theta_{0}$ by simply
an inversion is not of our concern. However, with observations of
$Y$, one can generate a nonparametric approximation for the underlying
true function $f$. In order to ensure the approximation
is feasible, the unknown true function $f$ should come from an approximable
functional class $\mathcal{F}$. We restrict this class to reproducing
kernel Hilbert space (RKHS). More discussion about RKHS will be given
in Section \ref{sub:Mathematical-Preliminary} of the Appendix.

\subsection{An Example}

We now illustrate the model in light of the imposed conditions, especially C9.
We start with a standard example satisfying our model specification \eqref{eq1}.  Consider a general separable supply and demand model
\begin{align}
Q &= D(Z,P;\theta_{D}) + \varepsilon_{D} \notag \\
P&=S(W,Q;\theta_{S}) + \varepsilon_{S}, \notag
\end{align}
where $Q$ is quantity, $P$ is price, $(Z,W)$ are characteristics, $(\varepsilon_{D},\varepsilon_{S})$ are shocks, $(\theta_{D},\theta_{S})$ the parameters of interest, and the functions $D(\cdot)$ and $S(\cdot)$ are allowed to be nonlinear functions. In the representation (\ref{eq1}) we have that $Y=(Q,P)$, $X=(Z,W)$, $\varepsilon=(\varepsilon_{D},\varepsilon_{S})$, $\theta=(\theta_{D},\theta_{S})$, and the structural model $\rho=(Q - D(\cdot), P - S(\cdot))$.  

If a simultaneous equation model satisfies C1-C8 and has the specification in \eqref{eq1}, one can establish the integrated distance function 
\begin{equation}
\int\mathbf{d}(H_{\theta}(x,\varepsilon),F(x)G_{\theta}(\varepsilon))d\mu(x,\varepsilon),\label{eq:integratedDF}
\end{equation}
which is continuous on $\Theta$ and $\theta_{0}$ is the unique global
minimum of this integrated distance function for a bounded measure $\mu(x,\varepsilon)$ of $(x,\varepsilon)$ and a compact parameter space $\Theta$. This result is a summary of Theorems 2 and 3 of BW. 

To consider an example satisfying C9, as in \citet{BenkardBerry2006}, define a non-separable supply and demand model,
\begin{equation*}
Q = D(Z,P, \varepsilon_{D} ;\theta_{D}),\,\, P =S(W,Q, \varepsilon_{S};\theta_{S}),
\end{equation*}
with a triangular structure for the joint distribution of $P$ and $Q$ such that
\begin{equation*}
\tilde{U}_{1}=\Psi(Q|Z,W),\,\,\tilde{U}_{2}=\Psi(P|Z,W,Q),
\end{equation*}
where $\Psi(\cdot)$ is the (known) joint distribution of the endgenous variables and exogenous shifters, and 
$\tilde{U}$'s are constructed such that they are independent of one another as well as independent of $(Z,W)$. The triangular construction implies that the reduced form takes the form $Q=f(Z,W,\tilde{U}_{1})$,
whereas in the general model the reduce form takes the form
\begin{equation*}
Q=f(Z,W,\varepsilon_{D},\varepsilon_{S}).
\end{equation*}
Thus, the triangular system cannot retrieve the true reduce form which satisfies C9. 

It is important to highlight the role of condition C9 and its link with the literature. This assumption specifies the underlying relation between $Y$ and $X$ when $\theta=\theta_{0}$, which can be nonparametrically recovered. Assumption C9 excludes the possibility of reduced form relation. Models  allowing for a reduced form have been addressed in \citet{BenkardBerry2006}. Their concern is whether a derivative condition on $\rho(\cdot)$ is sufficient to derive a reduced form, and consequently identification. Differently from BW and \citet{BenkardBerry2006}, the setup in this paper neither assume a derivative condition as in Lemma 3.1 of the later, nor assume the existence of a reduced form.

A generalization of \eqref{eq:integratedDF}
is the main concern in this paper. To estimate $\theta_{0}$ via \eqref{eq:integratedDF},
BW consider a specified bounded measure $\mu$ and a compact parameter
space. Importantly, both of these two conditions will be relaxed later
in this paper. We will estimate an optimal measure $\gamma$ (it differs
from the original notation $\mu$) and consider a general topological
vector space $\Theta$. To achieve this we first propose a new estimator.
We show that the new estimator generalizes \eqref{eq:integratedDF} and induces a dual representation.
Given this dual representation setup, we show that the problem is
still identifiable. The key device of our identification procedure
is to select a proper integrated distance function so that $H_{\theta}(x,\varepsilon)$
and $P_{\theta}(x,\varepsilon)$ are from a complete separable metric
space. If such an integrated distance similar to \eqref{eq:integratedDF}
can induce a complete separable metric space for $H_{\theta}(x,\varepsilon)$
and $P_{\theta}(x,\varepsilon)$ on $\mathcal{Q}$, then any $\theta$
generating the associated distance value of \eqref{eq:integratedDF} will be separated
from the others. Thus $\theta_{0}$, which gives a unique value of
this integrated distance, is well separated from the other $\theta$'s.%
\footnote{\textcolor{red}{\label{fn:completeSep}}This idea is based on Ascoli's
theorem, see i.e. \citet[Exercise 2.10][]{Shorack2000}. Ascoli's
theorem states that for a class of equicontinuous functions mapping
from a complete separable metric space $\mathscr{P}$ to another metric
space $\Theta$, any sequence of such functions with a compact
support will uniformly converge. The mapping in our setup is an estimator.
The complete separable metric space is for $H_{\theta}(x,\varepsilon)$
and $P_{\theta}(x,\varepsilon)$. Equicontinuouity implies that if
any $\theta$ such that $d(\theta,\theta_{0})\geq\delta$, then the
distance between $H_{\theta}$ and $H_{\theta_{0}}$ (or $P_{\theta}$
and $P_{\theta_{0}}$) is larger than some $\epsilon$. This is exactly
the identification condition.}

\section{Criterion Function \label{sub:Estimator}}
This section establishes the grounds for constructing the minimum integrated distance estimator in the next section.
First, we define the primal objective function from the independence condition. Second, we reformulate the problem with a new criterion function based on the  Monge-Wasserstein distance. Third, we establish the validity of the dual problem. Finally, we establish a kernel representation theorem. 

\subsection{Primal Objective Function}
With a bounded measure $\mu$ on $\mathcal{Z}$, under the independence assumption the
criterion function for estimation of $\theta_{0}$ is given by
\begin{equation}
M(\theta)=\int\mathbf{d}(H_{\theta}(x,\varepsilon),F(x)G_{\theta}(\varepsilon))d\mu(x,\varepsilon).\label{eq:OriginalProb}
\end{equation}
This distance function is based on the mean-square distance. The criterion
function is minimized at $\theta_{0}$. In the literature, statistical estimation has been
based on the empirical counterpart of $M(\theta)$, 
\begin{equation}
\widehat{M}(\theta)=\int\mathbf{d}(H_{n\theta}(x,\varepsilon),F_{n}(x)G_{n\theta}(\varepsilon))d\mu(x,\varepsilon),\label{eq:OriginalProbEmpirical}
\end{equation}
where $H_{n\theta}(x,\varepsilon)$, $F_{n}(x)$, and $G_{n\theta}(\varepsilon)$ are as defined in \eqref{eq:EmpiricalCDFs}.

An open question for this estimation procedure is how to construct
an estimator for the bounded measure $\mu$. Actually, in practice
the selection of $\mu$ will affect the estimation scheme substantially.
If $\mu$ is attached to a specific measure form, then the natural
question is what form of $\mu$ would be in practice. One possibility
for practical implementation would be to use a grid search on $\mu$
as in \citet{BrownWegkamp01}. However, this procedure
implies the use of a uniform empirical distribution for $\mu$, which
might not be the correct underlying distribution. Below, we show numerically
that an optimal choice of $\mu$ is quite different from the uniform
measure.

In \eqref{eq:OriginalProb}, BW require the integral measure $\mu(x,\varepsilon)$ to be independent
of $\theta$. Without this condition, there may exist some $\theta\neq\theta_{0}$
such that 
\[
n^{-1}\sum_{i=1}^{n}\mathbf{d}(H_{n\theta}(x_{i},\varepsilon_{i}),F_{n}(x_{i})G_{n\theta}(\varepsilon_{i}))=0.
\]
A simple choice of this empirical measure is to set $\mu(x,\varepsilon)=H_{n\theta}(x,\varepsilon)$.
That means the empirical summation is taken over the data $(x_{i},\varepsilon_{i})$
generated by the model $\rho(X,\varepsilon,\theta)$ at $\theta$.
The summation implies that the value of $\theta$ matters for evaluating
the integral. By endogenizing the role of $\theta$ in the integrated
minimum distance function, we can relax the requirement of independence
of $\mu$ in \eqref{eq:OriginalProb}.

Therefore, the first goal of this paper is to propose a new estimator
in which we endogenize the measure $\mu(x,\varepsilon)$ and then
estimate the new endogenized measure and $\theta_{0}$ simultaneously while making minimal assumptions
regarding the measure $\mu(x,\varepsilon)$. This is an important
innovation to practitioners since the estimation of $\mu(x,\varepsilon)$
affects the estimation of the parameters of interest. We will illustrate
that in the numerical simulations below. In addition, we provide a
general result for estimating these measures. Formally, we formulate
the following inferential procedure: 
\begin{description}
\item [{(M)}] Let the joint probability measure $H_{\theta}(x,\varepsilon)$
and the mixing probability measure $P_{\theta}(x,\varepsilon)=F(x)G_{\theta}(\varepsilon)$
be Borel probability measures defined on a domain $\mathcal{Q}=\mathcal{X}\times\Sigma$.
Given observations $ $$\mathbf{x}\in\mathcal{X}$ and generating
realizations $\varepsilon=(\rho(x_{1},y_{1},\theta),\dots,\rho(x_{n},y_{n},\theta))$,
we consider that $(\mathbf{x},\varepsilon)\sim H_{\theta}(x,\varepsilon)$
and $(\mathbf{x}',\varepsilon')\sim P_{\theta}(x',\varepsilon')$ are
drawn independently and identically distributed (i.i.d) from two unknown
probability measures on $\mathcal{Q}$. We represent $\mu(x,\varepsilon)$
as a transport measure $\gamma(x,\varepsilon;x',\varepsilon')$ that pushes $H_{\theta}(x',\varepsilon')$ towards
$P_{\theta}(x,\varepsilon)$. We show that there is an optimal $\gamma$
and represent this optimal measure using kernels. By this optimal
measure, we introduce a new criterion (Wasserstein distance function)
and develop the associated estimator and testing procedure.
\end{description}

\subsection{Transport Measure Representation}
In this section, we reformulate the primal problem described above and represent $\mu(x,\varepsilon)$ as a transport measure.
We start by reformulating the problem \eqref{eq:OriginalProb} with
a new criterion $W(\theta)$. The criterion is based on the Wasserstein
or Monge-Wasserstein distance. Following the definition in \citet[p.420,][]{Dudley2002},
we introduce the criterion: 
\begin{equation}
W(\theta):=\inf_{\gamma\in\Gamma(H_{\theta},P_{\theta})}\left\{ \int\mathbf{d}(q,q')d\gamma(q,q'):\: q\sim H_{\theta}(x,\varepsilon)\mbox{ and }q'\sim P_{\theta}(x,\varepsilon)\right\} ,\label{eq:Kantonovich}
\end{equation}
where $\gamma(\cdot)$ is the transport measure, $\Gamma(H_{\theta},P_{\theta})$ is the set of jointly distributions
on $\mathcal{Q}\times\mathcal{Q}$ with the marginals of $\gamma$ given by $H_{\theta}(x,\varepsilon)$
and $P_{\theta}(x,\varepsilon)$. Note that $H_{\theta}(x,\varepsilon)$ as a marginal of $\gamma$ is still jointly distributed for $(x,\varepsilon)$. The set $\Gamma$ refers to a product of probability measures. The function $\mathbf{d}(\cdot,\cdot)$ is a
metric function for samples on $\mathcal{Q}$. The differences between
the original problem \eqref{eq:OriginalProb} and the problem  \eqref{eq:Kantonovich} are
threefold.
First, the criterion function in \eqref{eq:Kantonovich} distinguishes
the joint observations $(q_{1},q_{2},\dots)=(\{x_{1},\varepsilon_{1}\},\{x_{2},\varepsilon_{2}\},\dots)$
and the marginal observation $(q_{1}^{'},q_{2}^{'},\dots)=(\{x_{1},\varepsilon_{1}^{'}\},\{x_{2},\varepsilon_{2}^{'}\},\dots)$.
The observations $\varepsilon$ in $Q:=(X,\varepsilon)$ are drawn
from $\varepsilon=\rho(X,Y,\theta)$ which are jointly distributed
with $X$. While the \emph{generic} observations $\varepsilon'$ in
$Q':=(X,\varepsilon')$ are drawn from a marginal distribution $G_{\theta}(\varepsilon)$
independently of $X$. Empirically, the number of observations in
these two cases are also different. We denote $\mathbf{q}$ as $(q_{1},\dots,q_{n})$
and denote $\mathbf{q}'$ as $(q_{1}^{'},\dots,q_{m}^{'})$.%
\footnote{More generally, we can think $\mathbf{q}'=(q_{1}^{'},\dots,q_{m}^{'})$
are from a resampling scheme of $\mathbf{q}$. Then even if $m=n$,
one should realize that $\mathbf{q}'\neq\mathbf{q}$. %
} Second, the measure $\gamma$ in \eqref{eq:Kantonovich} is defined
as a product measure of $H_{\theta}$ and $P_{\theta}$. Unlike the
unspecified $\mu(x,\varepsilon)$ in the original problem \eqref{eq:OriginalProb},
$\gamma$ is restricted to the set $\Gamma(H_{\theta},P_{\theta})$
such that any element in this set satisfies: 
\begin{equation} \label{eq:SampleMarginals}
H_{n\theta}(x_{i},\varepsilon_{i})=\frac{1}{m}\sum_{j=1}^{m}\gamma((x_{i},\varepsilon_{i}),(x_{j},\varepsilon_{j}^{'})),\quad P_{n\theta}(x_{j},\varepsilon_{j}^{'})=\frac{1}{n}\sum_{i=1}^{n}\gamma((x_{i},\varepsilon_{i}),(x_{j},\varepsilon_{j}^{'})).
\end{equation}
Third, the distance function in \eqref{eq:Kantonovich}, $\mathbf{d}(\cdot,\cdot)$,
is for samples $(x_{i},\varepsilon_{i})$ and $(x_{j},\varepsilon_{j}^{'})$
not for their c.d.f.s. 

Similarly to \eqref{eq:OriginalProbEmpirical}, the new minimum distance
estimator will be based on the empirical Wasserstein distance defined as
follows: 
\begin{align}
\widehat{W}(\theta):= & \inf_{\gamma\in\Gamma(H_{n\theta},P_{m\theta})}\left\{ \int\mathbf{d}(q,q')d\gamma(q,q'):\: q\sim H_{n\theta}(x,\varepsilon)\mbox{ and }q'\sim P_{m\theta}(x,\varepsilon)\right\}. \label{eq:Wdist}\\
= & \inf_{\gamma(q_{i},q_{j})\in\Gamma(H_{n\theta},P_{m\theta})}\left\{ \frac{1}{nm}\sum_{j=1}^{m}\sum_{i=1}^{n}\mathbf{d}(q_{i},q_{j}^{'}):\: q_{i}\sim H_{n\theta}(x,\varepsilon)\mbox{ and }q_{j}^{'}\sim P_{m\theta}(x,\varepsilon)\right\} \notag.
\end{align}
The subscript $m$ indicates that $q^{'}$ are drawn from $m$ samples.
One can consider $m=n$ as a special case. The above criterion function \eqref{eq:Wdist} and the corresponding marginals in \eqref{eq:SampleMarginals} contain $\gamma(\cdot)$. Thus, without further restrictions, $\gamma(\cdot)$ is an infeasible element with infinite dimensions. Next we provide a tractable representation of $\widehat{W}(\theta)$ in order to achieve feasible estimation. The infinite dimension issue will be solved by considering the dual representation of $\widehat{W}(\theta)$.

\begin{rem} The interpretation of \eqref{eq:Kantonovich} is the
cost in terms of a distance function $\mathbf{d}(\cdot,\cdot)$ of
transferring a mass $Q:=(X,\varepsilon)$ distributed according to
$H_{\theta}(x,\varepsilon)$ to a mass $Q':=(X',\varepsilon')$ distributed
according to $P_{\theta}(x,\varepsilon)$. The measure $\gamma\in\Gamma$
is the transportation schedule between $Q$ and $Q'$.\end{rem}

\begin{rem}The criterion $W(\cdot)$ comes from a general class called
$k$-Wasserstein distance $W^{k}(H,P):=\left\{ \inf_{\gamma\in\Gamma(H,P)}\int\mathbf{d}(q,q')^{k}d\gamma(q,q')\right\} ^{\frac{1}{k}}$.
Let $\mathscr{P}^{k}(\mathcal{Q})$ denote the space of Borel probability
measures on $\mathcal{Q}$ with the first $k$-th moments. For any
sample space with finite moment, the space of measures endowed with
$W^{k}$ metric is a complete separable metric space \citep[p.94,][]{CedricVillani2009}.
The most useful versions of $k$-Wasserstein distance are $k=1,2$.
In our context, $W(\theta)$ is the $1$-Wasserstein distance. Note
that both measures $H_{\theta}(x,\varepsilon)$ and $P_{\theta}(x,\varepsilon)$
belong to $\mathscr{P}^{1}(\mathcal{Q})$. Thus, $W(\theta)$ induces
a complete separable metric space for $H_{\theta}(x,\varepsilon)$
and $P_{\theta}(x,\varepsilon)$ on $\mathcal{Q}$. This complete
separable metric space with compact supports imply the identification
of $\theta_{0}$ in our dual problem.%
 Therefore, except the case that $H_{\theta_{0}}(x,\varepsilon)=P_{\theta_{0}}(x,\varepsilon)$,
any $H_{\theta}(x,\varepsilon)$ will differ from $P_{\theta}(x,\varepsilon)$
in $W(\theta)$-distance. Thus any $W(\theta)$ will be separated
from zero if $\theta\neq\theta_{0}$. In other words, $\theta_{0}$
is well separated from the other $\theta$'s.\textcolor{red}{{} }\end{rem}

\begin{rem} There are many possible choices of distances between
probability measures, such as the Levy-Prokhorov \citep{BrownMatzkin1998},
or the weak-{*} distance. Among them, $W^{k}$ metrizes weak convergence
\citep[Theorem 6.9,][]{CedricVillani2009}, that is, a sequence $\{P_{n}\}_{n\in\mathbb{N}}$
of measures converges weakly to $P$ if and only if $W^{k}(P_{n},P)\rightarrow0$.
However, as pointed out in \citet[p.98,][]{CedricVillani2009} ``Wasserstein
distances are rather strong,... a definite advantage over the weak-{*}distance\textquotedblright{}.
Also, it is not so difficult to combine information on convergence
in Wasserstein distance with some smoothness bound, in order to get
convergence in stronger distances. \end{rem}

\begin{rem} Note that, for clarity, we adhere to the strict association
of a random variable with a distribution. Thus $H_{n\theta}(x,\varepsilon)$
and $P_{m\theta}(x,\varepsilon)$ are associated with distinct random
variables ($q$ and $q'$, respectively) representing distinct states
of knowledge about the same set of variables. This view is slightly
different than the usual notion of a single random variable for which
we ``update\textquotedblright{} our belief. \end{rem}

\begin{rem} When the metric $\mathbf{d}(q,q')$ for distribution
laws on $\mathbb{R}$ is the absolute function $|q-q'|$, $W(\cdot)$
is the Gini index, see e.g. \citet[p.435,][]{Dudley2002}.
\end{rem}

\subsection{The Dual Problem \label{sub:The-Dual}}

While the independence condition is for c.d.f.s, the distance function
in the optimization described in \eqref{eq:Kantonovich} is for samples. Thus, by transferring the focus
from samples to c.d.f.s, we need a dual problem of \eqref{eq:Kantonovich}. We note that, in this paper, both the
identification and estimation will be conducted in this dual problem.

The dual representation of $W(\theta)$ in \eqref{eq:Kantonovich}
was introduced by Kantorovitch in order to solve a convex linear
program. For the empirical distance $\widehat{W}(\theta)$ in \eqref{eq:Wdist}, the dual corresponds
to the convex relaxation of a combinatorial problem when the densities
are sums of the same number of Diracs. This relaxation extends the
notion of $\Gamma(H_{n\theta},P_{m\theta})$ to arbitrary sum of weighted
Diracs, see for instance \citet{Villani2003}. The following is called
Kantorovich-Rubinstein theorem. It introduces the Kantorovich's duality
principle to $W(\theta)$ and $\widehat{W}(\theta)$. \begin{thm}
\label{KantorovichDual}Let $H_{\theta},P_{\theta}\in\mathscr{P}^{1}(\mathcal{X})$.
The dual problem to \eqref{eq:Kantonovich} by Kantorovich's duality
principle is: 
\[
W(\theta)=\|H_{\theta}(x,\varepsilon)-P_{\theta}(x,\varepsilon)\|_{l}^{*}:=\sup_{\|f\|_{l}\leq1}\left|\int fd\left\{ H_{\theta}(x,\varepsilon)-P_{\theta}(x,\varepsilon)\right\} \right|,
\]
where $\|f\|_{l}:=\sup_{x\neq y\in\mathcal{X}}|f(q)-f(q')|/\mathbf{d}(q,q')$
is the Lipschitz semi-norm for real valued continuous $f$ on $\mathcal{Q}$.
Here $d\{\cdot\}$ is the differential notation. The optimal transportation
problem for discrete measure $H_{n\theta}(x,\varepsilon)$ and $P_{m\theta}(x,\varepsilon)$
becomes 
\begin{align*}
\widehat{W}(\theta)=  \|H_{n\theta}(x,\varepsilon)-P_{m\theta}(x,\varepsilon)\|_{l}^{*}:= & \sup_{\|f\|_{l}\leq1}\left|\int fd\left\{ H_{n\theta}(x,\varepsilon)-P_{m\theta}(x,\varepsilon)\right\} \right|, \\
H_{n\theta} = \frac{1}{n^{2}}\left[\sum_{i=1}^{n}\sum_{j=1}^{n}\mathbf{1}\{x_{i}\leq x,\rho(z_{j},\theta)\leq\varepsilon\}\right], & \, P_{n\theta}= \frac{1}{n^{2}}\left[\sum_{i=1}^{n}\mathbf{1}\{x_{i}\leq x\}\right]\left[\sum_{i=1}^{n}\mathbf{1}\{\rho(z_{i},\theta)\leq\varepsilon\}\right].
\end{align*}
\end{thm} \begin{proof} Please refer to \citet[Theorem 11.8.2,][]{Dudley2002}.
Note that \citet[Theorem 11.8.2,][]{Dudley2002} gives a stronger
result which is for complete separable metric space. But the original
result in \citet{Rubinshtein1970} is for the compact metric space
$\mathcal{Q}$. \end{proof}

The integrals in Theorem \ref{KantorovichDual} are with respect to
empirical measures. In practice, these integrals are replaced by summations
over observations. Recall that $Q$ and $Q'$ are drawn from different
distributions. The empirical Wasserstein distance in \eqref{eq:Wdist} becomes 
\begin{equation}
\widehat{W}(\theta):=\sup_{\|f\|_{l}\leq1}\left|\frac{1}{n}\sum_{i=1}^{n}f(q_{i})-\frac{1}{m}\sum_{j=1}^{m}f(q_{j}^{'})\right|.\label{eq:empiricalMMD}
\end{equation}
where $m$ is not necessarily the same as $n$. 

The criterion function in \eqref{eq:empiricalMMD} is similar to the
total variation distance between $H_{n\theta}$ and $P_{m\theta}$.
Thus, the estimation problem based on integral distance function in
BW is modified to the one based on the total variation distance. The
distance between $H_{n\theta}$ and $P_{m\theta}$ is evaluated by
the comparison of means of $f(\cdot)$ w.r.t. $H_{n\theta}$ and $P_{n\theta}$.
The way of handling two different integrals or summations separately
makes the estimation more feasible. One advantage of this separation
is that the empirical c.d.f. of $G_{\theta}(\varepsilon)$ allows
for resampling scheme: 
\begin{equation*}
m^{-1}\sum_{j=1}^{m}\mathbf{1}\{\rho(x_{j}^{'},y_{j}^{'},\theta)\leq\varepsilon\},
\end{equation*}
where $\{x_{j}^{'},y_{j}^{'}\}_{j\leq m}$ are resamples of $\{x_{i},y_{i}\}_{i\leq n}$.
It also allows for other smoothing techniques for estimating $G_{\theta}(\varepsilon)$.
If the sample size is small, the smoothing technique becomes important
for identifying $\theta_{0}$. Notice that $F_{n}(x)G_{n\theta}(\varepsilon)\neq H_{n\theta}(x,\varepsilon)$ 
\begin{equation*}
\frac{1}{n^{2}}\left[\sum_{i=1}^{n}\mathbf{1}\{x_{i}\leq x\}\right]\left[\sum_{i=1}^{n}\mathbf{1}\{\rho(z_{i},\theta)\leq\varepsilon\}\right]\neq\frac{1}{n^{2}}\left[\sum_{i=1}^{n}\sum_{j=1}^{n}\mathbf{1}\{x_{i}\leq x,\rho(z_{j},\theta)\leq\varepsilon\}\right],
\end{equation*}
for given $(x,\varepsilon)$. Thus if $\theta\neq\theta_{0}$, we
have that
\begin{equation*}
\sum_{i=1}^{n}f(z,\varepsilon)\left[\delta_{\rho(z_{i},\theta)}(\varepsilon)\right]\left[\delta_{x_{i}}(x)\right]\neq\sum_{i=1}^{n}\sum_{j=1}^{n}f(z,\varepsilon)\left[\delta_{\{x_{i},\rho(z_{j},\theta)\}}(x,\varepsilon)\right].
\end{equation*}
But if the sample size is small, i.e. $n=10$, it is very likely to
have an equality from the above expression even if $\theta\neq\theta_{0}$. Nevertheless, with re-sampling devices, one can reduce this small samples issue.

The dual distance $\|H_{\theta}(x,\varepsilon)-P_{\theta}(x,\varepsilon)\|_{l}^{*}$
is an instance of the integral probability metric which has been used
in proving central limit theorems and in empirical process theory
\citep{Muller1997}. The dual distance measures the dis-similarity
between $ $$H_{\theta}(x,\varepsilon)$ and $P_{\theta}(x,\varepsilon)$
on the basis of samples drawn from each of them, by finding a well
behaved function which is large on the points drawn from $H_{\theta}(x,\varepsilon)$
and small on the points drawn from $P_{\theta}(x,\varepsilon)$. This is
equivalent to measure the difference between the mean function values
on two samples.

Theorem \ref{KantorovichDual} establishes the dual problem, and hence the objective function of interest. 
However, there are infinite candidates $f$ satisfying $\|f\|_{l}\leq1$ in equation \eqref{eq:empiricalMMD}.
Thus, both Kantorovich's dual $W(\theta)$ and empirical estimate criterion
$\widehat{W}(\theta)$ are criterion functions over infinite dimensions.
Before we establish identification and the estimator of interest we provide yet another representation to the problem which allows for practical implementation.
Therefore, in the next section, we embed this infinite dimensional problem into a tractable
form so that we can represent the criterion function.

\subsection{Kernel Representation \label{sec:Representation-Theorem}}

Now we establish a kernel representation theorem, such that 
we are able to embed the infinite dimensional criterion function in \eqref{eq:empiricalMMD} into a tractable
space. In doing this, the criterion function is represented in terms of a kernel-based
distance.

\subsubsection{Kernel Representation Theorem}

In this paper, we only consider
continuous kernel because every function induced by a continuous kernel
is also continuous. For a continuous kernel $\mathbf{k}$ on a compact
metric space $(\mathcal{X},d)$, we are interested in the space of
all functions induced by $\mathbf{k}$ that is \emph{dense} in $C(\mathcal{X})$ (the space of \emph{all continuous functions}),
i.e. for every function $f\in C(\mathcal{X})$ and every $\epsilon>0$
there exists a function $g$ induced by $\mathbf{k}$ with 
\[
\|f-g\|_{\infty}\leq\epsilon.
\]

\begin{defn*} If a kernel induces a dense set of $C(\mathcal{X})$,
we say that the kernel is able to \emph{represent} $C(\mathcal{X})$.
\end{defn*}
Next we provide the a kernel representation result.

\begin{thm} \label{Representation-Theorem}(Representation Theorem)
Let $(\mathcal{X},\mathbf{d})$ be a compact metric space and $\mathbf{k}$
is a kernel on $\mathcal{X}$ with $\mathbf{k}(x,x)>0$ for all $x\in\mathcal{X}$.
Suppose that we have an injective feature map $\Phi:\mathcal{X}\rightarrow\ell_{2}$
of $\mathbf{k}$ with $\Phi(x)=(\Phi_{n}(x))_{n\in\mathbb{N}}$. If
$\mathcal{A}:=\mbox{span}\{\Phi_{n}:n\in\mathbb{N}\}$ is an algebra
then $\mathbf{k}$ represents $C(\mathcal{X})$.
\end{thm} 
\begin{proof}
This is a direct result of representation for reproducing kernels (Theorem \ref{thm:Reproducing}) and Stone-Weierstrass theorem (Thereom \ref{thm:SWtheorem}).  
\end{proof}
In the Appendix, Section \ref{sub:Mathematical-Preliminary}, we state additional mathematical results on
the existence of the kernel function $\mathbf{k}(q)$ for $\mathcal{F}:=\{f\,:\,\|f\|_{l}\leq1\}$.
%People who do not intend to read the proof and/or are familiar with the technicalities of the representation theory could skip the following subsection and go to Section \ref{sub:-Distance} for main results of kernel representations.

\subsubsection{$W_{\mathcal{H}}(\theta)$-Distance \label{sub:-Distance}}

The quality of empirical Wasserstein distance $\widehat{W}(\theta)$
as a statistic depends on the class $\mathcal{F}:=\{f\,:\,\|f\|_{l}\leq1\}$.$ $
There are two requirements: (1) $\mathcal{F}$ must be large enough
so that $\widehat{W}(\theta)=0$ if and only if $H_{n\theta}(x,\varepsilon)=P_{n\theta}(x,\varepsilon)$;
(2) for the distance to be consistent, $\mathcal{F}$ needs to be
restrictive enough for the empirical estimate of $\widehat{W}(\theta)$
to converge to its expectation. We will use a unit ball in the Hilbert
space where all functions in the class $\mathcal{F}$ can be represented
by the elements in this ball. This representation will be shown to
satisfy both of the required properties.

In this paper, the class $\mathcal{F}$
is chosen as follows. Given observations $\mathbf{y}=(y_{1},\dots y_{n})$
and $y_{i}=f(x_{i},\varepsilon_{i})$ for an unknown function $f$
defined in condition C9, we consider a kernel function $\mathbf{k}(q)$
in $\mathcal{F}$. This kernel function has an inner product representation
$\mathbf{k}(q,\cdot)=\left\langle \Phi(q),\Phi(\cdot)\right\rangle $.
The function $\Phi$ in the representation approximates any $f\in\mathcal{F}$
such that for any $\epsilon>0$, there exists $\alpha$ as the coefficient
and 
\[
\left\Vert f(q)-\left\langle \alpha,\Phi(q)\right\rangle \right\Vert _{\infty}\leq\epsilon
\]
with $\epsilon$-degree of accuracy. 

Now we propose a new distance $W_{\mathcal{H}}(\theta)$ that
characterizes the distance between two distinct measures $H_{n\theta}(x,\varepsilon)$
and $P_{n\theta}(x,\varepsilon)$ for $\theta\neq\theta_{0}$ in the
optimal sense. The idea is to consider $\mathcal{F}:=\{f\,:\,\|f\|_{l}\leq1\}$
to be embedded in a unit ball $\mathcal{B}_{\mathcal{H}}$ of a Hilbert
space $\mathcal{H}$.%
\footnote{$\mathcal{F}$ is embedded in $\mathcal{B}_{\mathcal{H}}$ means that
there exists a map $\Pi:\mathcal{F}\rightarrow\mathcal{B}_{\mathcal{H}}$
such that the map is injective and preserving topological structure
of $\mathcal{F}$%
} For any function $f\in\mathcal{B}_{\mathcal{H}}$ there exists a
feature map $\Phi:\mathcal{\mathcal{Q}}\rightarrow\mathcal{H}$. Then
$\mathcal{H}$ is the feature space of a kernel $\mathbf{k}$ which represents
all continuous bounded functions in $\mathcal{F}$. In other words,
we construct a Hilbert space so that 
\[
\mathcal{F}:=\{f\,:\,\|f\|_{l}\leq1\}=\{f\,:\,\|f\|_{\mathcal{H}}\leq1\}=\mathcal{B}_{\mathcal{H}}.
\]
By Theorem \ref{Representation-Theorem} we know that a kernel $\mathbf{k}$
induced by the feature map $\Phi:\mathcal{\mathcal{Q}}\rightarrow\mathcal{H}$
can represent $f\in\mathcal{B}_{\mathcal{H}}$ and hence after $f\in\mathcal{F}$.
Then the same kernel will be used for defining the metric for $W(\theta)$.

Here is our construction. First, we construct a pseudo-metric on $\mathcal{B}_{\mathcal{H}}$
by using the kernel $\mathbf{k}$. Kernel $\mathbf{k}$ is induced
by $\mathcal{H}$. Then we will show that the new metric characterizes
the weak topology of $\widehat{W}(\theta)$ and in Section \ref{sub:Biased-and-Unbiased} below
we develop an asymptotically unbiased empirical estimate for $W(\theta)$ based
on this new metric. 
\begin{itemize}
\item[C10.] For any $\theta\in\Theta$, the kernel $\mathbf{k}$ is square-root
integrable w.r.t. both $H_{\theta}(x,\varepsilon)$ and $P_{\theta}(x,\varepsilon)$.
That is 
\[
\int_{\mathcal{Q}}\sqrt{\mathbf{k}(q,q)}dH_{\theta}(q)<\infty\:\mbox{and }\int_{\mathcal{Q}}\sqrt{\mathbf{k}(q,q)}dP_{\theta}(q)<\infty.
\]
and $\mathbf{k}(\cdot,\cdot)\in\mathcal{B}_{\mathcal{H}}$ is measurable
on $\mathcal{Q}\times\mathcal{Q}$. 
\end{itemize}
Let $\mathscr{P}(\mathcal{Q}):=\left\{ \mathcal{P}:\int_{\mathcal{Q}}\sqrt{\mathbf{k}(q,q)}d\mathcal{P}(q)<\infty\right\} $.
Condition C10 implies $\{H_{\theta}(x,\varepsilon),P_{\theta}(x,\varepsilon)\}\subset\mathscr{P}(\mathcal{Q})$.
Note that $\mathscr{P}(\mathcal{Q})$ is a subset of $\mathscr{P}^{1}(\mathcal{Q})$,
the space of Borel probability measures on $\mathcal{Q}$ with the
first moments. C10 restricts the support $\mathscr{P}^{1}(\mathcal{Q})$
in $W(\theta)$. However, it returns flexibility of defining pseudo-metric
for $\mathscr{P}(\mathcal{Q})$ $ $by kernels. \begin{thm} \label{thm:KernelW-distance}Assume
that the simultaneous model $\rho(X,Y,\theta)=\varepsilon$ satisfies
C1-C9. If $H_{\theta}(x,\varepsilon)$ and $P_{\theta}(x,\varepsilon)$
satisfy C10, then 
\begin{equation}
W_{\mathcal{H}}(\theta):=\left\Vert \underset{:=H_{\theta}\mathbf{k}}{\underbrace{\int_{\mathcal{Q}}\mathbf{k}(\cdot,q)dH_{\theta}(q)}}-\underset{:=P_{\theta}\mathbf{k}}{\underbrace{\int_{\mathcal{Q}}\mathbf{k}(\cdot,q)dP_{\theta}(q)}}\right\Vert _{\mathcal{H}},\label{eq:WH_distance}
\end{equation}
where $\mathcal{H}$ induces kernel $\mathbf{k}\in\mathcal{B}_{\mathcal{H}}$.\end{thm}
\begin{proof} See Appendix \ref{sub:Proof-of-Theorem WHdistance}.
\end{proof} 

Therefore, the minimum integrated distance estimator proposed in this paper will be defined below as a minimization of problem in equation \eqref{eq:WH_distance}.

As stated in Theorem \ref{Representation-Theorem}, an
algebra $\mathcal{A}:=\mbox{span}\{\Phi_{n}:n\in\mathbb{N}\}$ associating
with a sequence of $(\Phi_{n}(x))_{n\in\mathbb{N}}$ will represent
any $f\in\mathcal{F}\subset C(\mathcal{Q})$. By using the inner product
operation, we have 
\[
f=\left\langle \alpha,\Phi(\cdot)\right\rangle =\left\langle \left\langle \beta,\Phi(\cdot)\right\rangle ,\Phi(\cdot)\right\rangle =\left\langle \alpha',\mathbf{k}(\cdot,\cdot)\right\rangle ,
\]
where $\alpha$, $\beta$, and $\alpha'$ are some coefficients in
$\mathbb{R}$. The bilinear map $\left\langle \alpha',\cdot\right\rangle :\mathcal{F}\rightarrow\mathcal{B}_{\mathcal{H}}$
is injective and preserving the topological structure of $\mathcal{F}$.
Thus we can say $\mathcal{F}$ is embedded into a unit ball $\mathcal{B}_{\mathcal{H}}$
in the Hilbert space. The representation of $W_{\mathcal{H}}(\theta)$
yields a mapping from $\mathscr{P}(\mathcal{Q})$ to a Hilbert space
$\mathcal{H}$: 
\[
\mathcal{P}\mathbf{k}:\:\mathcal{P}\mapsto\int_{\mathcal{Q}}\mathbf{k}(\cdot,q)d\mathcal{P}(q)
\]
where $\mathcal{P}\in\mathscr{P}(\mathcal{Q})$ and $\mathbf{k}\in\mathcal{B}_{\mathcal{H}}$.
In order to verify that $W_{\mathcal{H}}(\theta)$ is a distance function,
we need the following corollary. \begin{cor} \label{cor:Metrization}$W_{\mathcal{H}}(\theta)$
metrizes the weak topology on $\mathscr{P}$.\end{cor} \begin{proof}
See Appendix \ref{sub:Proof-of-Corollary Metrization}. \end{proof}
$W_{\mathcal{H}}(\theta)$ measures the dis-similarity between $H_{\theta}$
and $P_{\theta}$, however, so far it is not clear that whether the
new distance $W_{\mathcal{H}}(\theta)$ relates to $W(\theta)$. 
Corollary \ref{cor:Metrization} states that all elements in $\mathscr{P}(\mathcal{Q})$ can be metrized by $W_{\mathcal{H}}(\theta)$. The following condition ensures the regularity of $W_{\mathcal{H}}(\theta)$ for
all $q\in\mathcal{Q}$. With these results, $W_{\mathcal{H}}(\theta)$ is comparable with $W(\theta)$ in the probability metric space as given in the next result.
\begin{itemize}
\item[C11.] Let $\sup_{q\in\mathcal{Q}}\mathbf{k}(q,q)\leq C$ and $\mathbf{d}_{\mathbf{k}}=\|\mathbf{k}(\cdot,q)-\mathbf{k}(\cdot,q')\|_{\mathcal{H}}$.
We assume that the metric space $(\mathcal{Q},\rho)$ is separable. 
\end{itemize}
\begin{thm} \label{Thm:EquivalentMetric}For any $H_{\theta}(x,\varepsilon)$
and $P_{\theta}(x,\varepsilon)$ satisfy C10 and $(\mathcal{Q},\mathbf{d}_{\mathbf{k}})$
satisfies C11, then 
\[
W_{\mathcal{H}}(\theta)\leq W(\theta)\leq\sqrt{W_{\mathcal{H}}^{2}(\theta)+2C}
\]
where the distance function in $W(\theta)$ is $\mathbf{d}_{\mathbf{k}}$.\end{thm}
\begin{proof} See Appendix \ref{sub:Proof-of-Theorem EquivMetric}
\end{proof} 
\begin{rem} The metric $\mathbf{d}_{\mathbf{k}}(\delta_{q},\delta_{q'})=\|\mathbf{k}(\cdot,q)-\mathbf{k}(\cdot,q')\|_{\mathcal{H}}$
used in Theorem \ref{Thm:EquivalentMetric} is called Hilbertian metric.
Given this metric, one can obtain the associated kernel $\mathbf{k}$
via 
\[
\mathbf{k}(q,q')=\frac{1}{2}\left[\mathbf{d}_{\mathbf{k}}^{2}(q,z)+\mathbf{d}_{\mathbf{k}}^{2}(q',z)-\mathbf{d}_{\mathbf{k}}^{2}(q,q')\right]
\]
for any $q,q',z\in\mathcal{Q}$. This is called three points interpolation.
\end{rem}
To conclude this section, note that we use a general Wasserstein criterion \eqref{eq:Kantonovich}
to measure the dis-similarity between these two sequences. The criterion
captures infinite intrinsic connections between $H_{\theta}(x,\varepsilon)$
and $P_{\theta}(x,\varepsilon)$ via a transition measure $\gamma$.
The measure $\gamma$ in Wasserstein criterion \eqref{eq:Kantonovich}
is the key mechanism to provide an optimal criterion function of the
minimum distance problem over two probability measures. By Kantonovich's
dual \eqref{eq:empiricalMMD} and kernel representation result in \eqref{eq:WH_distance},
we embed this infinite dimensional criterion function into a tractable
space. Then the criterion function is represented in terms of a kernel-based
distance. Since the new distance is equivalent to the total variation
distance, the modified estimation problem maintains finer topological
details. 

Given these results we are in position to establish identification, propose the estimator, and derive its asymptotic properties.

\section{Estimation and Inference\label{sec:Inference}}
In this section we first establish identification in the dual problem, then construct the new minimum integrated distance estimator $\theta^{*}$ for $\theta_{0}$ and finally provide the corresponding asymptotic properties, and discuss its optimality. Practical computation of the proposed estimator will be given in Section \ref{sec:Computation}.

\subsection{Identification\label{sub:Identification}}

In order to identify $\theta_{0}$, we need to make sure that when
$H_{\theta}\neq P_{\theta}$, the kernel based distance $W_{\mathcal{H}}(\theta)$
cannot attain zero. The following condition restricts the class of
kernels to be strictly positive definite. 
\begin{itemize}
\item[C12.] \label{Con:PositiveDefinite}The $k$ is strictly positive definite
on $\mathcal{Q}$ such that 
\[
\int_{\mathcal{Q}\times\mathcal{Q}}\mathbf{k}(q,q')d\mathcal{P}(q)d\mathcal{P}(q')>0
\]
for any $\mathcal{P}\in\mathscr{P}$ and $\mathscr{P}(\mathcal{Q}):=\left\{ \mathcal{P}:\int_{\mathcal{Q}}\sqrt{k(q,q)}d\mathcal{P}(q)<\infty\right\} $. \end{itemize}
\begin{example*}
There are many kernels satisfying the integrally strictly positive
definite condition. Here we list few of them: Gaussian kernel $\exp(-\sigma\|q-q'\|_{2}^{2})$
with $\sigma>0$; inverse multi-quadratics $(\sigma^{2}+\|q-q'\|_{2}^{2})^{-c}$
with $c>0$ and $\sigma>0$, Laplacian kernel $\exp(-\sigma\|q-q'\|_{1})$,
with $\sigma>0$. 
\end{example*}
\begin{lem} \label{lem:Identification}Let $k$ be a bounded kernel
on a metric space $(\mathcal{Q},\mathbf{d}_{\mathbf{k}})$. Let $H_{\theta}(x,\varepsilon)$
and $P_{\theta}(x,\varepsilon)$ satisfy C1-C12. Then the following
statement is true that $H_{\theta}(x,\varepsilon)\neq P_{\theta}(x,\varepsilon)$
but $W_{\mathcal{H}}(\theta)\neq0$ if and only if there exists a
finite non-zero signed Borel measure $\mathcal{P}\in\mathscr{P}(\mathcal{Q})$
that satisfies:

(I) $\int_{\mathcal{Q}\times\mathcal{Q}}\mathbf{k}(q,q')d\mathcal{P}(q)d\mathcal{P}(q')=0$.
(II) $\mathcal{P}(\mathcal{Q})=0$.\end{lem} \begin{proof} See Appendix
\ref{sub:Proof-of-Lemma Iden}.\end{proof} \begin{thm} For any $H_{\theta}(x,\varepsilon)$
and $P_{\theta}(x,\varepsilon)$ satisfy C10 and kernel $\mathbf{k}$
satisfies C11, if $H_{\theta}(x,\varepsilon)\neq P_{\theta}(x,\varepsilon)$,
then $W_{\mathcal{H}}(\theta)$ does not attain zero.\end{thm} \begin{proof}
Condition C12 implies that for any non-zero signed Borel measure $\mathcal{P}$,
(I) in Lemma \ref{lem:Identification} is violated. Therefore, if
$H_{\theta}(x,\varepsilon)\neq P_{\theta}(x,\varepsilon)$, then $W_{\mathcal{H}}(\theta)$
must be non-zero.\end{proof} \begin{rem} In fact, both $W_{\mathcal{H}}(\theta)$
and $W(\theta)$ are equivalent to the total variation distance 
\[
\sup_{q\in\mathcal{Q}}\left|H_{\theta}(q)-P_{\theta}(q)\right|
\]
if the distance function $\mathbf{d}(\cdot,\cdot)$ in the primal
problem \eqref{eq:Kantonovich} of $W(\theta)$ is chosen to be a
Hamming metric $H(x,y)=\mathbf{1}_{x\neq y}$. This result is given
in \citet[Proposition 1.3,][]{GozlanLeonard2010}. \end{rem}

\subsection{Convergence of $P_{n \theta }$ and $H_{n \theta}$}

In this paper, we only consider pre-compact space $(\mathscr{P},\mathbf{d})$ over Hilbert spaces (complete spaces) which means that the space $(\mathscr{P},\mathbf{d})$ is a \emph{compact metric space}. In other words, the covering numbers of $\mathscr{P}$ are finite and distance $\mathbf{d}(\cdot,\cdot)$ metrizes $\mathscr{P}$ by a Hilbertian type metric. The covering number measures the size of the function set $\mathscr{P}$. Notice that for a finite covering number, the Glivenko-Cantelli theorem (uniform law of large number) always holds. Thus, sample counterparts of $P_{\theta }, H_{\theta}$  converge. The rigorous probability theory argument is given in the Appendix \ref{sub:Prob-Theo-Preliminary}.

\subsection{$\sqrt{n}$-Biased and Unbiased Estimates\label{sub:Biased-and-Unbiased}}

Given the identification of the dual problem stated in the previous section, we are able to estimate the parameters of interest. 
The minimum integrated distance estimator, $\theta^{*}$, is defined as follows
\begin{align} \label{eq:WH_distance_sample}
\theta^{*} & =\arg\min_{\theta\in\Theta}\widehat{W}_{\mathcal{H}}(\theta):=\arg\min_{\theta\in\Theta} \left\Vert \int \mathbf{k}((x,\varepsilon),q)dH_{n\theta}(q)- \int \mathbf{k}((x,\varepsilon),q')dP_{n\theta}(q')\right\Vert _{\mathcal{H}},
\end{align}
where $\mathbf{k}((x,\varepsilon),\cdot) = \left\langle \Phi(x,\varepsilon),\Phi(\cdot)\right\rangle $ is the kernel function on a feature space, and $\Phi(\cdot)$ is the feature map such that the representation $f= \left\langle \alpha,\Phi(\cdot)\right\rangle$ is alway valid for any $f\in\mathcal{F}$. This criterion is a restatement of the minimization problem in \eqref{eq:WH_distance}.

The estimator of the empirical counterpart of $W_{\mathcal{H}}(\theta)$, $\widehat{W}_{\mathcal{H}}(\theta)$, defined in \eqref{eq:WH_distance_sample}, has an asymptotic bias.
This is because $W_{\mathcal{H}}(\theta)$ is an equivalent metric as total variation distance, as pointed out
in \citet{DevroyeGyorfi1990}, not all distributions have convergent
results for empirical probability measure with respect to the total
variation distance. The consequence of using $\widehat{W}_{\mathcal{H}}(\theta)$-distance
is to obtain an $\sqrt{n}$ or $\sqrt{m}$-inconsistent estimator.
This could bring serious issues for hypothesis testing or some
other weakly convergent outcomes. 

In this subsection, we first consider the regular consistency result
of $\widehat{W}_{\mathcal{H}}(\theta)$. It serves a baseline criterion
for the following analyses. The consistency result also shows that
$\widehat{W}_{\mathcal{H}}(\theta)$ is not $\sqrt{n}$ or $\sqrt{m}$-consistent.
Then, we propose an asymptotically unbiased criterion by simply taking the square
of $\widehat{W}_{\mathcal{H}}(\theta)$. We show that the $\sqrt{n}$
or $\sqrt{m}$-bias terms will be eliminated.

The following theorem establishes that $\widehat{W}_{\mathcal{H}}(\theta)$
is consistent for $W_{\mathcal{H}}(\theta)$ as $n$ or $m$ diverges
to infinity. But it is not $\sqrt{n}$ or $\sqrt{m}$-consistent.
\begin{thm} \label{Thm:Bias}Given Conditions C1-C12, we can assume
$0\leq\mathbf{k}(q,q')\leq C_{\mathbf{k}}$ for any $q$ and $q'$,
then 
\[
\Pr\left\{ \widehat{W}_{\mathcal{H}}(\theta)-W_{\mathcal{H}}(\theta)>2(\sqrt{C_{\mathbf{k}}/m}+\sqrt{C_{\mathbf{k}}/n})+\epsilon\right\} \leq2\exp\left(\frac{-\epsilon^{2}mn}{2C_{\mathbf{k}}(m+n)}\right)
\]
where $\epsilon$ is an arbitrarily small number.\end{thm} \begin{proof}
See Appendix \ref{sub:Proof-of-Theorem Bias}. 
\end{proof} 
Theorem \ref{Thm:Bias} shows that $\widehat{W}_{\mathcal{H}}(\theta)$ is
not $\sqrt{n}$ or $\sqrt{m}$-consistent for $W_{\mathcal{H}}(\theta)$.
The statistics $n^{\frac{1}{2}}\left\{ \widehat{W}_{\mathcal{H}}(\theta)-W_{\mathcal{H}}(\theta)\right\} $
or $m^{\frac{1}{2}}\left\{ \widehat{W}_{\mathcal{H}}(\theta)-W_{\mathcal{H}}(\theta)\right\} $
has a bias term $C_{\mathbf{k}}$ that does not vanish even if $\epsilon\rightarrow0$
and $n,m\rightarrow\infty$. It also provides an insight that the
bias $(C_{\mathbf{k}}/m)^{\frac{1}{2}}+(C_{\mathbf{k}}/n)^{\frac{1}{2}}$
is related to the kernel size.

Now we propose the square of $W_{\mathcal{H}}(\theta)$-distance as
the new criterion distance and will show that this distance leads
to an asymptotically unbiased estimate. As the proof in Lemma \ref{lem:Identification},
the square of $W_{\mathcal{H}}(\theta)$-distance is: 
\begin{align*}
S_{\mathcal{H}}(\theta):=[W_{\mathcal{H}}(\theta)]^{2}= & \left\Vert H_{\theta}\mathbf{k}-P_{\theta}\mathbf{k}\right\Vert _{\mathcal{H}}^{2}=\left\langle H_{\theta}\mathbf{k}-P_{\theta}\mathbf{k},\, H_{\theta}\mathbf{k}-P_{\theta}\mathbf{k}\right\rangle \\
= & \int_{\mathcal{Q}\times\mathcal{Q}}\mathbf{k}(q,q')d(H_{\theta}-P_{\theta})(q)d(H_{\theta}-P_{\theta})(q').
\end{align*}
Given the norm $\left\Vert H_{\theta}\mathbf{k}-P_{\theta}\mathbf{k}\right\Vert _{\mathcal{H}}^{2}$,
$S_{\mathcal{H}}(\theta)$ may be easily computed in terms of kernel functions. This leads
to our empirical estimate of the $S_{\mathcal{H}}(\theta)$: 
\begin{equation}
\widehat{S}_{\mathcal{H}}(\theta):=\frac{1}{m(m-1)}\sum_{i\neq j}^{m}h(q_{i},q_{j}),\label{eq:Square-Empirical-Distance}
\end{equation}
when $m=n$ in \eqref{eq:empiricalMMD} and $h(\cdot,\cdot)$ is a
$U$-statistic such that $h(q_{i},q_{j}):=\mathbf{k}(x_{i},x_{j})+\mathbf{k}(\varepsilon_{i},\varepsilon_{j})-\mathbf{k}(x_{i},\varepsilon_{j})-\mathbf{k}(x_{j},\varepsilon_{i})$.
In general, if $m\neq n$ in \eqref{eq:empiricalMMD}, then we have
\begin{equation}
\widehat{S}_{\mathcal{H}}(\theta):=\left[\frac{1}{m^{2}}\sum_{i\neq j}^{m}\mathbf{k}(x_{i},x_{j})-\frac{2}{mn}\sum_{i,j=1}^{m,n}\mathbf{k}(x_{i},\varepsilon_{j})+\frac{1}{n^{2}}\sum_{i\neq j}^{n}\mathbf{k}(\varepsilon_{i},\varepsilon_{j})\right],\label{eq:Square-Empirical-Distance-General}
\end{equation}
The consistent statistical distance base on the following theorem,
which is a straightforward application of the large deviation bound
on $U$-statistics of the Hoeffding bound. The next theorem shows consistency of the estimator $\theta^{*}$.
\begin{thm} \label{thm:Unbias}Given C1-C12, we can assume $0\leq\mathbf{k}(q,q')\leq C_{\mathbf{k}}$,
from which it follows $-2C_{\mathbf{k}}\leq h(q,q')\leq2C_{\mathbf{k}}$
for any $q$ and $q'$. Define $\theta^{*}$ as 
\[
\widehat{h}(\theta^{*}|\Theta)=\inf_{\theta\in\Theta}\frac{1}{n^{2}}\sum_{i\neq j}^{n}h(q_{i},q_{j}),\;\mbox{s.t.}\: q\sim H_{n\theta},q'\sim P_{n\theta}.
\]
Assume that for each $\theta\neq\theta_{0}$, there is an open set
$B_{\theta}$ such that $\theta\in B_{\theta}$ and $h(\theta|B_{\theta})>0$
\[
h(\theta^{*}|B_{\theta})=\inf_{\theta\in B_{\theta}}\int h(q,q')dH_{\theta}(q)dP_{\theta}(q'),\;\mbox{s.t.}\: q\sim H_{\theta},q'\sim P_{\theta}.
\]
If $\Theta$ is not compact, assume further that there exists compact
$\mathcal{K}\subseteq\Theta$ such that $\theta_{0}\in\mathcal{K}$
and $h(\theta|\mathcal{K}^{c})>0$. Then, 
\begin{equation}
\Pr\bigl\{\limsup_{n\to\infty}|\theta^{*}-\theta_{0}|\geq\epsilon\bigr\}=0.\label{eq:mle.conv}
\end{equation}
\end{thm} \begin{proof} See Appendix \ref{sub:Proof-of-Theorem-Unbias}.
\end{proof}

\subsection{Asymptotic Distribution of the Unbiased Statistic\label{sub:Asymptotic-Distribution}}

The $\sqrt{n}$-inconsistency of the baseline criterion $\widehat{W}_{\mathcal{H}}(\theta)$
motivates us to consider another statistic, $\widehat{S}_{\mathcal{H}}(\theta)$.
Root-$n$ asymptotic bias is eliminated by taking a quadratic form of $\widehat{W}_{\mathcal{H}}(\theta)$.
Then, the asymptotic unbiased estimator is valid for hypothesis testing since
it induces an unbiased statistic. This asymptotic unbiased statistic has a standard
Gaussian limiting distribution with root-$n$ rate. It is a nonparametric
test statistic and therefore its asymptotic properties may fit a more
general framework beyond the current setting. 

The following two results derive the limiting distribution of $\widehat{S}_{\mathcal{H}}(\theta^{*})$, where $\theta^{*}$ is given in \eqref{eq:WH_distance_sample}.

\begin{thm} \label{thm:WeakConvergence} Under conditions of Theorem \ref{thm:Unbias}, the statistic $\widehat{S}_{\mathcal{H}}(\theta^{*})$
converges in distribution to
\[
n\widehat{S}_{\mathcal{H}}(\theta^{*})\rightsquigarrow\sum_{l=1}^{\infty}\lambda_{l}\left[z_{l}^{2}-1\right]
\]
where $z_{l}\rightsquigarrow\mathcal{N}(0,1)$ and $z_{l}$ are i.i.d.,
$\lambda_{l}=(\lambda_{l}^{(1)}+\lambda_{l}^{(2)})$ are the solutions
to the eigenvalue equations 
\begin{align*}
\int\tilde{\mathbf{k}}(x,\cdot)\psi_{l}(x)dF_{\theta_{0}}(x) & =\lambda_{l}^{(1)}\psi_{l}(\cdot)\\
\int\tilde{\mathbf{k}}(\varepsilon,\cdot)\psi_{l}(\varepsilon)dG_{\theta_{0}}(\varepsilon) & =\lambda_{l}^{(2)}\psi_{l}(\cdot)
\end{align*}
where $\psi_{l}(\cdot)$ is the eigenfunction and 
\[
\tilde{\mathbf{k}}(x_{i},x_{j}):=\mathbf{k}(x_{i},x_{j})-\mathbb{E}_{x}\mathbf{k}(x_{i},x)-\mathbb{E}_{x}\mathbf{k}(x,x_{j})+\mathbb{E}_{x,x'}\mathbf{k}(x,x')
\]
is the centered kernel. In particular, $h(q_{i},q_{j})$ in \eqref{eq:Square-Empirical-Distance}
is $\tilde{\mathbf{k}}(x_{i},x_{j})+\tilde{\mathbf{k}}(\varepsilon_{i},\varepsilon_{j})-\tilde{\mathbf{k}}(x_{i},\varepsilon_{j})-\tilde{\mathbf{k}}(\varepsilon_{i},x_{j})$.
\end{thm} \begin{proof} See Appendix \ref{sub:Proof-of-Theorem-WeakConvergence}.
\end{proof}

\begin{thm} Let $\theta^{*}$ be the value which gives
the minimum distance of $\widehat{S}_{\mathcal{H}}(\theta)$. Under conditions of Theorem \ref{thm:Unbias}, it follows
that $\widehat{S}_{\mathcal{H}}(\theta^{*})$ converges in distribution
to a Gaussian random variable such that 
\[
\sqrt{n}\left(\widehat{S}_{\mathcal{H}}(\theta^{*})-S_{\mathcal{H}}(\theta_{0})\right)\rightsquigarrow\mathcal{N}(0,\sigma_{s}^{2}),
\]
where $\sigma_{s}^{2}:=4\left(\mathbb{E}_{q}\left[\mathbb{E}_{q'}h(q,q')\right]^{2}-\left[\mathbb{E}_{q,q'}h(q,q')\right]^{2}\right)$.
\end{thm} \begin{proof} With Theorem \ref{thm:WeakConvergence},
the proof follows directly from the result in \citet[p. 193 Theorem B,][]{Serfling1980}.
\end{proof}

\subsection{Optimal Information}

In this section we show the proposed estimator is more efficient than that in \citet{BrownWegkamp2002} (BW).
To derive an information matrix of $\theta^{*}$ comparable with that
of BW, we need additional assumptions:
\begin{itemize}
\item[C13.] The integrated minimum distance criterion function has a positive
definite second derivative matrix $\mathbf{V}$ at $\theta_{0}$.
\item[C14'.] Define $D_{\theta}(q)\equiv H_{\theta}(q)-P_{\theta}(q)$ is
differentiable at $\theta_{0}$ in $L_{2}(\mu)$ where $\mu$ is an
underlying probability measure for $q=(x,\varepsilon)$. Assume that
\begin{equation}
\lim_{\|\theta-\theta_{0}\|\rightarrow0}\int\left(\frac{D_{\theta}(q)-(\theta-\theta_{0})^{T}\varDelta(q)}{\|\theta-\theta_{0}\|}\right)^{2}d\mu(q)=0,\label{eq:DiffMean}
\end{equation}
where $\varDelta$ is the differentiable mean of $D_{\theta}(q)$
in $L_{2}(\mu)$.
\end{itemize}
Condition C13 refers to the condition A.7 in BW and it is about the
second derivative matrix of the criterion function $M(\theta)$. In our setting, the integrated
criterion function in C13 is $\widehat{W}_{\mathcal{H}}(\theta)$ which
always has second derivative if the kernel function is smooth and
second order differentiable. C14' requires the existence of differentiable
mean $\varDelta(q)$ in $L_{2}(\mu)$ for $D_{\theta}(q)$. The differentiable
mean $\varDelta(q)$ is essential to the information contained in
BW's estimator. The limiting covariance matrix of $n^{-1/2}(\hat{\theta}_{BW}-\theta)$
in BW is $4\mathbf{V}^{-1}\mathbf{W}\mathbf{V}^{-1}$, where $\mathbf{V}$
is the second derivative matrix of $M(\theta)$ and $\mathbf{W}$
is
\begin{align*}
\mathbf{W}= & \int\int\varDelta(q)\varDelta(q')\times\\
 & \underset{(*)}{\underbrace{\left[F(x)F(x')G(\varepsilon_{min})+F(x_{min})G(\varepsilon)G(\varepsilon')+H(x_{min},\varepsilon_{min})-3H(q)H(q')\right]}}d\mu(q)d\mu(q'),
\end{align*}
and $\varepsilon_{min}:=\varepsilon\wedge\varepsilon'$, $x_{min}:=x\wedge x'$.
The $(*)$ term in $\mathbf{W}$ is a weight for the product measure
$d(\mu(q)\times\mu(q'))$. 

We show that if there is a unique true $\mu$ in C14' then assumption
C14' is equivalent to the minimum optimal measure $\gamma$ in our
setting. The idea is that $\mathbf{W}$ comes from a simplification
of the following covariance term:
\begin{equation*}
\lim_{\|\theta-\theta_{0}\|\rightarrow0}\left(\int\varDelta(q)D_{\theta}(q)d\mu(q)\right)\left(\int\varDelta(q')D_{\theta}(q')d\mu(q')\right)^{T},
\end{equation*}
which is equivalent (in the weak topological sense) to $\lim_{\|\theta-\theta_{0}\|\rightarrow0}W_{\mathcal{H}}(\theta)$.
The optimal measure corresponds to the most efficient information.
We modify condition C14' in order to emphasize that endogenous effect
of the measure $\mu$ in simultaneous equations.
\begin{itemize}
\item[C14.] $D_{\theta}(q)$ has a differentiable quadratic mean (DQM) of
$\varDelta(q)$ at $\theta_{0}$ such that
\begin{equation}
\lim_{\|\theta-\theta_{0}\|\rightarrow0}\int\left(\frac{\sqrt{h_{\theta}(q)}-\sqrt{p_{\theta}(q)}-(\theta-\theta_{0})^{T}\varDelta_{w}(q)\mathbf{k}^{-1/2}(q)}{\|\theta-\theta_{0}\|}\right)^{2}w_{\theta}(q)dq=0,\label{eq:DiffMean-1}
\end{equation}
where $0<\mathbf{k}(q)<1$ is the kernel on $\mathcal{Q}$ such that
$\mathbf{k}(q)=\left\langle \Phi(\cdot),\Phi(q)\right\rangle $ such
that $Y=f(q)=\left\langle \alpha,\Phi(q)\right\rangle $ for some
coefficient $\alpha$ and the feature map $\Phi(\cdot)$. The endogenous
weight is embedded in $\mathcal{H}$ such that $w_{\theta}\in\mathcal{H}$
for any $\theta$ and at $\theta_{0}$,
\begin{equation*}
w_{\theta_{0}}^{1/2}(q)dq=d\left(\sqrt{H_{\theta_{0}}}-\sqrt{P_{\theta_{0}}}\right)(q).
\end{equation*}
\end{itemize}
C14 has two differences with C14'. First, we consider the DQM of $\sqrt{dD_{\theta}(q)}$
in C14 instead of the differentiable mean of $D_{\theta}(q)$. Second,
DQM of $\sqrt{dD_{\theta}(q)}$ is associated with an endogenous weight
$w_{\theta}(q)$. Assume
that the simultaneous model $\rho(X,Y,\theta)=\varepsilon$ satisfies
C1-C9. The idea of making these differences is to attain a better representation
for the derivative\footnote{The square root of density may not exist in general. The assumption of DQM induces a score involving the endogenous term $w_{\theta}$ that is comparable to the one in BW.}: 
\[
\varDelta_{w}(q)\approx(\sqrt{dD_{\theta}(q)w_{\theta}(q)})'=\frac{1}{2\sqrt{D_{\theta}(q)w_{\theta}(q)}}(D_{\theta}(q)w_{\theta}(q))'
\]
so that $\left.(\sqrt{D_{\theta}(q)w_{\theta}(q)})'\right|_{\theta=\theta_{0}}\times w_{\theta_{0}}^{-1/2}=\left.(\log D_{\theta}(q)w_{\theta}(q))'\right|_{\theta=\theta_{0}}=\iota_{\theta_{0}}(q)$
which is the information score. 

The information form in a minimum Hellinger distance with DQM condition
has the following result:
\begin{align*}
\lim_{\|\theta-\theta_{0}\|\rightarrow0}\mathbf{d}_{\mathbf{H}}^{2}(H_{\theta},FG_{\theta})= & \lim_{\|\theta-\theta_{0}\|\rightarrow0}\frac{1}{2}\int\left(\sqrt{h_{\theta}(q)w_{\theta}(q)}-\sqrt{p_{\theta}(q)w_{\theta}(q)}\right)^{2}dq\\
\approx & \frac{1}{2}(\theta-\theta_{0})^{T}\mathcal{I}(\theta_{0})(\theta-\theta_{0})^{T}
\end{align*}
where $\mathcal{I}(\theta_{0})$ is the Fisher type (or BW type) information
matrix. The following theorem shows that estimator based on $W_{\mathcal{H}}(\theta)$
or $S_{\mathcal{H}}(\theta)$-distance share the same form. 

\begin{thm} \label{thm:optimal-info}Given C1-C14, $S_{\mathcal{H}}(\theta)$
can be expressed as
\[
S_{\mathcal{H}}(\theta)=\max_{f\in\mathcal{F}}\int\left|h_{\theta}(q)-p_{\theta}(q)\right|^{2}f(q)dq.
\]
Let $S_{\mathcal{H}}(\theta)$ attain its minimum at $(\theta^{*},\mathbf{k})$.
The score of $S_{\mathcal{H}}(\theta)$ at $\theta_{0}$ is 
\[
\int\varDelta_{w}(q)\left(\frac{w_{\theta_{0}}}{\mathbf{k}}\right)^{\frac{1}{2}}d(H_{\theta_{0}}-P_{\theta_{0}})(q).
\]
The covariance matrix 
\[
\left[\int\left(\varDelta_{w}(q)\left(\frac{w_{\theta}}{w_{\theta^{*}}}\right)^{\frac{1}{2}}(q)\right)d(H_{\theta}-P_{\theta})(q')\right]\left[\int\left(\varDelta_{w}(q)\left(\frac{w_{\theta_{0}}}{w_{\theta^{*}}}\right)^{\frac{1}{2}}(q')\right)d(H_{\theta}-P_{\theta})(q')\right]^{T}
\]
which is optimal at $w_{\theta^{*}}=\mathbf{k}$. \end{thm} 

\begin{proof}See Appendix \ref{sub:Proof-of-Theorem optimal-info}
\end{proof}
\begin{rem*}
The densities $h_{\theta}(q)$ and $p_{\theta}(q)$ may not exist
in the previous general inference results. However, in order to distinguish
the optimal measure from the underlying measure $\mu$ in BW, we modify
C14' (a condition on probability measures $H_{\theta}$ and $P_{\theta}$)
to C14 (a condition on their densities). Thus, instead of using $d\mu$,
we use $w_{\theta}(q)dq$ to emphasize the important role of selecting the
optimal measure.
\end{rem*}

\section{Computation and Numerical Experiment}\label{sec:Computation}

\subsection{Computation of the Estimator}
Theorems \ref{Thm:Bias} and \ref{thm:Unbias} induce the equivalence
between $\widehat{W}_{\mathcal{H}}(\theta)$ and $\widehat{S}_{\mathcal{H}}(\theta)$.
Thus the estimator $\theta^{*}$ minimizes the empirical $\widehat{W}_{\mathcal{H}}(\theta)$-distance
or $\widehat{S}_{\mathcal{H}}(\theta)$, the square of $\widehat{W}_{\mathcal{H}}(\theta)$-distance.
By minimizing the distance function, the statistic $\widehat{S}_{\mathcal{H}}(\theta^{*})$
gives a regular consistent solution to the dual problem in \eqref{eq:empiricalMMD}:
\begin{align*}
\theta^{*} & =\arg\min_{\theta\in\Theta}\widehat{W}(\theta):=\arg\min_{\theta\in\Theta}\sup_{\|f\|_{l}\leq1}\left|\frac{1}{m}\sum_{i=1}^{m}f(q_{i})-\frac{1}{n}\sum_{j=1}^{n}f(q_{j}^{'})\right|,\\
 & =\arg\min_{\theta\in\Theta}\widehat{S}_{\mathcal{H}}(\theta):=\arg\min_{\theta\in\Theta}\left[\frac{1}{m^{2}}\sum_{i\neq j}^{m}\mathbf{k}(x_{i},x_{j})-\frac{2}{mn}\sum_{i,j=1}^{m,n}\mathbf{k}(x_{i},\varepsilon_{j})+\frac{1}{n^{2}}\sum_{i\neq j}^{n}\mathbf{k}(\varepsilon_{i},\varepsilon_{j})\right],\\
\mbox{s.t.} & \quad(x_{i},\varepsilon_{i})=q_{i}\sim H_{n\theta},\qquad(x_{j},\varepsilon_{j})=q_{j}^{'}\sim P_{n\theta}.
\end{align*}

The simplest form of implementing this estimate is to use the grid
search. One can evaluate $\theta$ at different locations and then
compute $\widehat{S}_{\mathcal{H}}(\theta)$ thorough the kernel representation.
This idea is feasible for one dimension case. However, due to the
fact that the simultaneous equation problem in \eqref{eq:empiricalMMD}
has at least two variables $(x_{i},\varepsilon_{i})$, grid search
is computational demanding.

Based on the duality, we propose a kernel-based Kantorovitch formulation.
The formulation will induce a Linear Programming (LP) problem for
the dual representation of 
\[
C_{i,j}:=\left[\mathbf{k}(x_{i},x_{j})-\mathbf{k}(x_{i},\varepsilon_{j})\right]-\left[\mathbf{k}(x_{i},\varepsilon_{j})-\mathbf{k}(\varepsilon_{i},\varepsilon_{j})\right]\geq0.
\]
Let $\gamma_{i,j}(\theta)$ be the empirical counterpart of the transportation
measure $\gamma\in\Gamma(H_{\theta},P_{\theta})$ in \eqref{eq:Kantonovich}.
Then $dH_{n\theta}=\sum_{j=1}^{m}\gamma_{i,j}(\theta)$ and $dP_{n\theta}=\sum_{i=1}^{n}\gamma_{i,j}(\theta)$.
The kernel-based Kantorovitch formulation of \eqref{eq:Kantonovich}
is 
\begin{equation}
\gamma_{i,j}(\theta^{*})\in\arg\min_{\gamma_{i,j}(\theta)\in\Gamma(H_{n\theta},P_{n\theta})}\sum_{i,j}\gamma_{i,j}(\theta)C_{i,j}.\label{eq:LP-dual}
\end{equation}
Note that \eqref{eq:LP-dual} is a standard LP problem. Every LP has
a dual. The neatest way to write the dual of \eqref{eq:LP-dual} is:
\begin{equation}
\max_{u_{i},v_{j}}\sum_{i}u_{i}+\sum_{j}v_{j}\quad\mbox{s.t. }u_{i}+v_{j}\geq C_{i,j}\label{eq:LP-dual-dual}
\end{equation}
where $u_{i}$ and $v_{j}$ are the Lagrange multipliers of the constraints
\[
dH_{n\theta}(x_{i},\varepsilon_{i})=\sum_{j=1}^{m}\gamma_{i,j}(\theta),\qquad dP_{n\theta}(x_{j},\varepsilon_{j})=\sum_{i=1}^{n}\gamma_{i,j}(\theta),
\]
where the measures are sums of Diracs.

The dual problem \eqref{eq:LP-dual-dual} can be solved by the simplex
method which is feasible for two or higher dimensional problems. It
means that we could set up a simultaneous equation problem and solve
it using the simplex method as with any LP problem. The duality between
\eqref{eq:LP-dual-dual} and \eqref{eq:LP-dual} is given by the following
corollary. \begin{cor} \label{cor:ImpDual}If \eqref{eq:LP-dual-dual}
has a solution $(u,v)$ such that $u_{i}+v_{j}=C_{i,j}$, then this
solution induces a solution $\gamma_{i,j}(\theta^{*})$ in \eqref{eq:LP-dual},
and vice versa.\end{cor} \begin{proof} See Appendix \ref{sub:Proof-of-Corollary Dual}.
\end{proof} The optimal $\widehat{S}_{\mathcal{H}}(\theta^{*})$-distance
is defined as 
\begin{align*}
 & \widehat{S}_{\mathcal{H}}(\theta^{*})=\sum_{i,j}\gamma_{i,j}(\theta^{*})C_{i,j}\\
\mbox{s.t.}\qquad & dH_{n\theta^{*}}(x_{i},\varepsilon_{i})=\sum_{j=1}^{m}\gamma_{i,j}(\theta^{*})\mbox{ for any }i,\\
 & dP_{n\theta^{*}}(x_{j},\varepsilon_{j})=\sum_{i=1}^{n}\gamma_{i,j}(\theta^{*})\mbox{ for any }j.
\end{align*}
The scheme of computing $\theta^{*}$ is summarized below. 
\begin{enumerate}
\item[] Step 1. Initialization. 
\item[] Step 2. Generate $H_{n\theta}(x_{i},\varepsilon_{i})$
and $P_{n\theta}(x_{j},\varepsilon_{j})$ for given $\theta$. Solve
problem \eqref{eq:LP-dual} via its dual \eqref{eq:LP-dual-dual}.
Obtain $\theta^{*}$. 
\item[] Step 3. Compute $\widehat{S}_{\mathcal{H}}(\theta^{*})$
by the kernel representation. If $\widehat{S}_{\mathcal{H}}(\theta^{*})\neq0$,
then go to Step 2, otherwise done. 
\end{enumerate}
Sampling for $H_{n\theta}(x_{i},\varepsilon_{i})$
and $P_{n\theta}(x_{j},\varepsilon_{j})$ in Step 2 has several options. One can draw the samples $\varepsilon(\theta)$ from $\rho(z,\theta)$ for given $z$ and $\theta$, then obtain the empirical c.d.f.   $ n^{-2} \sum_{i=1}^{n}\sum_{j=1}^{n}\mathbf{1}\{x_{i}\leq x,\rho(z_{j},\theta)\leq\varepsilon\}$  for  $H_{n\theta}$ and $ n^{-2} \sum_{i=1}^{n}\mathbf{1}\{x_{i}\leq x\} \sum_{i=1}^{n}\mathbf{1}\{\rho(z_{i} \theta)\leq\varepsilon\}$ for $P_{n\theta}$. On the other hand, one can resampling $z$ in order to generate more observations for $\varepsilon(\theta)$  from $\rho(z,\theta)$. With a larger (generic) sample size of $z$ and $\varepsilon(\theta)$, one can obtain a smoother empirical c.d.f.. 

Computation for $\widehat{S}_{\mathcal{H}}(\theta^{*})$ in Step 3 needs a LP solver. We use Dikin\rq{}s method in this solver. To outline the routine in Step 3, first we construct the kernel representation such that $\mathbf{y}=\alpha^{T} \Phi(\mathbf{x})$ where $\mathbf{k}(\mathbf{x},\mathbf{x})= \left\langle \Phi(\mathbf{x}),\Phi(\mathbf{x})\right\rangle$, then we use $\mathbf{k}$ to obtain $C_{i,j}$, finally we solve the LP given in \eqref{eq:LP-dual-dual}.

\subsection{Numerical Experiment}
The following numerical
experiment illustrates more details about the computation. We generate
a random vector $\mathbf{q}=(\mathbf{x},\varepsilon)$ which is jointly
distributed as $H_{n\theta}$. The product of marginal empirical distribution
of $\mathbf{x}$ and $\varepsilon$ is $P_{n\theta}$, another empirical
c.d.f.. Figure \ref{fig:Point-clouds-initial} shows two point clouds.
The squares explain the distribution of $H_{n\theta}$ and the dots
explain the distribution of $P_{n\theta_{0}}$ at the value $\theta=\theta_{0}$.
When $\theta=\theta_{0}=1$, the elements $\mathbf{x}$ and $\varepsilon$
are distributed independently as two smaller point clouds. The size
of each dot (square) is proportional to its probability weight.

The specification of this computation is given as follows. The samples $q_{i} = (x_{1,i}, x_{2,i},$  $ \varepsilon_{1,i,\theta}, \varepsilon_{2, i,\theta})_{i\leq 140}$ are jointly generated by a bivariate Gaussian distribution with means $(0, -0.5)$ and $(0, 0.25+ \theta)$ for $\mathbf{x}$ and $\varepsilon$ respectively. All the variances are set to $1$. The correlation between $\mathbf{x}_{2}$ and $\varepsilon_{2}(\theta)$ is $1 - \theta$ while zero for the others. When $\theta_{0}=1$, $\mathbf{x}$ and $\varepsilon(\theta_{0})$ are independent since the correlation value is zero. For $q\rq{}_{i}$,  $(x_{1,i}, x_{2,i})_{i\leq 150}$ and $(\varepsilon_{1,i,\theta}, \varepsilon_{2, i,\theta})_{i\leq 150}$, the samples are drawn independently from two bivariate Gaussian distributions. For $\mathbf{x}$, the mean and variance are $(0,-0.5)$ and $1$. For $\varepsilon(\theta)$, the mean and variance are $(0, 0.25+\theta) $ and $1$. Obviously, these two groups of samples are equivalent under the true $\theta_{0}$. 

\begin{figure}
\begin{centering}
\includegraphics[clip,scale=0.8]{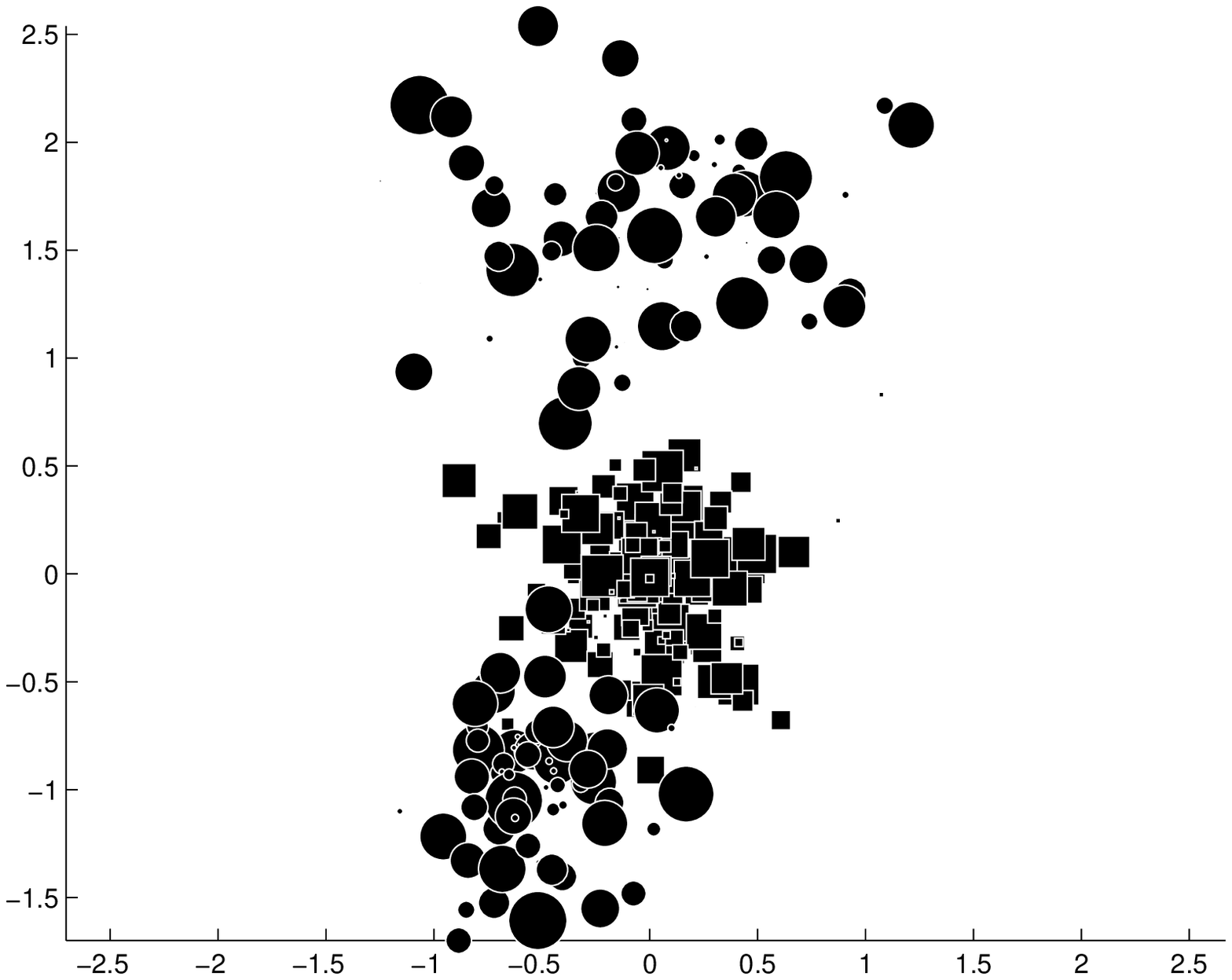} 
\par\end{centering}

\caption{\label{fig:Point-clouds-initial}Point clouds of $\mathbf{q}\sim H_{n\theta}$
and $\mathbf{q}'\sim P_{n\theta_{0}}$.}
\end{figure}

The computation of $\theta$ can be considered as pushing $\mathbf{q}\sim H_{n\theta}$
(or $\mathbf{q}'\sim P_{n\theta}$) towards a new $\mathbf{q}\sim H_{n\theta_{0}}$
(or $\mathbf{q}'\sim P_{n\theta_{0}}$). Under the independence condition,
$\mathbf{q}$ is equivalently distributed as $\mathbf{q}'$ for $\theta=\theta_{0}$.
Figure \ref{fig:Point-clouds-Transform} gives an illustration about
the transformation procedure. The initial points $\mathbf{q}'\sim P_{n\theta}$
at $\theta=0$ are pushed towards to $\mathbf{q}'\sim P_{n\theta_{0}}$
at $\theta=\theta_{0}=1$. The mechanism behind this transformation
relies on the transportation measure $\gamma_{i,j}(\theta)$ from
the LP problem \eqref{eq:LP-dual}. By solving the LP problem, $\theta^{*}$
is obtained when $\gamma_{i,j}(\theta^{*})$ is achieved. Figure \ref{fig:Transition-Matrix}
shows the matrix plot of the optimal $\gamma_{i,j}(\theta^{*})$.

\begin{figure}
\begin{centering}
\includegraphics[clip,scale=0.8]{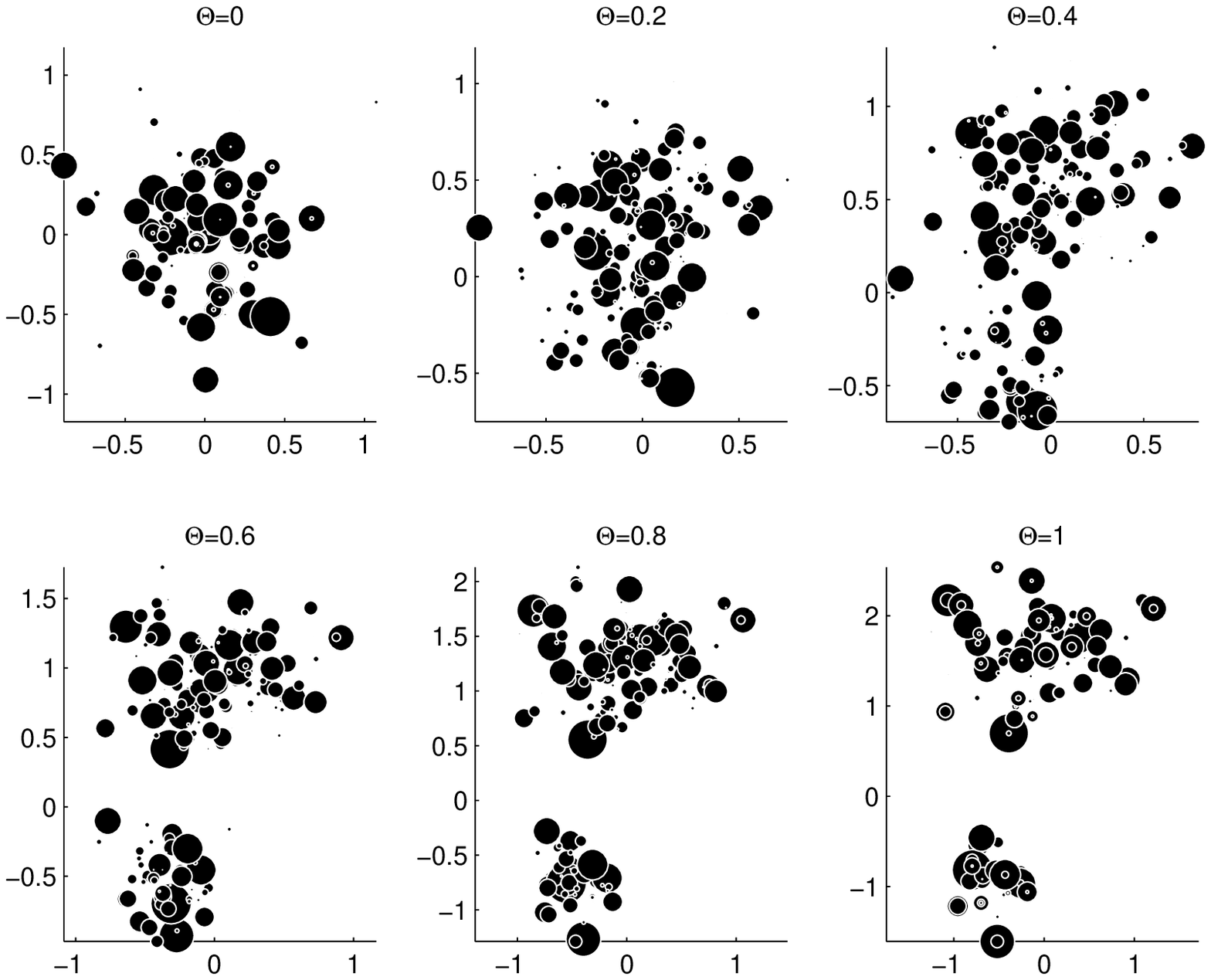} 
\par\end{centering}

\caption{\label{fig:Point-clouds-Transform}Point clouds of $\mathbf{q}'\sim P_{n\theta}$
are pushed towards to $\mathbf{q}'\sim P_{n\theta_{0}}$.}
\end{figure}

In the experiment, we consider the case $n=140$, $m=150$. So $\gamma_{i,j}(\theta)$
is contained in a $150\times140$ matrix. Figure \ref{fig:Transition-Matrix}
shows that for optimal $\gamma_{i,j}(\theta^{*})$ most of the transitions
betweens $q_{i}$ and $q_{j}^{'}$ are never happen, namely with probability
zero. The transition matrix is very sparse and has few non-zero entities.
We also compute the marginal probabilities by summing up $\gamma_{i,j}(\theta^{*})$
w.r.t. $j$ and $i$ respectively: 
\[
dH_{n\theta^{*}}(x_{i},\varepsilon_{i})=\sum_{j=1}^{m}\gamma_{i,j}(\theta^{*}),\qquad dP_{n\theta^{*}}(x_{j},\varepsilon_{j})=\sum_{i=1}^{n}\gamma_{i,j}(\theta^{*}).
\]
The smoothed density plots in Figure \ref{fig:Marginal-Distributions-Gamma}
shows that neither $dH_{n\theta^{*}}(x_{i},\varepsilon_{i})$ nor
$dP_{n\theta^{*}}(x_{i},\varepsilon_{i})$ is close to uniform distribution.
Also, the pattern of the transition matrix in Figure \ref{fig:Transition-Matrix}
shows that the optimal $\gamma_{i,j}(\theta^{*})$ is not uniformly
distributed. In fact, in Figure \ref{fig:Marginal-Distributions-Gamma},
the marginal densities seem have mixture patterns and have very irregular
shapes. This feature is important. It means that in a simultaneous
equation estimation problem: 
\[
\int\mathbf{d}(H_{n\theta}(x,\varepsilon),P_{n\theta}(x,\varepsilon))d\mu(x,\varepsilon),
\]
the optimal choice of $\mu$ could be excluded from any regular specified
distribution, i.e. uniform, normal, etc. Instead of giving a priori
specification of $\mu$, one can endogenize the effect of $\mu$ as
a transportation measure. By computing the new estimator, one will
obtain $\theta^{*}$ and an optimal measure $\gamma_{i,j}(\theta^{*})$
for $\mu$ simultaneously.

\begin{figure}
\begin{centering}
\includegraphics[clip,scale=0.8]{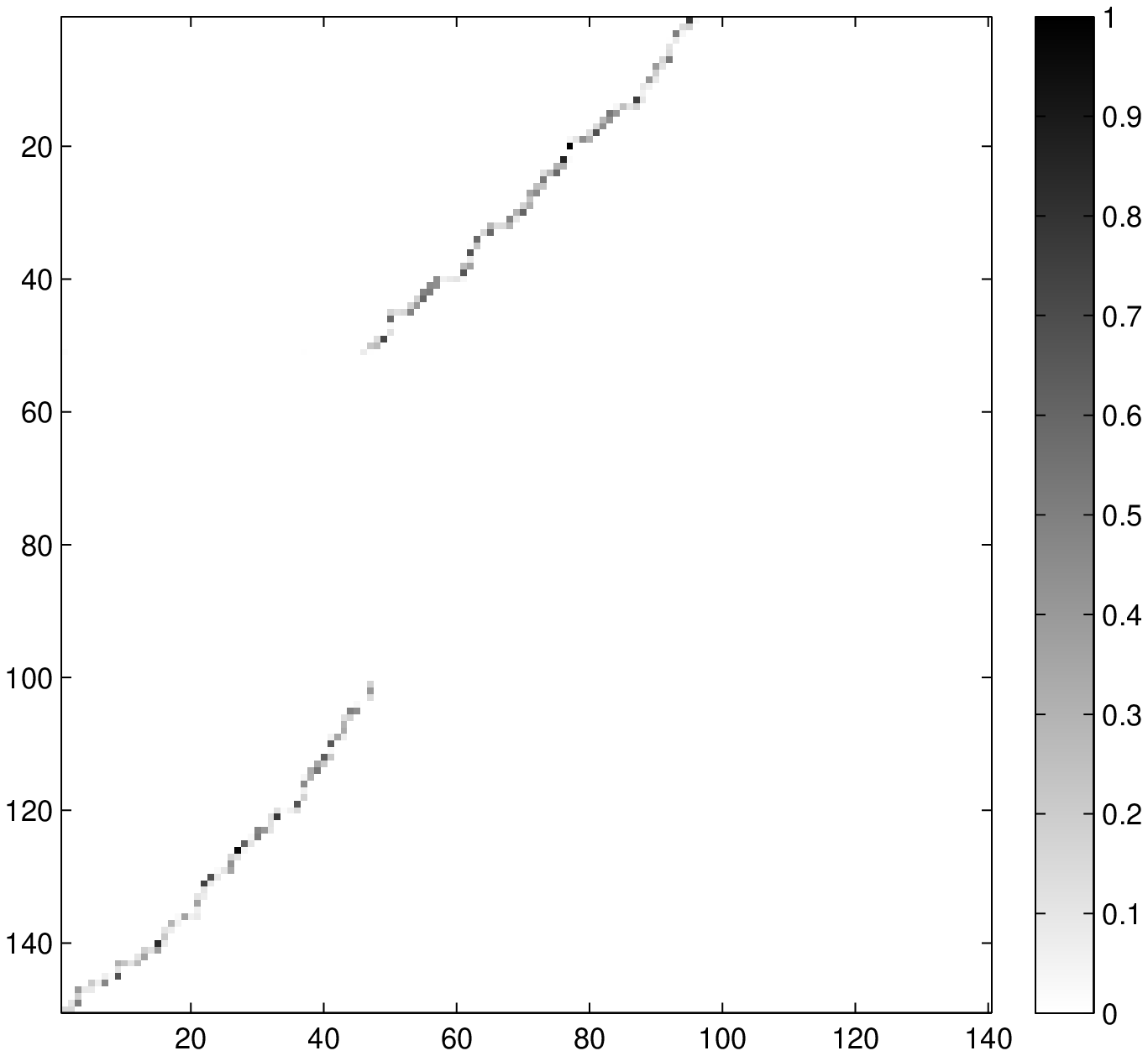} 
\par\end{centering}

\caption{\label{fig:Transition-Matrix} A matrix plot of $\gamma_{i,j}(\theta^{*})$.}
\end{figure}

\begin{figure}
\begin{centering}
\includegraphics[clip,scale=0.8]{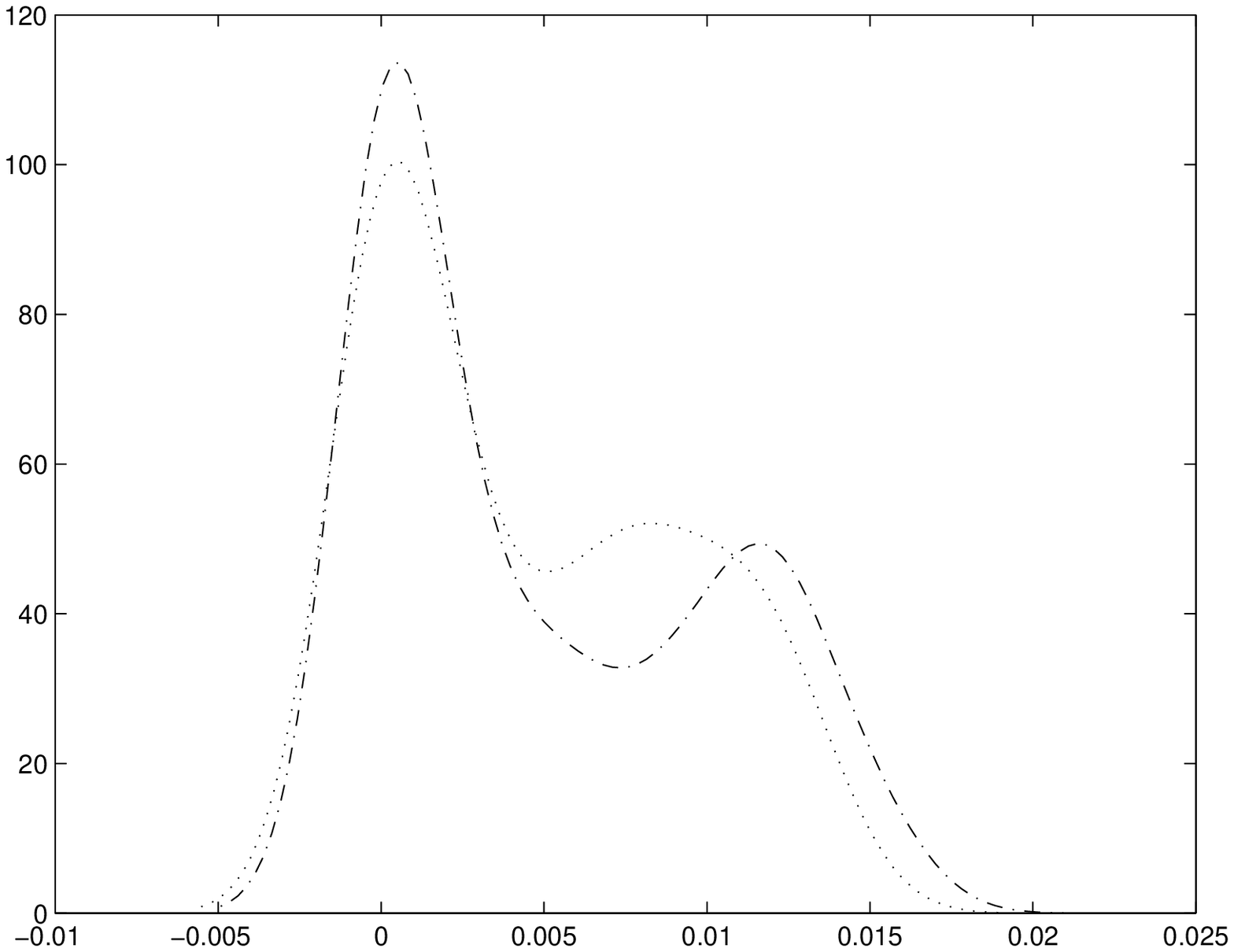} 
\par\end{centering}

\caption{\label{fig:Marginal-Distributions-Gamma} Density plots of $dH_{n\theta^{*}}$
and $dP_{n\theta^{*}}$.}
\end{figure}

Although the role of $\gamma_{i,j}(\theta^{*})$ is a statistical
inference mechanism, its meaning is far more than that. In economic
theory, $\gamma_{i,j}(\theta^{*})$ can induce an optimal transportation
in the matching problem. The optimal inference is then an optimal
matching mechanism. Nonparametric representation of the optimal distance
is applicable for solving some general matching problems. Figure \ref{fig:Optimal-Plan}
shows the role of $\gamma_{i,j}(\theta^{*})$ as pushing the initial
points (square) towards to the posterior points (dots). The lines
in Figure \ref{fig:Optimal-Plan} represent the connections. A line
connecting the square $(x_{i},\varepsilon_{i})$ and the dot $(x_{j}^{'},\varepsilon_{j}^{'})$
means that $\gamma_{i,j}(\theta^{*})C_{i,j}$ for these two groups
is non-zero.

\begin{figure}
\begin{centering}
\includegraphics[clip,scale=0.8]{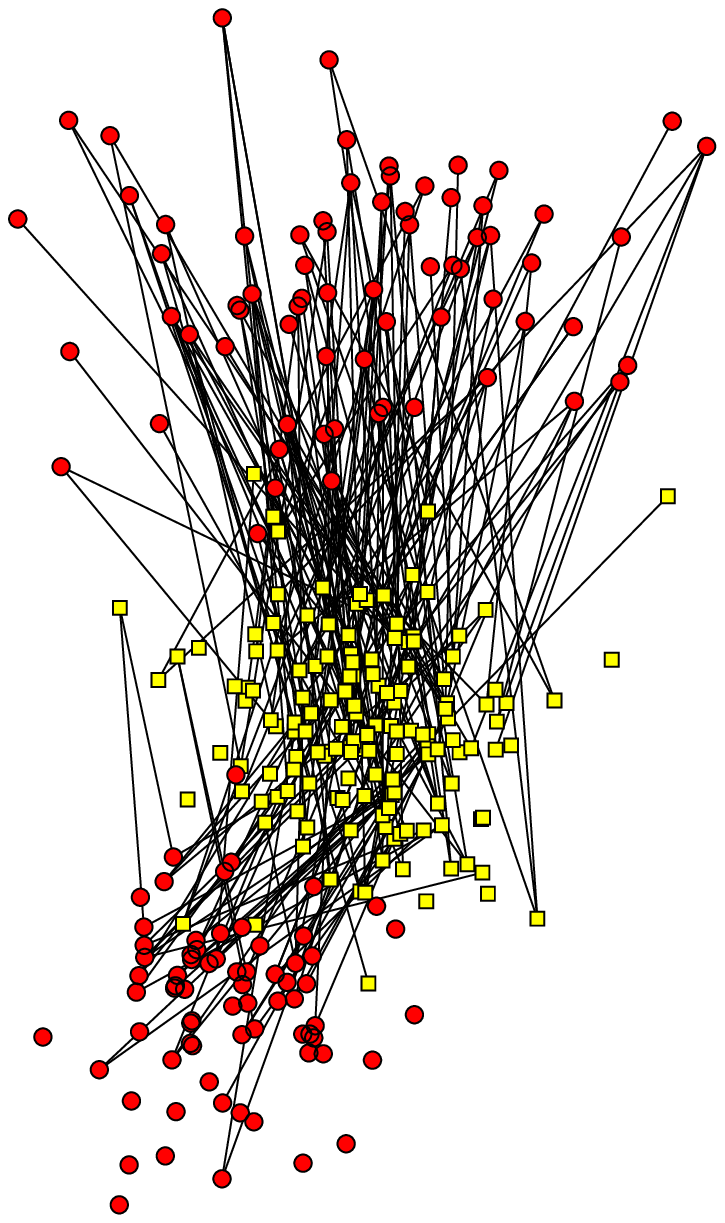} 
\par\end{centering}

\caption{\label{fig:Optimal-Plan} Measure $\gamma_{i,j}(\theta^{*})$ as an
optimal action.}
\end{figure}

\section{Related Likelihood-based Inference\label{sec:Related-Inferential-Methods}}

If the model is fully specified, then the c.d.f.s $H_{\theta}(x,\varepsilon)$
and $P_{\theta}(x,\varepsilon)$ in \eqref{eq:Kantonovich} are some
parametric distributions. The estimation problem becomes: 
\[
\min_{\theta}W(\theta):=\min_{\theta}\inf_{\gamma\in\Gamma(H_{\theta},P_{\theta})}\left\{ \int\mathbf{d}(q,q')d\gamma(q,q'):\: q\sim H_{\theta}(x,\varepsilon)\mbox{ and }q'\sim P_{\theta}(x,\varepsilon)\right\} .
\]
Unlike the empirical optimal distance function $\widehat{W}(\theta)$,
the integrable $W(\theta)$ can have densities for $H_{\theta}(x,\varepsilon)$
and $P_{\theta}(x,\varepsilon)$. We define $p_{\theta}(q')=dP_{\theta}$
to be the density of $\theta$ with respect to some Lebesgue measure.
In parametric frameworks, we can assume that such a density always
exists. Similarly, the density of $H_{\theta}$ also exists and is
denoted by $h_{\theta}(q)$. The conditional density defines the likelihood
function $L(\theta;\mathbf{q}')=\Pr(\mathbf{q}'|\theta)$.

The essence of our approach is to find a map that pushes forward $H_{\theta}(x,\varepsilon)$
to $P_{\theta}(x,\varepsilon)$. Recall that $q$ is a random variable
distributed according to $H_{\theta}(x,\varepsilon)$ and that $q'$
is a random variable distributed according to $P_{\theta}(x,\varepsilon)$.
Then we seek a map $f$ that satisfies the following specification:
\[
q'=f(q),\quad\mbox{where }q\sim H_{\theta}(x,\varepsilon)\mbox{ and }q'\sim P_{\theta}(x,\varepsilon).
\]
Equivalently, we seek a map $f$ which pushes forward $H_{\theta}(x,\varepsilon)$
to $P_{\theta}(x,\varepsilon)$. The map necessarily depends on the
data and the form of the likelihood function. If the independence
condition is satisfied, an identity map would suffice; otherwise more
complicated functions are necessary.

By using the conditional updating argument (Bayes' rule), the posterior
is 
\[
\tilde{p}_{\theta}(q')=\frac{L(\theta;\mathbf{q}')p_{\theta}(f(q))}{\int L(\theta;\mathbf{q}')p_{\theta}(f(q))d\theta}
\]
where $p_{\theta}(f(q))$ emphasizes that $q'=f(q)$. Recall that
the transformed density $p_{\theta}(f(q))$ should be equal to $h_{\theta}(q)$
at $\theta=\theta_{0}$ such that $p_{\theta}(f(q))=h_{\theta}(q)$.
On the other hand, the optimal transportation argument says that $h_{\theta}(q)$
can be transformed into the probability density $p_{\theta}(f(q))$
by the forward push function $f(\cdot)$. Thus we can setup the following
argument at $\theta=\theta_{0}$: 
\begin{align*}
h(\theta) & =L(\theta;\mathbf{q}')p_{\theta}(f(q))/c
\end{align*}
where $c=\int L(\theta;\mathbf{q}')p_{\theta}(q')d\theta$ is a constant
for any $\theta$. This argument defines an alternate criterion for
$W(\theta)$.

At $\theta=\theta_{0}$, $p_{\theta}(f(q))=h_{\theta}(q)$ for any
$q$. This condition means that the ratio $p_{\theta}(f(q))/h_{\theta}(q)$
is one or equivalently $\log p_{\theta}(f(q))-\log h_{\theta}(q)=0$
when $\theta=\theta_{0}$. Thus, we can consider the following criterion
\[
\min_{\theta,f}\sum_{i=1}^{n}\left\{ \log h_{\theta}(q_{i})-\log L(\theta;q_{i}^{'})-\log p_{\theta}(f(q_{i}))\right\} .
\]
Although this criterion function has a simpler representation, it
may not as tractable as $W(\theta)$ in practice due to the unknown
function $f$. The transportation function $f$ belongs to a unknown
functional class. Without further restriction of this class, it is
infeasible to select the optimal $f$ in the above argument. The following
example is to show that with a Gaussian restriction, the push forward
function $f$ is attainable.

\paragraph{Example: The Linear Gaussian Case}

Suppose that the discrepancy between $\mathbf{q}'$ and $\mathbf{q}$
is additive: $\mathbf{q}'=A(\theta)+\mathbf{q}$. In this case, the
map would thus depend on the data $\mathbf{q}'$, the forward model
$A(\cdot)$, and the distribution of the observational $\mathbf{q}$.

A common further assumption is that $\mathbf{q}$ is Gaussian, zero-mean,
and independent of the model parameters, i.e., $\mathbf{q}\sim\mathcal{N}(0,\Sigma_{n})$,
leading to the following form for the likelihood function: 
\[
L(\theta;\mathbf{q}')=\frac{1}{(2\pi)^{\frac{m}{2}}(\mbox{det}\Sigma)^{\frac{1}{2}}}\exp\left(-\frac{1}{2}\left\Vert A(\theta)-\mathbf{q}'\right\Vert _{\Sigma_{n}}^{2}\right)
\]
where $\|u\|_{\Sigma}:=\|\Sigma^{-\frac{1}{2}}u\|$, for any positive
symmetric matrix $\Sigma$.

Let $A(\theta)=A\theta$. Because $\mathbf{q}$ is distributed as
a Gaussian variable and $A\theta$ is an additive term, $\mathbf{q}'$
has the Gaussian distribution as well. Thus $p_{\theta}(\mathbf{q}')$
is from Gaussian family. Without loss of generality, we can set $\mathbf{q}'-\mathbf{q}\sim\mathcal{N}(0,\Sigma_{M})$.
Then the criterion function becomes: 
\begin{align*}
 & \sum_{i=1}^{n}\left\{ \log h_{\theta}(q_{i})-\log L(\theta;q_{i}^{'})-\log p_{\theta}(f(q_{i}))\right\} \\
\propto & \left\{ -\frac{1}{2}\left\Vert A\theta-\mathbf{q}'\right\Vert _{\Sigma_{n}}^{2}+\left\Vert \theta\right\Vert _{\Sigma_{M}}^{2}\right\} \\
\propto & \left\{ -\frac{1}{2}\left(\|\theta\|_{\Sigma_{n}}^{2}+\theta^{T}(A^{T}\Sigma_{n}^{-1}A+\Sigma_{M}^{-1})\theta-2(\mathbf{q}^{'})^{T}\Sigma_{n}^{-1}A\theta\right)\right\} \\
\propto & \left\{ -\frac{1}{2}\left(\|\theta\|_{\Sigma_{n}}^{2}+\theta^{T}\Sigma\theta-2\mu^{T}\Sigma^{-1}\theta\right)\right\} 
\end{align*}
where $\Sigma=(A^{T}\Sigma_{n}^{-1}A+\Sigma_{M}^{-1})^{-1}$ is the
updated covariance matrix and $\mu=\Sigma A^{T}\Sigma_{n}^{-1}\mathbf{q}$'
is the updated mean. Then the optimal transportation mechanism is equivalent to
the optimal information updating for a posterior.

\section{Conclusion\label{sec:Conclusion}}

In this paper we considered estimation and inference in
semiparametric simultaneous equation models where the observable variables 
are separable of the unobserved errors. We consider
estimation using a model based on the independence condition that
induces a minimum distance between a joint cumulative
distribution function and the product of its marginal distribution
functions. We developed a new minimum integrated distance estimator which generalizes BW approach
by using Kantorovich's formulation of the optimal transportation problem.
This generalization is important because it allows for estimating
the weighted measure nonparametrically, and hence does not impose a priori structure on the weighting measure. 
The new estimator endogenizes the measure and 
estimates the measure and the parameters of interest simultaneously while making minimal assumptions
regarding the measure. It also provides an automatic optimal measure selection. 
Moreover, the proposed methods do not require knowledge of
the reduced form of the model. The estimation provides
greater generality by dealing with probability measures on compact
metric spaces without assuming existence of densities. 
We establish the asymptotic properties of the estimator, and show that the asymptotic statistics of empirical estimates have standard convergent results and are provided for different statistical analyses. The proposed estimator is more efficient than its competitors. In addition, we provided a tractable implementation for computing the estimator in practice.

%Return to the questions proposed in Problem (M) in Section \ref{sub:Model-Setup}.
%We reply the questions in their corresponding order.

%Any continuous bounded function evaluated at $H_{\theta}(x,\varepsilon)$
%and $P_{\theta}(x,\varepsilon)$ will construct two different realized
%sequences. We use a general Wasserstein criterion \eqref{eq:Kantonovich}
%to measure the dis-similarity between these two sequences. The criterion
%captures infinite intrinsic connections between $H_{\theta}(x,\varepsilon)$
%and $P_{\theta}(x,\varepsilon)$ via a transition kernel $\mu(x,\varepsilon)$.
%The transition kernel in Wasserstein criterion \eqref{eq:Kantonovich}
%is the key mechanism to provide an optimal criterion function of the
%minimum distance problem over two probability measures. By Kantonovich's
%dual \eqref{eq:empiricalMMD} and kernel representation theorem \eqref{eq:WH_distance},
%we embed this infinite dimensional criterion function into a tractable
%space. Then the criterion function is represented in terms of a kernel-based
%distance. Since the new distance is equivalent to the total variation
%distance, the modified estimation problem maintains more topological
%details and thus it gives a more efficient estimator.

\newpage{}

 \bibliographystyle{econometrica}

% MinDistSEM.bib is the name of our database

\newpage{}

\appendix
\let\normalsize\small 
\appendix 
\small

\section{Mathematical Preliminary\label{sub:Mathematical-Preliminary}}

\begin{defn*} \emph{(Commutative algebra)} A \emph{commutative algebra}
$\mathcal{A}$ is a vector space equipped with an additional associative
$+:\mathcal{A}\times\mathcal{A}\rightarrow\mathcal{A}$ and commutative
multiplication $\cdot\,:\mathcal{A}\times\mathcal{A}\rightarrow\mathcal{A}$
such that 
\[
x\cdot(y+z)=x\cdot y+x\cdot z,\quad\lambda(x\cdot y)=(\lambda x)\cdot y
\]
holds for all $x,$ $y$, $z\in\mathcal{A}$ and $\lambda\in\mathbb{R}$.
An important algebra is the space $C(\mathcal{X})$ of \emph{all continuous
functions} $f:\mathcal{X}\rightarrow\mathbb{R}$ on the compact metric
space $(\mathcal{X},\mathbf{d})$ endowed with the \emph{supremum
norm} 
\[
\|f\|_{\infty}:=\sup_{x\in\mathcal{X}}|f(x)|.
\]

\end{defn*} The following approximation theorem states that certain
sub-algebras of $C(\mathcal{X})$ generate the whole space. This is
the starting point of representing our dual distance $\widehat{W}(\theta)$.
\begin{thm} \label{thm:SWtheorem} (Stone-Weierstrass Theorem) Let $(\mathcal{X},d)$ be
a compact metric space and $\mathcal{A}\subset C(\mathcal{X})$ be
an algebra. Then $\mathcal{A}$ is dense in $C(\mathcal{X})$ if the
following conditions hold.

1. $\mathcal{A}$ does not vanish such that there is no $f(x)=0$
for any $x\in\mathcal{X}$.

2. $\mathcal{A}$ separates points such that for all $x,y\in\mathcal{X}$
with $x\neq y$ there exists an $f\in\mathcal{A}$ with $f(x)\neq f(y)$.\end{thm}
\begin{defn*} \emph{(Feature maps and Feature Spaces)} A \emph{kernel}
on $(\mathcal{X},d)$ is a function $\mathbf{k}:\mathcal{X}\times\mathcal{X}\rightarrow\mathbb{R}$.
There exists a Hilbert space $\mathcal{H}$ and a map $\Phi:\mathcal{X}\rightarrow\mathcal{H}$
with 
\[
\mathbf{k}(x,y)=\left\langle \Phi(x),\Phi(y)\right\rangle 
\]
for all $x,y\in\mathcal{X}$. $\Phi$ is called a \emph{feature map}
and $\mathcal{H}$ is called a \emph{feature space} of $\mathbf{k}$.
Note that both $\mathcal{H}$ and $\Phi$ are far from being unique.%
\footnote{However, for a given kernel, there exists a canonical feature space
which is called the reproducing kernel Hilbert space (RKHS). %
} A function $f:\mathcal{X}\rightarrow\mathbb{R}$ is induced by $\mathbf{k}$
if there exists an element $y\in\mathcal{H}$ such that $f=\left\langle y,\Phi(\cdot)\right\rangle $.
The definition of this inner product is independent of $\Phi$ and
$\mathcal{H}$. \end{defn*}

\begin{thm} \label{thm:Reproducing} (Moore-Aronszajn) Let $(\mathcal{X},\mathbf{d})$ be a compact metric space
and $\mathbf{k}$ is a kernel on $\mathcal{X}$. Suppose $\mathbf{k}$
induces a dense set of $C(\mathcal{X})$. Then for all compact and
mutually disjoint subsets $\mathcal{X}_{1},\dots,\mathcal{X}_{n}\subset\mathcal{X}$,
all $\alpha_{1},\dots,\alpha_{n}\in\mathbb{R}$ and any $\epsilon>0$,
there exists a function $g$ induced by $\mathbf{k}$ with $\|g\|_{\infty}\leq\max_{i}|\alpha_{i}|+\epsilon$
such that 
\[
\left\Vert g_{|\cup_{i=1}^{n}\mathcal{X}_{i}}-\sum_{i=1}^{n}\alpha_{i}\mathbf{1}\{\mathcal{X}_{i}\}\right\Vert _{\infty}\leq\epsilon,
\]
where $g_{|\mathcal{X}}$ denotes the restriction of $g$ to $\mathcal{X}$.
\end{thm}

\section{Probabilistic Preliminary\label{sub:Prob-Theo-Preliminary}}

\begin{defn*}
\emph{(Weak topology, weak convergence) Weak topology} on a space
of probability measures $\mathscr{P}:=\{\mathcal{P}:\,\mathcal{P}\mbox{ is a probability measure}\}$
is the weakest topology such that the map $\mathcal{P}\rightarrow\int_{\mathcal{Q}}fd\mathcal{P}$
is continuous for all $\mathcal{P}\in C_{b}(\mathcal{Q})$ where $C_{b}(\mathcal{Q})$
is the \emph{space of all bounded continuous functions}. A sequence
of $\mathcal{P}_{n}$ is said to \emph{converge weakly} to $\mathcal{P}$,
$\mathcal{P}_{n}\rightsquigarrow\mathcal{P}$, if and only if $\int_{\mathcal{Q}}fd\mathcal{P}_{n}\rightarrow\int_{\mathcal{Q}}fd\mathcal{P}$
for every $f\in C_{b}(\mathcal{Q})$. \end{defn*} A metric $\mathbf{d}(\cdot,\cdot)$
on $\mathscr{P}$ is said to \emph{metrize} the weak topology if the
topology induced by $\mathbf{d}(\cdot,\cdot)$ coincides with the
weak topology, which is defined as follows: If $\{\mathcal{P}_{1},\mathcal{P}_{2},\dots\}\in\mathscr{P}$
and 
\[
\mathcal{P}_{n}\rightsquigarrow\mathcal{P}\,\mbox{ is equivalent to }\mathbf{d}(\mathcal{P}_{n},\mathcal{P})\rightarrow0\;\mbox{as }n\rightarrow\infty,
\]
then the topology induced by $\mathbf{d}(\cdot,\cdot)$ coincides
with the weak topology. \begin{defn*} \emph{(Compact metric space})
We denote the \emph{closed ball} with radius $r$ and center $x$
by $B_{\mathbf{d}}(\mathcal{P},r):=\{\mathcal{P}'\in\mathscr{P}:\mathbf{d}(\mathcal{P},\mathcal{P}')\leq r\}$.
The \emph{covering numbers} of $\mathscr{P}$ are defined by 
\[
\mathfrak{N}((\mathscr{P},\mathbf{d}),r):=\min\left\{ n\in\mathbb{N}\cup\infty:\exists\mathcal{P}_{1},\dots,\mathcal{P}_{n}\text{ with }\mathscr{P}\subset\cup_{i=1}^{n}B_{\mathbf{d}}(\mathcal{P}_{i},r)\right\} 
\]
for all $r>0$. The space $(\mathscr{P},\mathbf{d})$ is \emph{pre-compact}
if and only if $ $$\mathfrak{N}((\mathscr{P},\mathbf{d}),r)$ is
\emph{finite} for all $r>0$. If the space $(\mathscr{P},\mathbf{d})$
is \emph{complete}, then $\mathscr{P}$ is \emph{compact} if and only
if $\mathscr{P}$ is \emph{pre-compact}. \end{defn*} Based on this
argument, throughout the paper, we only consider pre-compact space
$(\mathscr{P},\mathbf{d})$ over Hilbert spaces (complete spaces).
In other words, the covering numbers of $\mathscr{P}$ are finite
and distance $\mathbf{d}(\cdot,\cdot)$ metrizes $\mathscr{P}$ by
a Hilbertian type metric. The space $(\mathscr{P},\mathbf{d})$ is
a \emph{compact metric space}.

\section{Proof of Main Theorems}

\subsection{\label{sub:Proof-of-Theorem WHdistance}Proof of Theorem \ref{thm:KernelW-distance}}

\begin{proof} Consider the family 
\[
\mathscr{P}(\mathcal{Q}):=\left\{ \mathcal{P}:\int_{\mathcal{Q}}\sqrt{\mathbf{k}(q,q)}d\mathcal{P}(q)<\infty\right\} .
\]
Let $T:\mathcal{H}\rightarrow\mathbb{R}$ be a linear functional defined
as $ $$T[h]:=\int_{\mathcal{Q}}h(q)d\mathcal{P}(q)$ with 
\[
\|T\|=\sup_{h\in\mathcal{H},h\neq0}\frac{|T[h]|}{\|h\|_{\mathcal{H}}}.
\]
Now we have 
\begin{align*}
|T[h]|=\left|\int_{\mathcal{Q}}hd\mathcal{P}(q)\right|\leq\int_{\mathcal{Q}}|h|d\mathcal{P}(q) & =\int_{\mathcal{Q}}|\left\langle h,\mathbf{k}(\cdot,q)\right\rangle _{\mathcal{H}}|d\mathcal{P}(q)\\
 & \leq\int_{\mathcal{Q}}\sqrt{\mathbf{k}(q,q)}\|h\|_{\mathcal{H}}d\mathcal{P}(q).
\end{align*}
It implies $\|T\|<\infty$, for any $\mathcal{P}\in\mathscr{P}(\mathcal{Q})$.
Then we know $T$ is a bounded linear functional on $\mathcal{H}$.
By Riesz representation theorem \citep[Theorem II.4,][]{ReedSimon1980},
for each $\mathcal{P}\in\mathbb{P}$, there exists a unique $\lambda_{\mathcal{P}}\in\mathcal{H}$
such that $T[h]=\left\langle h,\lambda_{\mathcal{P}}\right\rangle _{\mathcal{H}}$,
for any $h\in\mathcal{H}$.

Let $h=\mathbf{k}(\cdot,u)\in\mathcal{B}_{\mathcal{H}}$ for some
$u\in\mathcal{Q}$. Then 
\[
T[\mathbf{k}(\cdot,u)]=\mathcal{P}\mathbf{k}=\left\langle \mathbf{k}(\cdot,u),\lambda_{\mathcal{P}}\right\rangle _{\mathcal{H}}=\lambda_{\mathcal{P}}(u).
\]
Replace $H_{\theta}(x,\varepsilon)$ and $P_{\theta}(x,\varepsilon)$
with arbitrary $\mathfrak{P}\in\mathcal{P}$ in the expression, then
we have 
\begin{align*}
\left|\int_{\mathcal{Q}}hdH_{\theta}(q)-\int_{\mathcal{Q}}hdP_{\theta}(q)\right|=\left|\left\langle h,\lambda_{H_{\theta}}\right\rangle _{\mathcal{H}}-\left\langle h,\lambda_{P_{\theta}}\right\rangle _{\mathcal{H}}\right| & =\left|\left\langle h,\lambda_{H_{\theta}}-\lambda_{P_{\theta}}\right\rangle _{\mathcal{H}}\right|.
\end{align*}
This implies $W_{\mathcal{H}}(\theta)$ define a norm 
\begin{align*}
W_{\mathcal{H}}(\theta):= & \sup_{\|h\|_{\mathcal{H}}\leq1}\left|\int_{\mathcal{Q}}hdH_{\theta}(q)-\int_{\mathcal{Q}}hdP_{\theta}(q)\right|\\
= & \sup_{\mathbf{k}\in\mathcal{B}_{\mathcal{H}}}\left|\left\langle \mathbf{k}(\cdot,u),\lambda_{H_{\theta}}\right\rangle _{\mathcal{H}}-\left\langle \mathbf{k}(\cdot,u),\lambda_{P_{\theta}}\right\rangle _{\mathcal{H}}\right|\\
= & \sup_{u}\left|\lambda_{H_{\theta}}(u)-\lambda_{P_{\theta}}(u)\right|=\left\Vert \lambda_{H_{\theta}}-\lambda_{P_{\theta}}\right\Vert _{\mathcal{H}}=\|H_{\theta}\mathbf{k}-P_{\theta}\mathbf{k}\|_{\mathcal{H}}
\end{align*}
The second equality uses the fact that any $\{h:\|h\|_{\mathcal{H}}\leq1\}$
can be represented by $\mathbf{k}\in\mathcal{B}_{\mathcal{H}}$ \end{proof}

\subsection{\label{sub:Proof-of-Theorem EquivMetric}Proof of Theorem \ref{Thm:EquivalentMetric}}

\begin{proof} When $(\mathcal{Q},\rho)$ is separable, we replace
the general distance function $\mathbf{d}$ in \eqref{eq:Kantonovich}
with $\mathbf{d}_{\mathbf{k}}$, then $W(\theta)$ has the following
form \citep[p.420,][]{Dudley2002} 
\[
W(\theta):=\inf_{\gamma\in\Gamma(H_{\theta},P_{\theta})}\left\{ \int\mathbf{d}_{\mathbf{k}}(q,q')d\gamma(q,q'):\: q\sim H_{\theta}(x,\varepsilon)\mbox{ and }q'\sim P_{\theta}(x,\varepsilon)\right\} .
\]
Condition C10 means that $(\mathcal{Q},\rho)$ is bounded.

Lower bound: For any $\gamma\in\Gamma(H_{\theta},P_{\theta})$, we
have 
\begin{align*}
\left|\int_{\mathcal{Q}}fd(H_{\theta}-P_{\theta})\right|= & \left|\int_{\mathcal{Q}\times\mathcal{Q}}\left(f(q)-f(q')\right)d\gamma(q,q')\right|\\
\leq & \int_{\mathcal{Q}\times\mathcal{Q}}\left|f(q)-f(q')\right|d\gamma(q,q')\\
= & \int_{\mathcal{Q}\times\mathcal{Q}}\left|\left\langle f,\,\mathbf{k}(\cdot,q)-\mathbf{k}(\cdot,q')\right\rangle _{\mathcal{H}}\right|d\gamma(q,q')\\
\leq & \|f\|_{\mathcal{H}}\int_{\mathcal{Q}\times\mathcal{Q}}\left\Vert \mathbf{k}(\cdot,q)-\mathbf{k}(\cdot,q')\right\Vert _{\mathcal{H}}d\gamma(q,q').
\end{align*}
The first equality uses the property of marginal probability of $\gamma(q,q')$.
After taking supremum over $f\in\mathcal{B}_{\mathcal{H}}$, we have
\[
W_{\mathcal{H}}(\theta)=\sup_{f\in\mathcal{B}_{\mathcal{H}}}\left|\int_{\mathcal{Q}}fd(H_{\theta}-P_{\theta})\right|\leq\int_{\mathcal{Q}\times\mathcal{Q}}\mathbf{d}_{\mathbf{k}}(q,q')d\gamma(q,q')
\]
for any $\gamma\in\Gamma(H_{\theta},P_{\theta})$. Thus after taking
the infimum over $\gamma$, we have $W_{\mathcal{H}}(\theta)\leq W(\theta)$.

Upper bound: $W(\theta)$ distance is bounded as follows: 
\begin{align*}
W(\theta)= & \inf_{\gamma\in\Gamma(H_{\theta},P_{\theta})}\int_{\mathcal{Q}\times\mathcal{Q}}\|\mathbf{k}(\cdot,q)-\mathbf{k}(\cdot,q')\|_{\mathcal{H}}d\gamma(q,q')\\
\leq & \int_{\mathcal{Q}\times\mathcal{Q}}\|\mathbf{k}(\cdot,q)-\mathbf{k}(\cdot,q')\|_{\mathcal{H}}dH_{\theta}(q)dP_{\theta}(q')\\
\overset{(1)}{\leq} & \left[\int_{\mathcal{Q}\times\mathcal{Q}}\|\mathbf{k}(\cdot,q)-\mathbf{k}(\cdot,q')\|_{\mathcal{H}}^{2}dH_{\theta}(q)dP_{\theta}(q')\right]^{\frac{1}{2}}\\
\overset{(2)}{\leq} & \left[\int_{\mathcal{Q}}\mathbf{k}(q,q)d(H_{\theta}+P_{\theta})(q)-2\int_{\mathcal{Q}\times\mathcal{Q}}\mathbf{k}(q,q')H_{\theta}(q)dP_{\theta}(q')\right]^{\frac{1}{2}}\\
\overset{(3)}{\leq} & \sqrt{2C+W_{\mathcal{H}}(\theta)}
\end{align*}
$\overset{(1)}{\leq}$ follows Jensen's inequality. $\overset{(2)}{\leq}$
follows 
\begin{align*}
 & \left[\int_{\mathcal{Q}\times\mathcal{Q}}\left\langle \mathbf{k}(\cdot,q)-\mathbf{k}(\cdot,q')\,,\,\mathbf{k}(\cdot,q)-\mathbf{k}(\cdot,q')\right\rangle _{\mathcal{H}}dH_{\theta}(q)dP_{\theta}(q')\right]^{\frac{1}{2}}\\
= & \left[\int_{\mathcal{Q}}\mathbf{k}(q,q)dH_{\theta}(q)+\int_{\mathcal{Q}}\mathbf{k}(q',q')dP_{\theta}(q')-2\int_{\mathcal{Q}\times\mathcal{Q}}\mathbf{k}(q,q')H_{\theta}(q)dP_{\theta}(q')\right]^{\frac{1}{2}}\\
\leq & \left[\int_{\mathcal{Q}}\mathbf{k}(q,q)d(H_{\theta}+P_{\theta})(q)-2\int_{\mathcal{Q}\times\mathcal{Q}}\mathbf{k}(q,q')H_{\theta}(q)dP_{\theta}(q')\right]^{\frac{1}{2}}.
\end{align*}
$\overset{(3)}{\leq}$ follows $ $ 
\begin{align*}
 & \int_{\mathcal{Q}}\mathbf{k}(q,q)d(H_{\theta}+P_{\theta})(q)-2\int_{\mathcal{Q}\times\mathcal{Q}}\mathbf{k}(q,q')dH_{\theta}(q)dP_{\theta}(q')\\
\leq & \int_{\mathcal{Q}}\mathbf{k}(q,q)d(H_{\theta}+P_{\theta})(q)+\left[\int_{\mathcal{Q}\times\mathcal{Q}}\mathbf{k}(q,q')(dH_{\theta}(q)-dP_{\theta}(q'))^{2}\right]\\
\leq & H_{\theta}\mathbf{k}+P_{\theta}\mathbf{k}+\left[\int_{\mathcal{Q}}\mathbf{k}(q,q)d(H_{\theta}-P_{\theta})(q)\right]^{2}\\
\leq & 2C+\left[\sup_{f\in\mathcal{B}_{\mathcal{H}}}\int_{\mathcal{Q}}fd(H_{\theta}-P_{\theta})\right]^{2}=2C+W_{\mathcal{H}}(\theta)
\end{align*}
where we set $C=\max\{\|H_{\theta}\mathbf{k}\|_{\mathcal{H}},\|P_{\theta}\mathbf{k}\|_{\mathcal{H}}\}$.
\end{proof}

\subsection{\label{sub:Proof-of-Theorem Bias}Proof of Theorem \ref{Thm:Bias}}

\begin{proof} In order to maintain simple expressions, without loss
of generality the kernel functions $\mathbf{k}(q,q_{i})$ is expressed
by $f(q_{i})=\mathbf{k}(q,q_{i})$ for any $q$.

The upper bound of $\widehat{W}_{\mathcal{H}}(\theta)-W_{\mathcal{H}}(\theta)$
follows 
\begin{align*}
\left|\widehat{W}_{\mathcal{H}}(\theta)-W_{\mathcal{H}}(\theta)\right|= & \left|\sup_{f\in\mathcal{B}_{\mathcal{H}}}H_{\theta}f-P_{\theta}f-\sup_{f\in\mathcal{B}_{\mathcal{H}}}\left(\frac{1}{m}\sum_{i=1}^{m}f(q_{i})-\frac{1}{n}\sum_{i=1}^{n}f(q_{i}^{'})\right)\right|\\
\leq & \underset{\Delta(H_{\theta},P_{\theta},\mathbf{q},\mathbf{q}')}{\underbrace{\sup_{f\in\mathcal{B}_{\mathcal{H}}}\left|H_{\theta}f-P_{\theta}f-\frac{1}{m}\sum_{i=1}^{m}f(q_{i})+\frac{1}{n}\sum_{i=1}^{n}f(q_{i}^{'})\right|}}.
\end{align*}
We can provide an upper bound on the difference between $\Delta(H_{\theta},P_{\theta},Q,Q')$
and its expectation. Changing either of $(x_{i},\varepsilon_{i})\sim H_{\theta}$
or $(x_{i}^{'},\varepsilon_{i}^{'})\sim P_{\theta}$ in $\Delta(H_{\theta},P_{\theta},Q,Q')$
results in changes in magnitude of at most $2C_{\mathbf{k}}^{1/2}/m$
or $2C_{\mathbf{k}}^{1/2}/n$, respectively. We can then apply Theorem
\ref{thm:McDiarmid's-inequality} 
\begin{equation}
\Pr\left\{ (\Delta(H_{\theta},P_{\theta},\mathbf{q},\mathbf{q}')-\mathbb{E}[\Delta(H_{\theta},P_{\theta},\mathbf{q},\mathbf{q}')])>\epsilon\right\} <\exp\left(-\frac{\epsilon^{2}mn}{2C_{\mathbf{k}}(m+n)}\right)\label{eq:bias-bounded-intermediate}
\end{equation}
where the denominator in the exponent comes from 
\[
m\left(2C_{\mathbf{k}}^{\frac{1}{2}}/m\right)^{2}+n\left(2C_{\mathbf{k}}^{\frac{1}{2}}/n\right)^{2}=4C_{\mathbf{k}}\left(\frac{1}{m}+\frac{1}{n}\right)=4C_{\mathbf{k}}\frac{n+m}{mn}.
\]

We then apply the symmetrization technique again to attain the upper
bound of $\Delta(H_{\theta},P_{\theta},\mathbf{q},\mathbf{q}')$.
The procedure called the ghost sample approach, i.e. a second set
of observations drawn from the same distribution, follows \citet[p.108,][]{vanderVaartWellner1996}.
Denote $\tilde{Q}\sim H_{\theta}$ an i.i.d sample of size $m$ drawn
independently of $Q$, similarly denote $\tilde{Q}\sim P_{\theta}$
in the same way. 
\begin{align*}
 & \mathbb{E}[\Delta(H_{\theta},P_{\theta},\mathbf{q},\mathbf{q}')]\\
= & \mathbb{E}\sup_{f\in\mathcal{B}_{\mathcal{H}}}\left|H_{\theta}f-P_{\theta}f-\frac{1}{m}\sum_{i=1}^{m}f(q_{i})+\frac{1}{n}\sum_{i=1}^{n}f(q_{i}^{'})\right|\\
\overset{(a)}{=} & \mathbb{E}\sup_{f\in\mathcal{B}_{\mathcal{H}}}\left|\mathbb{E}\left(\frac{1}{m}\sum_{i=1}^{m}f(\tilde{q}_{i})\right)-\mathbb{E}\left(\frac{1}{n}\sum_{i=1}^{n}f(\tilde{q}_{i}^{'})\right)-\frac{1}{m}\sum_{i=1}^{m}f(q_{i})+\frac{1}{n}\sum_{i=1}^{n}f(q_{i}^{'})\right|\\
\overset{(b)}{\leq} & \mathbb{E}\sup_{f\in\mathcal{B}_{\mathcal{H}}}\left|\frac{1}{m}\sum_{i=1}^{m}f(\tilde{q}_{i})-\frac{1}{n}\sum_{i=1}^{n}f(\tilde{q}_{i}^{'})-\frac{1}{m}\sum_{i=1}^{m}f(q_{i})+\frac{1}{n}\sum_{i=1}^{n}f(q_{i}^{'})\right|\\
\overset{(c)}{=} & \mathbb{E}\sup_{f\in\mathcal{B}_{\mathcal{H}}}\left|\frac{1}{m}\sum_{i=1}^{m}\sigma_{i}\left(f(\tilde{q}_{i})-f(q_{i})\right)+\frac{1}{n}\sum_{i=1}^{n}\sigma_{i}\left(f(\tilde{q}_{i}^{'})-f(q_{i}^{'})\right)\right|\\
\overset{(d)}{\leq} & \mathbb{E}\sup_{f\in\mathcal{B}_{\mathcal{H}}}\left|\frac{1}{m}\sum_{i=1}^{m}\sigma_{i}\left(f(\tilde{q}_{i})-f(q_{i})\right)\right|+\mathbb{E}\sup_{f\in\mathcal{B}_{\mathcal{H}}}\left|\frac{1}{n}\sum_{i=1}^{n}\sigma_{i}\left(f(\tilde{q}_{i}^{'})-f(q_{i}^{'})\right)\right|\\
\overset{(e)}{\leq} & 2\left[R_{m}(\mathcal{B}_{\mathcal{H}},Q,H_{\theta})+R_{n}(\mathcal{B}_{\mathcal{H}},Q',P_{\theta})\right]\overset{(f)}{\leq}2\left[(C_{\mathbf{k}}/m)^{\frac{1}{2}}+(C_{\mathbf{k}}/n)^{\frac{1}{2}}\right].
\end{align*}
$\overset{(a)}{=}$ follows the ghost sample argument, $\overset{(b)}{\leq}$
follows Jensen's inequality, $\overset{(c)}{=}$ follows the symmetrization
technique, $\overset{(d)}{\leq}$ follows the triangle inequality,
$\overset{(e)}{\leq}$ and $\overset{(f)}{\leq}$ follow Theorem \ref{thm:Rademacher-symmetrization}.

Substitute above result into \eqref{eq:bias-bounded-intermediate},
we have 
\[
\Pr\left\{ \Delta(H_{\theta},P_{\theta},\mathbf{q},\mathbf{q}')-2\left[(C_{\mathbf{k}}/m)^{\frac{1}{2}}+(C_{\mathbf{k}}/n)^{\frac{1}{2}}\right]>\epsilon\right\} \leq\exp\left(-\frac{\epsilon^{2}mn}{2C_{\mathbf{k}}(m+n)}\right).
\]
Substitute this bound into the upper bound of $\left|\widehat{W}_{\mathcal{H}}(\theta)-W_{\mathcal{H}}(\theta)\right|$,
then we get the result. \end{proof}

\subsection{\label{sub:Proof-of-Theorem-Unbias}Proof of Theorem \ref{thm:Unbias}}

\begin{proof} If $\Theta$ is compact, take $\mathcal{K}=\Theta$.
Apply the bound in \citet[p. 25][]{Hoeffding1963} to Theorem \ref{Thm:Bias}:
\[
\Pr\left\{ \widehat{S}_{\mathcal{H}}(\theta)-S_{\mathcal{H}}(\theta)>\epsilon\right\} \leq\exp\left(\frac{-\epsilon^{2}n}{8C_{\mathbf{k}}^{2}}\right).
\]
Note the $S_{\mathcal{H}}(\theta_{0})=0$. It follows that 
\[
\Pr\{\widehat{S}_{\mathcal{H}}(\theta^{*})\geq\epsilon\}=\Pr\left\{ n^{-2}\sum_{i\neq j}h(q_{i},q_{j})\geq\epsilon\right\} \leq\exp\left(\frac{-\epsilon^{2}n}{8C_{\mathbf{k}}^{2}}\right).
\]
Note that if $\epsilon\rightarrow0$ and $\epsilon^{2}n\rightarrow\infty$,
$\exp\left(-\epsilon^{2}n/8C_{\mathbf{k}}^{2}\right)\rightarrow0$.

If $\Theta$ is not compact, we do the following step. For fixed $\epsilon>0$,
take $B_{\theta_{0}}$ to be the open interval centered at $\theta_{0}$
of length $2\epsilon$. Since $\mathcal{K}\setminus B_{\theta_{0}}$
is a compact set, and $\{B_{\theta}:\theta\in\mathcal{K}\setminus B_{\theta_{0}}\}$
is an open cover, we may extract a finite sub-cover, say, $B_{\theta_{1}},\ldots,B_{\theta_{k}}$.
For notational simplicity, rename $\mathcal{K}^{c}$ and these sets
as $\Theta_{1},\ldots,\Theta_{m}$, so that $\Theta=B_{\theta_{0}}\cup(\bigcup_{s=1}^{m}\Theta_{s})$
and $h(\theta|\Theta_{s})>0$.

Write $c_{s}=h(\theta|\Theta_{s})$. Then by Theorem \ref{Thm:Bias},
at this $\theta$ 
\[
\frac{1}{n^{2}}\sum_{i\neq j}^{n}h(q_{i},q_{j})\to c_{s},\quad\text{with probability~1 for each \ensuremath{s}}.
\]
Let $A_{s}$ denote the set of data sequences for which this convergence
holds and set $A=\bigcap_{s=1}^{m}A_{s}$. Then $\Pr(A)=1$ and $n^{-2}\sum_{i\neq j}h(q_{i},q_{j})\to c_{s}>0$
for all $\mathbf{q}=(q_{1},q_{2},\ldots,)\in A$. If ``i.o.'' stands
for ``infinitely often,'' then we have: 
\begin{align*}
\bigl\{\mathbf{q}:\limsup_{n\to\infty}|\theta^{*}(q_{1},\ldots,q_{n})-\theta_{0}|\geq\epsilon\bigr\} & \subseteq\bigcup_{s=1}^{m}\bigl\{\mathbf{q}:\theta^{*}(q_{1},\ldots,q_{n})\in\Theta_{s}\text{ i.o.}\bigr\}\\
 & \subseteq\bigcup_{s=1}^{m}\Bigl\{\mathbf{q}:\inf_{\theta\in\Theta_{s}}\frac{1}{n}\sum_{i\neq j}h(q_{i},q_{j})>0\text{ i.o.}\Bigr\}\\
 & \subseteq\bigcup_{s=1}^{m}\Bigr\{\mathbf{q}:\hat{h}(\theta^{*}|\Theta_{s})>0\text{ i.o.}\Bigr\}\\
 & \subseteq\bigcup_{s=1}^{m}A_{s}^{c}.
\end{align*}
Since the last set is $A^{c}$ and $\Pr(A^{c})=0$, the result \eqref{eq:mle.conv}
follows. \end{proof}

\subsection{\label{sub:Proof-of-Theorem-WeakConvergence}Proof of Theorem \ref{thm:WeakConvergence}}

\begin{proof} The proof follows the procedure in \citet[Section 5.5.2,][]{Serfling1980}.
Under $\theta_{0}$, by the definition of $\tilde{\mathbf{k}}(x_{i},x_{j})$,
we have 
\begin{equation}
\begin{aligned}\mathbb{E}_{x}\tilde{\mathbf{k}}(x,x_{j}): & =\mathbb{E}_{x}\mathbf{k}(x,x_{j})-\mathbb{E}_{x}\left[\mathbb{E}_{x'}\mathbf{k}(x,x')\right]-\mathbb{E}_{x}\mathbf{k}(x,x_{j})+\mathbb{E}_{x,x'}\mathbf{k}(x,x')\\
 & =-\mathbb{E}_{x}\left[\mathbb{E}_{x'}\mathbf{k}(x,x')\right]+\mathbb{E}_{x,x'}\mathbf{k}(x,x')=0.
\end{aligned}
\label{eq:Proof-of-ChiSq-0}
\end{equation}
Step 1: To prove that the kernel $\tilde{\mathbf{k}}(x_{i},x_{j})$
function has the following form: 
\begin{equation}
\tilde{\mathbf{k}}(x_{i},x_{j})=\sum_{l=1}^{\infty}\lambda_{l}^{(1)}\psi_{l}(x_{i})\psi_{l}(x_{j}),\label{eq:Proof-ChiSquare-expansion}
\end{equation}
where $\lambda_{l}^{(1)}$ and $\psi_{l}$ are eigenvalues and basis
function of the eigenvalue equation 
\[
\int_{\mathcal{Q}}\tilde{\mathbf{k}}(x,x')\psi_{l}(x)dF(x)=\lambda_{l}^{(2)}\psi_{l}(q').
\]

Note that a Hilbert subspace endowed with an inner product over $\tilde{\mathbf{k}}(\cdot,\cdot)$
such that $\langle f(\cdot),\tilde{\mathbf{k}}(\cdot,x')\rangle=f(x')$
for all $f\in\mathcal{B}_{\mathcal{H}}$ and $\langle\tilde{\mathbf{k}}(\cdot,x),\tilde{\mathbf{k}}(\cdot,x')\rangle=\tilde{\mathbf{k}}(x,x')$.
Let $\psi_{l}$ be basis for $\mathcal{B}_{\mathcal{H}}$ and $B_{l}$
be i.i.d. $\mathcal{N}(0,\sigma_{l}^{2})$ with $\sigma_{l}\downarrow0$
as $l\rightarrow\infty$. We will construct Karhunen-Loeve expansion
in the following fashion. Suppose for any $x\in\mathcal{X}$, 
\[
f(x)=\sum_{l=1}^{\infty}B_{l}\psi_{l}(x).
\]
Note that $\mathbb{E}[f(x)]=0$ while the covariance kernel function
is 
\begin{align*}
\tilde{\mathbf{k}}(x,x')= & \mathbb{E}_{B}\left[\sum_{l=1}^{\infty}B_{l}\psi_{l}(x)B_{l}\psi_{l}(x')\right]\\
= & \sum_{l=1}^{\infty}\sigma_{l}^{2}\psi_{l}(x)\psi_{l}(x')
\end{align*}
The covariance function $\tilde{\mathbf{k}}$ of zero mean Gaussian
vectors completely characterizes the space $\mathcal{B}_{\mathcal{H}}$.

On the other hand, suppose we fix $q$. Then from Theorem \ref{Representation-Theorem},
for continuous function, $\tilde{\mathbf{k}}(x,\cdot)=\sum_{i=1}^{\infty}\alpha_{i}(x)$
for some basis functions $\alpha_{i}$ uniformly. Now fix $q'$ and
we have 
\[
\tilde{\mathbf{k}}(x,x')=\sum_{i=1}^{\infty}\alpha_{i}(x)\sum_{j=1}^{\infty}\beta_{j}(x')=\sum_{i,j=1}^{\infty}\alpha_{i}(x)\beta_{j}(x')
\]
Again, $\beta_{j}$ are basis functions. For finite numbers of the
basis $\alpha_{1},\cdots,\alpha_{L}$ and $\beta_{1},\cdots,\beta_{L}$,
it can be representable as a linear combination of orthogonal basis
$\psi$ using the Gram-Schmidt orthogonalization. So we have 
\[
\tilde{\mathbf{k}}(x,x')=\lim_{L\rightarrow\infty}\sum_{l=1}^{L}\lambda_{l}^{(1)}\psi_{l}(x)\psi_{l}(x').
\]
For continuous centered kernel $\tilde{\mathbf{k}}$, define linear
operator $\langle\cdot,\tilde{\mathbf{k}}\rangle:\mathcal{B}_{\mathcal{H}}\to\mathcal{B}_{\mathcal{H}}$
as 
\[
\langle f(\cdot),\tilde{\mathbf{k}}(\cdot,x')\rangle=\int\tilde{\mathbf{k}}(x,x')f(x)\; dF(x)=f(x').
\]
This is a compact self-adjoint operator%
\footnote{An operator $T$ is compact if the closure of $T(\mbox{close ball})$
is compact and it is self-adjoint if $\langle Tf,g\rangle=\langle f,Tg\rangle$. %
} which yields unique countable eigenvalues $\lambda_{i}$ and orthonormal
eigenfunctions $\psi_{i}$ such that%
\footnote{This eigenequation is called a Fredholm equation of the first kind.
Given covariance function, there is a numerical technique for estimating
$\lambda_{i}$ and $\psi_{i}$ numerically.%
} 
\begin{equation}
\begin{array}{cc}
\int\tilde{\mathbf{k}}(x,x')\psi_{l}(x)dF(x)=\lambda_{l}^{(1)}\psi_{l}(x'),\\
\int\psi_{i}(x)\psi_{j}(x)dF(x)=\delta_{ij}=\begin{cases}
1 & \mbox{if }i=j\\
0 & \mbox{otherwise}
\end{cases}
\end{array}\label{eq:Proof-ChiSquare-1}
\end{equation}
with $\lambda_{\infty}=0$. We will order them such that $\lambda_{0}>\lambda_{1}>\lambda_{2}>\cdots$.
Hence identifying $\lambda_{i}=\sigma_{i}^{2}$. Similarly, we can
obtain the results for $\tilde{\mathbf{k}}(\varepsilon_{i},\varepsilon_{j})=\sum_{l=1}^{\infty}\lambda_{l}^{(2)}\psi_{l}(\varepsilon_{i})\psi_{l}(\varepsilon_{j})$.

Step 2: To prove that 
\begin{align*}
n\widehat{S}_{\mathcal{H}}(\theta^{*}):=\frac{1}{n}\sum_{i\neq j}^{n}h(q_{i},q_{j})= & \frac{1}{n}\sum_{i\neq j}^{n}\left[\tilde{\mathbf{k}}(x_{i},x_{j})+\tilde{\mathbf{k}}(\varepsilon_{i},\varepsilon_{j})-\tilde{\mathbf{k}}(x_{i},\varepsilon_{j})-\tilde{\mathbf{k}}(\varepsilon_{i},x_{j})\right]\\
\rightsquigarrow & \sum_{l=1}^{\infty}\lambda_{l}\left[z_{l}^{2}-1\right].
\end{align*}

Firstly, consider the limit of $n^{-1}\sum_{i\neq j}^{n}\tilde{\mathbf{k}}(x_{i},x_{j})$.
With the result in \eqref{eq:Proof-ChiSquare-expansion}, we have
\begin{align*}
\frac{1}{n}\sum_{i\neq j}\tilde{\mathbf{k}}(x_{i},x_{j}) & =\frac{1}{n}\sum_{i\neq j}\sum_{l=1}^{\infty}\lambda_{l}^{(1)}\psi_{l}(x_{i})\psi_{l}(x_{j})\\
 & \overset{(a)}{=}\frac{1}{n}\sum_{l=1}^{\infty}\lambda_{l}^{(1)}\left(\left(\sum_{i=1}^{n}\psi_{l}(x_{i})\right)^{2}-\sum_{i=1}^{n}\psi_{l}^{2}(x_{i})\right)
\end{align*}
$\overset{(a)}{=}$ comes from $(\sum_{i}a_{i})^{2}=\sum_{i}a_{i}(a_{i}+\sum_{i\neq j}a_{j})=\sum_{i}a_{i}^{2}+\sum_{i\neq j}a_{i}a_{j}$.

For any finite $l$, taking expectation over \eqref{eq:Proof-ChiSquare-1}
gives 
\[
\lambda_{l}^{(1)}\mathbb{E}_{x}[\psi_{l}(x)]=\int\mathbb{E}_{x}\tilde{\mathbf{k}}(x',x)\psi_{l}(x')dF(x')=0
\]
by the result of \eqref{eq:Proof-of-ChiSq-0}. Hence $\mathbb{E}_{x}[\psi_{l}(x)]=0$
because $\lambda_{l}\neq0$ if $l\neq\infty$. Also, from \eqref{eq:Proof-ChiSquare-1},
we know that $\mbox{Cov}(\psi_{i}(x),\psi_{j}(x))=\delta_{ij}$. By
the Lindeberg-Levy CLT, 
\[
n^{-1/2}\sum_{i=1}^{n}\psi_{l}(x_{i})\rightsquigarrow\mathcal{N}(0,1).
\]
From \eqref{eq:Proof-ChiSquare-1}, it is also known that $\mathbb{E}_{x}[\psi_{l}^{2}(x)]=1$
for any $l$. Then strong LLN gives 
\[
n^{-1}\left(\sum_{i=1}^{n}\psi_{l}^{2}(x_{i})\right)\overset{p}{\rightarrow}1.
\]
Hence 
\[
\frac{1}{n}\sum_{l=1}^{\infty}\lambda_{l}^{(1)}\left(\left(\sum_{i=1}^{n}\psi_{l}(x_{i})\right)^{2}-\sum_{i=1}^{n}\psi_{l}^{2}(x_{i})\right)\rightsquigarrow\sum_{l=1}^{\infty}\lambda_{l}^{(1)}\left[z_{l}^{2}-1\right]
\]
where $z_{l}^{2}$ is a standard normal random variable.

Secondly, consider the limit of $n^{-1}\sum_{i\neq j}^{n}\tilde{\mathbf{k}}(\varepsilon_{i},\varepsilon_{j})$.
Following the previous procedure, we can derive 
\[
\frac{1}{n}\sum_{i\neq j}\tilde{\mathbf{k}}(\varepsilon_{i},\varepsilon_{j})=\frac{1}{n}\sum_{l=1}^{\infty}\lambda_{l}^{(2)}\left(\left(\sum_{i=1}^{n}\psi_{l}(\varepsilon_{i})\right)^{2}-\sum_{i=1}^{n}\psi_{l}^{2}(\varepsilon_{i})\right)\rightsquigarrow\sum_{l=1}^{\infty}\lambda_{l}^{(2)}\left[z_{l}^{2}-1\right].
\]

Finally, consider the limit of $n^{-1}\sum_{i\neq j}^{n}\left[\tilde{\mathbf{k}}(x_{i},\varepsilon_{j})+\tilde{\mathbf{k}}(\varepsilon_{i},x_{j})\right]$.
Because by independence condition, at $\theta_{0}$, $x$ and $\varepsilon$
are independent. Hence $\mbox{Cov}(\psi_{i}(x),\psi_{j}(\varepsilon))=0$
for all $i,j$. Then we can conclude 
\[
\frac{1}{n}\sum_{i\neq j}^{n}\left[\tilde{\mathbf{k}}(x_{i},\varepsilon_{j})+\tilde{\mathbf{k}}(\varepsilon_{i},x_{j})\right]\rightarrow0.
\]
Combine these three limit forms, the result follows. \end{proof}

\subsection{\label{sub:Proof-of-Theorem optimal-info} Proof of Theorem \ref{thm:optimal-info}}

\begin{proof} The proof uses some properties of Hellinger distance
function. With C14, a Hellinger distance can be approximately represented
as
\begin{align*}
\lim_{\|\theta-\theta_{0}\|\rightarrow0}\mathbf{d}_{\mathbf{H}}^{2}(H_{\theta},FG_{\theta})= & \lim_{\|\theta-\theta_{0}\|\rightarrow0}\frac{1}{2}\int\left(\sqrt{h_{\theta}(q)w_{\theta}(q)}-\sqrt{p_{\theta}(q)w_{\theta}(q)}\right)^{2}dq\\
= & \lim_{\|\theta-\theta_{0}\|\rightarrow0}\frac{1}{2}\int\left(\sqrt{h_{\theta}(q)}-\sqrt{p_{\theta}(q)}\right)^{2}w_{\theta}(q)dq\\
= & \lim_{\|\theta-\theta_{0}\|\rightarrow0}\frac{1}{2}\int\left((\theta-\theta_{0})^{T}\varDelta_{w}(q)\right)\left((\theta-\theta_{0})^{T}\varDelta_{w}(q)\right)^{T}\left(\frac{w_{\theta}(q)}{\mathbf{k}(q)}\right)dq+o_{p}(1)\\
\approx & \frac{1}{2}(\theta-\theta_{0})^{T}\int(\iota_{\theta_{0}}(q))d(H_{\theta}-P_{\theta})\left[\int(\iota_{\theta_{0}}(q))d(H_{\theta}-P_{\theta})\right]^{T}(\theta-\theta_{0})^{T}
\end{align*}
where $\iota_{\theta_{0}}(q)=\left.(\log D_{\theta}(q)w_{\theta}(q))'\right|_{\theta=\theta_{0}}$.

Step 1: By definition, $(\theta^{*},\mathbf{k})$ is the minimizer
of $\hat{W}_{\mathcal{H}}(\theta)$ such that 
\[
(\theta^{*},\mathbf{k})=\arg\min_{\theta}\max_{f\in\mathcal{H}}\left|\int f(q)dH_{\theta}-\int f(q)dP_{\theta}\right|,
\]
then $ $$(\theta^{*},w_{\theta^{*}})$ satisfies the following inequalities
\begin{align*}
\underset{:=2\sqrt{2}\mathbf{d}_{\mathbf{H}}^{2}(H_{\theta^{*}},FG_{\theta^{*}})}{\underbrace{\sqrt{4\int\left(\sqrt{h_{\theta^{*}}(q)}-\sqrt{p_{\theta^{*}}(q)}\right)^{2}w_{\theta^{*}}(q)dq}}} & \overset{(a)}{\geq}\sqrt{\int\left|\sqrt{h_{\theta^{*}}(q)}-\sqrt{p_{\theta^{*}}(q)}\right|^{2}w_{\theta^{*}}(q)dq}\\
\times & \sqrt{\int\left|\sqrt{h_{\theta^{*}}(q)}+\sqrt{p_{\theta^{*}}(q)}\right|^{2}w_{\theta^{*}}(q)dq}\\
 & \overset{(b)}{\geq}\int\left|\sqrt{h_{\theta^{*}}(q)}-\sqrt{p_{\theta^{*}}(q)}\right|\left|\sqrt{h_{\theta^{*}}(q)}+\sqrt{p_{\theta^{*}}(q)}\right|w_{\theta^{*}}(q)dq\\
 & \overset{(c)}{=}\int\left|h_{\theta^{*}}(q)-p_{\theta^{*}}(q)\right|w_{\theta^{*}}(q)dq\\
 & \overset{(d)}{\geq}\left|\int h_{\theta^{*}}(q)w_{\theta^{*}}(q)dq-\int p_{\theta^{*}}(q)w_{\theta^{*}}(q)dq\right|
\end{align*}
where $\overset{(a)}{\geq}$ comes from
\begin{align*}
4\int\left(\sqrt{h_{\theta^{*}}(q)}-\sqrt{p_{\theta^{*}}(q)}\right)^{2}w_{\theta^{*}}(q)dq & \geq\int\left|\sqrt{h_{\theta^{*}}(q)}-\sqrt{p_{\theta^{*}}(q)}\right|^{2}w_{\theta^{*}}(q)dq+\\
 & \int\left|\sqrt{h_{\theta^{*}}(q)}+\sqrt{p_{\theta^{*}}(q)}\right|^{2}w_{\theta^{*}}(q)dq\\
 & \geq\int\left|\sqrt{h_{\theta^{*}}(q)}-\sqrt{p_{\theta^{*}}(q)}\right|^{2}w_{\theta^{*}}(q)dq\times\\
 & \int\left|\sqrt{h_{\theta^{*}}(q)}+\sqrt{p_{\theta^{*}}(q)}\right|^{2}w_{\theta^{*}}(q)dq
\end{align*}
by using the properties $\left|\sqrt{h_{\theta^{*}}(q)}-\sqrt{p_{\theta^{*}}(q)}\right|<1$
and $0<\int\left|\sqrt{h_{\theta}(q)}-\sqrt{p_{\theta}(q)}\right|^{2}w_{\theta}(q)dq<1$;
$\overset{(b)}{\geq}$ follows from Cauchy\textendash{}Schwarz inequality
\[
\sqrt{\int\left|a(q)\right|^{2}w_{\theta^{*}}(q)dq\times\int\left|b(q)\right|^{2}w_{\theta^{*}}(q)dq}\geq\int\left|a(q)\right|\left|b(q)\right|^{2}w_{\theta^{*}}(q)dq
\]
where $a(q)=\sqrt{h_{\theta^{*}}(q)}-\sqrt{p_{\theta^{*}}(q)}$ and
$b(q)=\sqrt{h_{\theta^{*}}(q)}+\sqrt{p_{\theta^{*}}(q)}$; $\overset{(c)}{=}$
is from $|a(q)|\cdot|b(q)|=|a(q)b(q)|$; $\overset{(d)}{\geq}$ follows
from H�lder's inequality.

On the other hand, we have 
\begin{align*}
\max_{\mathcal{F}}\sqrt{S_{\mathcal{H}}(\theta^{*})}=\sqrt{\int\left|h_{\theta^{*}}(q)-p_{\theta^{*}}(q)\right|^{2}\mathbf{k}(q)dq} & \overset{(a)}{\geq}\sqrt{\int\left|h_{\theta^{*}}(q)-p_{\theta^{*}}(q)\right|^{2}w_{\theta^{*}}(q)dq}\\
 & \overset{(b)}{\geq}\int\left|h_{\theta^{*}}(q)-p_{\theta^{*}}(q)\right|w_{\theta^{*}}(q)dq\\
 & \overset{(c)}{=}\int\left|\sqrt{h_{\theta^{*}}(q)}-\sqrt{p_{\theta^{*}}(q)}\right|w_{\theta^{*}}(q)dq\times\\
 & \int\left|\sqrt{h_{\theta^{*}}(q)}+\sqrt{p_{\theta^{*}}(q)}\right|w_{\theta^{*}}(q)dq\\
 & \overset{(d)}{\geq}\underset{:=2\mathbf{d}_{\mathbf{H}}^{2}(H_{\theta^{*}},FG_{\theta^{*}})}{\underbrace{\int\left(\sqrt{h_{\theta^{*}}(q)}-\sqrt{p_{\theta^{*}}(q)}\right)^{2}w_{\theta^{*}}(q)dq}}
\end{align*}
where $\overset{(a)}{\geq}$ comes from the condition C14 
\[
\min_{\theta}\max_{w_{\theta}\in\mathcal{F}}\int\left|h_{\theta}(q)-p_{\theta}(q)\right|^{2}w_{\theta}(q)dq\geq\min_{\theta}\int\left|h_{\theta}(q)-p_{\theta}(q)\right|^{2}w_{\theta}(q)dq
\]
for any $w_{\theta}\in\mathcal{F}$; $\overset{(b)}{\geq}$ is from
H�lder's inequality; $\overset{(c)}{=}$ is from $|a(q)|\cdot|b(q)|=|a(q)b(q)|$;
$\overset{(d)}{\geq}$ is from $\left|\sqrt{h_{\theta^{*}}(q)}+\sqrt{p_{\theta^{*}}(q)}\right|\geq\left|\sqrt{h_{\theta^{*}}(q)}-\sqrt{p_{\theta^{*}}(q)}\right|$
for non-negative $h_{\theta^{*}}(q)$ and $p_{\theta^{*}}(q)$.

Step 2: Information in $\mathbf{d}_{\mathbf{H}}^{2}(H_{\theta^{*}},P_{\theta^{*}})$.

Consider an expansion of $2\mathbf{d}_{\mathbf{H}}^{2}(H_{\theta^{*}},P_{\theta^{*}})$
around $\theta_{0}$. 
\begin{align*}
2\mathbf{d}_{\mathbf{H}}^{2}(H_{\theta^{*}},P_{\theta^{*}})= & \underset{(I)}{\underbrace{\int\left(\sqrt{h_{\theta^{*}}(q)}-\sqrt{p_{\theta^{*}}(q)}-(\sqrt{h_{\theta_{0}}(q)}-\sqrt{p_{\theta_{0}}(q)})\right)^{2}w_{\theta^{*}}(q)dq}}\\
+ & \underset{(II)}{\underbrace{\int\left(\sqrt{h_{\theta_{0}}(q)}-\sqrt{p_{\theta_{0}}(q)}\right)^{2}w_{\theta^{*}}(q)dq}}\\
- & \underset{(III)}{\underbrace{2\int\left(\sqrt{h_{\theta^{*}}(q)}-\sqrt{p_{\theta^{*}}(q)}\right)\left(\sqrt{h_{\theta_{0}}(q)}-\sqrt{p_{\theta_{0}}(q)}\right)w_{\theta^{*}}(q)dq}.}
\end{align*}
Term $(I)$ has the following properties
\[
(I)\geq\left|\int\left(h_{\theta^{*}}(q)-p_{\theta^{*}}(q)\right)w_{\theta^{*}}(q)dq\right|-\left|\int\left(h_{\theta_{0}}(q)-p_{\theta_{0}}(q)\right)w_{\theta_{0}}(q)dq\right|
\]
and 
\[
(I)\leq\int\left|h_{\theta^{*}}(q)-p_{\theta^{*}}(q)\right|\mathbf{k}(q)dq+\int\left|h_{\theta_{0}}(q)-p_{\theta_{0}}(q)\right|\mathbf{k}(q)dq\leq\sqrt{S_{\mathcal{H}}(\theta^{*})}+\sqrt{S_{\mathcal{H}}(\theta_{0})},
\]
and term $(II)$ is $2\mathbf{d}_{\mathbf{H}}^{2}(H_{\theta_{0}},P_{\theta_{0}})$
follows the bounds
\[
\left|\int\left(h_{\theta_{0}}(q)-p_{\theta_{0}}(q)\right)w_{\theta_{0}}(q)dq\right|\leq(II)\leq\int\left|h_{\theta_{0}}(q)-p_{\theta_{0}}(q)\right|\mathbf{k}(q)dq\leq\sqrt{S_{\mathcal{H}}(\theta_{0})},
\]
and term $(III)$ by condition C14 has the following representation
\begin{align*}
(III) & =2(\theta^{*}-\theta_{0})^{T}\int\varDelta_{w}(q)\mathbf{k}^{-1/2}(q)(\sqrt{h_{\theta_{0}}(q)}-\sqrt{p_{\theta_{0}}(q)})w_{\theta^{*}}(q)dq+o(\|\theta^{*}-\theta_{0}\|)\\
 & =(\theta^{*}-\theta_{0})^{T}\int\varDelta_{w}(q)\left(\frac{w_{\theta^{*}}}{w_{\theta_{0}}\mathbf{k}}\right)d(H_{\theta_{0}}-P_{\theta_{0}})(q)+o(\|\theta^{*}-\theta_{0}\|)
\end{align*}
by using the fact that $ $ $w_{\theta_{0}}^{1/2}(q)dq=d\left(\sqrt{H_{\theta_{0}}}-\sqrt{P_{\theta_{0}}}\right)=(\sqrt{h_{\theta_{0}}(q)}-\sqrt{p_{\theta_{0}}(q)})^{-1}d(H_{\theta_{0}}-P_{\theta_{0}})/2$.

As $\theta^{*}\overset{p}{\rightarrow}\theta_{0}$, 
\[
0=\lim_{\theta^{*}\overset{p}{\rightarrow}\theta_{0}}W_{\mathcal{H}}(\theta^{*})\geq\left|\int\left(h_{\theta^{*}}(q)-p_{\theta^{*}}(q)\right)w_{\theta^{*}}(q)dq\right|
\]
and $S_{\mathcal{H}}(\theta^{*})\rightarrow0$ ensure that $(I)$
and $(II)$ are ignorable and the representation 
\[
(\theta^{*}-\theta_{0})^{T}\int\varDelta_{w}(q)\left(\frac{w_{\theta_{0}}}{\mathbf{k}}\right)^{\frac{1}{2}}d(H_{\theta_{0}}-P_{\theta_{0}})(q)
\]
in $(III)$ gives the information score function at $\theta_{0}$
\[
\int\varDelta_{w}(q)\left(\frac{w_{\theta_{0}}}{\mathbf{k}}\right)^{\frac{1}{2}}d(H_{\theta_{0}}-P_{\theta_{0}})(q).
\]

From the result of Delta method on minimum distance estimator, see
i.e. proof of Therem 3.3 (iii) in BW, it is known that 
\[
\theta^{*}=\theta_{0}-2\mathbf{V}^{-1}\int\varDelta_{w}(q)\left(\frac{w_{\theta_{0}}}{\mathbf{k}}\right)^{\frac{1}{2}}d(H_{\theta_{0}}-P_{\theta_{0}})(q)+o(1)
\]
where $\mathbf{V}$ is the derivative of $S_{\mathcal{H}}(\theta)$
assumed in condition C13. Then the covariance term is
\[
\left[\int\left(\varDelta_{w}(q)\left(\frac{w_{\theta_{0}}}{\mathbf{k}}\right)^{\frac{1}{2}}(q)\right)d(H_{\theta_{0}}-P_{\theta_{0}})(q')\right]\left[\int\left(\varDelta_{w}(q)\left(\frac{w_{\theta_{0}}}{\mathbf{k}}\right)^{\frac{1}{2}}(q')\right)d(H_{\theta_{0}}-P_{\theta_{0}})(q')\right]^{T}.
\]
Because $w_{\theta}\in\mathcal{F}$, $\mathbf{k}:=\arg\max_{f\in\mathcal{F}}\int fd(H_{\theta}-P_{\theta})(q)$
and $\|\mathbf{k}(q)\|\leq1$, we can see 
\[
\int\left(\varDelta_{w}(q)\left(\frac{w_{\theta_{0}}}{\mathbf{k}}\right)^{\frac{1}{2}}\right)d(H_{\theta_{0}}-P_{\theta_{0}})(q)\leq\int\left(\varDelta_{w}(q)\left(\frac{w_{\theta_{0}}}{f}\right)^{\frac{1}{2}}\right)d(H_{\theta_{0}}-P_{\theta_{0}})(q)
\]
for any $f\in\mathcal{F}$. Thus $(\theta^{*},\mathbf{k})$ attains
the optimal information.\end{proof}

\section{Proof of Other Statements}

\subsection{\label{sub:Proof-of-Corollary Metrization}Proof of Corollary \ref{cor:Metrization}}

\begin{proof} By the definition of metrization, we need to show that
for $\{P_{1},P_{2},\dots\}\subset\mathscr{P}$, $P_{n}\rightsquigarrow P$
if and only if $\left\Vert P_{n}\mathbf{k}-P\mathbf{k}\right\Vert _{\mathcal{H}}\rightarrow0$
as $n\rightarrow\infty$.

($\Rightarrow$) Because $\mathbf{k}$ is bounded and continuous,
so it is obvious that $P_{n}\rightsquigarrow P$ implies $\left\Vert P_{n}\mathbf{k}-P\mathbf{k}\right\Vert _{\mathcal{H}}\rightarrow0$
as $n\rightarrow\infty$.

($\Leftarrow$) $\mathcal{H}$ induces a dense set in $\mathcal{B}_{\mathcal{H}}\subset C(\mathcal{Q})$.
For the space of continuous bounded functions, any $\mathbf{k}\in\mathcal{F}_{\mathcal{H}}$
and any every $\epsilon>0$, there exists a $g\in\mathcal{F}_{\mathcal{H}}$
such that $\sup|\mathbf{k}-g|\leq\epsilon$. Therefore 
\begin{align*}
\left|\int\mathbf{k}dP_{n}-\int\mathbf{k}dP\right|= & \left|\int\mathbf{k}dP_{n}-\int gdP_{n}+\int\mathbf{k}dP-\int gdP+\int gdP_{n}-\int gdP\right|\\
\leq & \left|\int\mathbf{k}dP_{n}-\int gdP_{n}\right|+\left|\int\mathbf{k}dP-\int gdP\right|+\left|\int gdP_{n}-\int gdP\right|\\
\leq & 2\epsilon+\left|\int gdP_{n}-\int gdP\right|\leq2\epsilon+\|g\|_{\mathcal{H}}\|P_{n}\mathbf{k}-P\mathbf{k}\|_{\mathcal{H}}.
\end{align*}
If $\|P_{n}\mathbf{k}-P\mathbf{k}\|_{\mathcal{H}}\rightarrow0$ as
$n\rightarrow\infty$ and $\epsilon$ is arbitrary, $|P_{n}\mathbf{k}-P\mathbf{k}|\rightarrow0$
for any $\mathbf{k}\in\mathcal{B}_{\mathcal{H}}$ or say $P_{n}\rightsquigarrow P$.
\end{proof}

\subsection{\label{sub:Proof-of-Lemma Iden}Proof of Lemma \ref{lem:Identification}}

\begin{proof} ( $\Leftarrow$ ) Suppose a finite non-zero signed
Borel measure $\mathcal{P}$ satisfies Condition (I) and (II) in Lemma
\ref{lem:Identification}. By Jordan decomposition theorem \citep[Theorem 5.6.1,][]{Dudley2002},
there exist unique positive measure $\mathcal{P}^{+}$ and $\mathcal{P}^{-}$
such that $\mathcal{P}=\mathcal{P}^{+}-\mathcal{P}^{-}$ and $\mathcal{P}^{+}\bot\mathcal{P}^{-}$.
By Condition (II) in Lemma \ref{lem:Identification}, we have $\mathcal{P}^{+}(\mathcal{Q})=\mathcal{P}^{-}(\mathcal{Q})=:\alpha$.
Let $H_{\theta}=\alpha^{-1}\mathcal{P}^{+}$ and $P_{\theta}=\alpha^{-1}\mathcal{P}^{-}$.
Thus, $H_{\theta}\neq P_{\theta}$. (Or one can set $P_{\theta}=\alpha^{-1}\mathcal{P}^{+}$
and $H_{\theta}=\alpha^{-1}\mathcal{P}^{-}$.) Then, by 
\begin{align*}
W_{\mathcal{H}}^{2}(\theta)= & \left\Vert H_{\theta}\mathbf{k}-P_{\theta}\mathbf{k}\right\Vert _{\mathcal{H}}^{2}=\left\langle H_{\theta}\mathbf{k}-P_{\theta}\mathbf{k},\, H_{\theta}\mathbf{k}-P_{\theta}(q)\right\rangle \\
= & \left\langle \int_{\mathcal{Q}}\mathbf{k}(\cdot,q)dH_{\theta}(q),\int_{\mathcal{Q}}\mathbf{k}(\cdot,q')dH_{\theta}(q')\right\rangle +\left\langle \int_{\mathcal{Q}}\mathbf{k}(\cdot,q)dP_{\theta}(q),\int_{\mathcal{Q}}\mathbf{k}(\cdot,q')dP_{\theta}(q')\right\rangle \\
 & -2\left\langle \int_{\mathcal{Q}}\mathbf{k}(\cdot,q)dH_{\theta}(q),\int_{\mathcal{Q}}\mathbf{k}(\cdot,q')dP_{\theta}(q')\right\rangle \\
= & \int_{\mathcal{Q}\times\mathcal{Q}}\mathbf{k}(q,q')dH_{\theta}(q)dH_{\theta}(q')+\int_{\mathcal{Q}\times\mathcal{Q}}\mathbf{k}(q,q')dP_{\theta}(q)dP_{\theta}(q')\\
 & -2\int_{\mathcal{Q}\times\mathcal{Q}}\mathbf{k}(q,q')dH_{\theta}(q)dP_{\theta}(q')\\
= & \int_{\mathcal{Q}\times\mathcal{Q}}\mathbf{k}(q,q')d(H_{\theta}-P_{\theta})(q)d(H_{\theta}-P_{\theta})(q')=0.
\end{align*}
The last equality is set by Condition (I) in Lemma \ref{lem:Identification}.

( $\Rightarrow$ ) Suppose there exists $H_{\theta}\neq P_{\theta}$
such that $W_{\mathcal{H}}(\theta)=0$. Let $\mathcal{P}=P_{\theta}-H_{\theta}$.
Then $\mathcal{P}$ is a finite non-zero signed Borel measure that
satisfies $\mathcal{P}(\mathcal{Q})=0$. Note that 
\[
W_{\mathcal{H}}^{2}(\theta)=\int_{\mathcal{Q}\times\mathcal{Q}}\mathbf{k}(q,q')d(H_{\theta}-P_{\theta})(q)d(H_{\theta}-P_{\theta})(q')=0.
\]
Thus Condition (I) in Lemma \ref{lem:Identification} satisfies. \end{proof}

\subsection{\label{sub:Proof-of-Corollary Dual}Proof of Corollary \ref{cor:ImpDual}}

\begin{proof} By the property of Lagrange multipliers $u_{i}$, $v_{j}$,
the problem \eqref{eq:LP-dual} can be rewritten as a minimax problem:
\begin{align*}
V=\min_{\gamma_{i,j}(\theta)\in\Gamma(H_{n\theta},P_{n\theta})}\max_{u_{i},v_{j}} & \sum_{i,j}\gamma_{i,j}(\theta)C_{i,j}+\sum_{i}u_{i}-\sum_{i,j}u_{i}\gamma_{i,j}(\theta)+\sum_{j}v_{j}-\sum_{i,j}v_{j}\gamma_{i,j}(\theta)\\
=\min_{\gamma_{i,j}(\theta)\in\Gamma(H_{n\theta},P_{n\theta})}\max_{u_{i},v_{j}} & \sum_{i}u_{i}+\sum_{j}v_{j}+\sum_{i,j}\gamma_{i,j}(\theta)\left(C_{i,j}-u_{i}-v_{j}\right).
\end{align*}
By the saddlepoint argument, we have 
\[
V=\max_{u_{i},v_{j}}\min_{\gamma_{i,j}(\theta)\in\Gamma(H_{n\theta},P_{n\theta})}\sum_{i}u_{i}+\sum_{j}v_{j}+\sum_{i,j}\gamma_{i,j}(\theta)\left(C_{i,j}-u_{i}-v_{j}\right).
\]
Consider $\gamma_{i,j}(\theta)$ as the Lagrange multipliers associated
to $u_{i}+v_{j}\geq C_{i,j}$. Hence, we have the dual formulation
\begin{align*}
\max_{u_{i},v_{j}} & \sum_{i}u_{i}+\sum_{j}v_{j}\\
\mbox{s.t.}\quad & u_{i}+v_{j}\geq C_{i,j}.
\end{align*}
The result follows. \end{proof}

\section{Auxiliary Results}

\subsection{Large Deviation Bounds}

\begin{thm} \label{thm:McDiarmid's-inequality}(McDiarmid's inequality)
Let $c(\cdot):\mathcal{Q}^{m}\rightarrow\mathbb{R}$ where $\mathcal{Q}^{m}=\mathcal{Q}\times\cdots\times\mathcal{Q}$.
For all $i\in\{1,\dots,m\}$, there exists $C_{i}<\infty$ such that
\[
\sup_{(q_{1},\dots,q_{m})\in\mathcal{Q}^{m},\tilde{q}\in\mathcal{Q}}\left|f(q_{1},\dots,q_{m})-f(q_{1},\dots,q_{i-1},\tilde{q},q_{i+1},\dots,q_{m})\right|\leq C_{i}.
\]
Then for all probability measures $p$ and every $\epsilon>0$, 
\[
\Pr\left\{ (f(q_{1},\dots,q_{m})-\mathbb{E}[f(q_{1},\dots,q_{m})])>\epsilon\right\} <\exp\left(-\frac{2\epsilon^{2}}{\sum_{i=1}^{m}C_{i}^{2}}\right)
\]
where the expectation is defined for $(q_{1},\dots,q_{m})$. \end{thm}

\begin{thm} \label{thm:Rademacher-symmetrization}(Rademacher symmetrization)
Let $\mathcal{B}_{\mathcal{H}}$ is the unit ball in the Hilbert space
$\mathcal{H}$ on a compact domain $\mathcal{Q}$. The associated
kernel is bounded such that $0\leq\mathbf{k}(q,q')\leq C_{\mathbf{k}}$.
Let $Q$ be an i.i.d. sample of size $m$ drawn according to a probability
measure $H_{\theta}(q)$ on $\mathcal{Q}$, and let $\sigma_{i}$
be i.i.d random variable taking values in $\{-1,1\}$ with equal probability.
The Rademacher average 
\[
R_{m}(\mathcal{B}_{\mathcal{H}},Q,H_{\theta}):=\mathbb{E}_{H_{\theta},\sigma}\sup\left|\frac{1}{m}\sum_{i=1}^{m}\sigma_{i}f(q_{i})\right|\leq\sqrt{(C_{\mathbf{k}}/m)}
\]
where the expectation is taken w.r.t. $Q\sim H_{\theta}$ and $\sigma$
jointly.\end{thm} 
\end{document}